\documentclass[reqno,12pt]{amsart}
\usepackage[colorlinks=true, linkcolor=blue, citecolor=blue]{hyperref}

\usepackage{amssymb}
\usepackage{amsmath, graphicx, rotating}
\usepackage{color}
\usepackage{soul}
\usepackage[dvipsnames]{xcolor}

\usepackage{ifthen}
\usepackage{xkeyval}
\usepackage{todonotes}
\usepackage[T1]{fontenc}
\usepackage{lmodern}
\usepackage[english]{babel}

\usepackage{ upgreek }
\usepackage{stmaryrd}
\SetSymbolFont{stmry}{bold}{U}{stmry}{m}{n}
\usepackage{amsthm}
\usepackage{float}

\usepackage{ bbm }
\usepackage{ stmaryrd }
\usepackage{ mathrsfs }
\usepackage{ frcursive }
\usepackage{ comment }

\usepackage{pgf, tikz}
\usetikzlibrary{shapes}
\usepackage{varioref}
\usepackage{enumitem}

\definecolor{rouge}{rgb}{0.7,0.00,0.00}
\definecolor{vert}{rgb}{0.00,0.5,0.00}
\definecolor{bleu}{rgb}{0.00,0.00,0.8}
\usepackage[margin=1in]{geometry}
\newtheorem{theorem}{Theorem}[section]
\newtheorem*{theorem*}{Theorem}
\newtheorem{lemma}[theorem]{Lemma}

\newtheorem{corollary}[theorem]{Corollary}
\newtheorem{proposition}[theorem]{Proposition}

\labelformat{hypothesis}{\textbf{M\kern-0.1mm#1}}

\newtheorem{condition}{Condition}

\newtheorem{conditionA}{A\kern-0.1mm}
\labelformat{conditionA}{\textbf{A\kern-0.1mm#1}}

\theoremstyle{definition}

\newtheorem{remark}[theorem]{Remark}

\def \eref#1{\hbox{(\ref{#1})}}

\numberwithin{equation}{section}

\def\geq{\geqslant}
\def\leq{\leqslant}

\def\RR{\mathbb{R}}

\def\EE{\mathbb{E}}
\def\NN{\mathbb{N}}

\def\vare{{\varepsilon}}
\def \eref#1{\hbox{(\ref{#1})}}

\def\EE{\mathbb{ E}}

\begin{document}

\title[Averaging principle for stochastic Burgers equation]
{ Averaging principle for one dimensional stochastic Burgers equation }

\author{Zhao Dong$^{1,2}$} \email{dzhao@amt.ac.cn}
 \footnotetext[1]{RCSDS, Academy of Mathematics and Systems Science, Chinese Academy of Sciences, Beijing 100190, China}
 \footnotetext[2]{School of Mathematical Sciences, University of Chinese Academy of Sciences, Beijing 100049, China}

\author{Xiaobin Sun$^{3}$} \email{xbsun@jsnu.edu.cn}
 \footnotetext[3]{School of Mathematics and Statistics, Jiangsu Normal University, Xuzhou, 221116, China}

\author{Hui Xiao$^{1,2,4}$}  \email{hui.xiao@univ-ubs.fr}
 \footnotetext[4]{Universit\'{e} de Bretagne-Sud, LMBA UMR CNRS 6205, Vannes, France}
 
\author{Jianliang Zhai$^{5}$}  \email{zhaijl@ustc.edu.cn}
 \footnotetext[5]{School of Mathematical Science, University of Science and Technology of China, Hefei, 230026, China}

%\author{Zhao Dong}
%\curraddr[Dong, Z.]{ University of Chinese Academy of Sciences, Beijing, 100190, China;
%  Academy of Mathematics and Systems Science,
%  Chinese Academy of Sciences (CAS), Beijing, 100190, China}
%\email{dzhao@amt.ac.cn}

%\author{Xiaobin Sun}
%\curraddr[Sun, X.]{ School of Mathematics and Statistics, Jiangsu Normal University, Xuzhou, 221116, China}
%\email{xbsun@jsnu.edu.cn}

%\author{Hui Xiao}
%\curraddr[Xiao, H.]
%{ University of Chinese Academy of Sciences, Beijing, 100190, China;
% Universit\'{e} de Bretagne-Sud, LMBA UMR CNRS 6205, Vannes, France}
%\email{hui.xiao@univ-ubs.fr}

%\author{Jianliang Zhai}
%\curraddr[Zhai, J.]{ School of Mathematical Science, University of Science and Technology of China, Hefei, 230026, China.}
%\email{zhaijl@ustc.edu.cn}

\begin{abstract}
In this paper, we consider the averaging principle
for one dimensional stochastic Burgers equation with slow and fast time-scales.
Under some suitable conditions,
we show that the slow component strongly converges to the solution
of the corresponding averaged equation.
Meanwhile, when there is no noise in the slow component equation,
we also prove that the slow component weakly converges to the solution
of the corresponding averaged equation with the order of convergence $1-r$, for any $0<r<1$.
\end{abstract}

%\date{\today}

\maketitle

\section{Introduction}

Many multiscale problems arise from
material sciences, chemistry, fluids dynamics, biology, ecology, climate dynamics and other application areas,
see, e.g., \cite{BR,WE,HKW,KEE,EEJ,MCCTB, WTRY16} and references therein.
E and Engquist \cite{WE} pointed out `` Problems in these areas are often
multiphysics in nature; namely, the processes at different
scales are governed by physical laws of different
character: for example, quantum mechanics
at one scale and classical mechanics at another.''
For instance,
dynamics of chemical reaction networks often take
place on notably different times scales, from the order of
nanoseconds($\rm 10^{-9}$ s) to the order of several days, the use of two-time
or multi-time scales is  common.
%This is especially true for gene
%regulatory networks.
Another example with multiple
time scales is that of protein folding. While the
time scale for the vibration of the covalent bonds
is on the order of femtoseconds ($\rm 10^{-15}$ s), folding
time for the proteins may very well be on the order
of seconds.

Many two-time scale/slow-fast systems can be formally written as
\begin{equation} \label{ExamAverag02}
\left\{\begin{array}{l}
\displaystyle
d X^{\varepsilon}_t
=  b_1(  X^{\varepsilon}_t,  Y^{\varepsilon}_t) dt
 +  \sigma_1(X^{\varepsilon}_t,  Y^{\varepsilon}_t)d W_t^1 ,\quad  X^{\varepsilon}_0=x \\
d  Y^{\varepsilon}_t
= \frac{1}{\varepsilon} b_2(  X^{\varepsilon}_t,  Y^{\varepsilon}_t) dt
 +  \frac{1}{\sqrt{\varepsilon} } \sigma_2(X^{\varepsilon}_t,  Y^{\varepsilon}_t)d W_t^2 ,\quad  Y^{\varepsilon}_0=y,
 \end{array}\right.
\end{equation}
where $W_t^1$, $W_t^2$ are independent Wiener processes, and the small parameter $\varepsilon$ quantifies the ratio of the
$X^{\varepsilon}$ and $Y^{\varepsilon}$ time scales.
%\todo{and the small parameter $\varepsilon$ quantifies the ratio of the time scales between $X^{\varepsilon}$ and $Y^{\varepsilon}$?}
%The components of $X^{\varepsilon}$ are called slow variables, while those of $Y^{\varepsilon}$ are called fast variable.
For many practical problems, 
it is of interest to study  the behavior of the system (\ref{ExamAverag02}) for $\varepsilon<<1$, 
and how dynamics
of this system depends on $\varepsilon$ as $\varepsilon\rightarrow 0$.
However, since $\varepsilon<<1$, it is often very difficult to directly calculate $X^{\varepsilon}$, and systems of this type are problematic for computer simulations. The averaging principle can be applied to solve these problems
of this type. Roughly speaking,
if the dynamics for $Y^\varepsilon$ with $X^\varepsilon=x$
%\todo{$X^\varepsilon= x$?}
fixed has an invariant probability measure $\mu^{x}(dy)$ and
the following integrals exist
$$
\bar{b}_1(x):=\int b_1(x,y)\mu^{x}(dy),\ \ \ \ \bar{\sigma}_1(x):=\int \sigma_1(x,y)\mu^{x}(dy),
$$
then under appropriate assumptions on all coefficients in the system (\ref{ExamAverag02}), the effective dynamics for $X^\varepsilon$ in the limit
of $\varepsilon\rightarrow 0$ is a stochastic differential equation:
%todo{ the component $X^\varepsilon$, as $\varepsilon \to 0$, converges to the stochastic differential equation: }
$$
d \bar{X}_t
=  \bar{b}_1(  \bar{X}_t) dt
 +  \bar{\sigma}_1(\bar{X}_t)d W_t^1 ,\quad  \bar{X}_0=x.
$$
The theory of averaging principle has a long history and rich results.
%, such as,
%celestial mechanics, wireless communication, signal processing,
%oscillation theory and radiophysics.
Bogoliubov and Mitropolsky \cite{BM} first studied the averaging principle
for the deterministic systems.
Later on,
%The first result for deterministic systems was studied by Bogoliubov and Mitropolsky \cite{BM},
%then by Volosov \cite{V} for ordinary differential equations.
%In 1968,
the theory of averaging principle for stochastic differential equations
 was firstly established by Khasminskii \cite{K1}.
Since then, averaging principle for stochastic reaction-diffusion systems
has become an active research area which attracted much attention.
%a number of mathematicians.
For example,
%Freidlin and Wentzell \cite{FW} established the theory of
%averaging principle from a deeper understanding of this phenomenon.
Cerrai and Freidlin \cite{CF} proved the averaging principle
for a general class of stochastic reaction-diffusion systems,
which extended the classical Khasminskii-type averaging principle
for finite dimensional systems to infinite dimensional systems.
Recently, based on the averaging principle, the fast flow asymptotics
for a stochastic reaction-diffusion-advection
equation are obtained by Cerrai and Freidlin \cite{CF18}.
We refer to \cite{B1,C1, WLV, FL,FLL,FLWL,IP,KY,Li,V1,WR12,WRD12}
and references therein for more interesting results on this topic.

To the best of our knowledge, there are rarely studies to deal with highly nonlinear term on this topic. In this paper, we are interested in studying the averaging principle for one dimensional stochastic Burgers, i.e., considering the following stochastic slow-fast system on the interval $[0,1]$:
\begin{equation}\left\{\begin{array}{l}\label{Equation}
\displaystyle
\frac{\partial X^{\varepsilon}_t(\xi)}{\partial t}=\big[\Delta X^{\varepsilon}_t(\xi)+\frac{1}{2}\frac{\partial}{\partial \xi}(X^{\varepsilon}_t(\xi))^2+f(X^{\varepsilon}_t(\xi), Y^{\varepsilon}_t(\xi))\big]+\frac{\partial W^{Q_{1}}}{\partial t}(t,\xi),\quad X^{\varepsilon}_0=x \\
\frac{\partial Y^{\varepsilon}_t(\xi)}{\partial t}
=\frac{1}{\varepsilon}\big[\Delta Y^{\varepsilon}_t(\xi)+g(X^{\varepsilon}_t(\xi), Y^{\varepsilon}_t(\xi))\big]+\frac{1}{\sqrt{\varepsilon}}\frac{\partial W^{Q_2}}{\partial t}(t,\xi),\quad Y^{\varepsilon}_0=y\\
X^{\varepsilon}_t(0)=X^{\varepsilon}_t(1)=Y^{\varepsilon}_t(0)
=Y^{\varepsilon}_t(1)=0, \end{array}\right.
\end{equation}
where $\varepsilon >0$ is a small parameter
describing the ratio of time scales between the slow component $X^{\varepsilon}_t$
and fast component $Y^{\varepsilon}_t$.
The coefficients $f$ and $g$ satisfy some suitable conditions.
$\{W^{Q_1}_t\}_{t\geq 0}$ and  $\{W^{Q_2}_t\}_{t\geq 0}$ are $L^2(0,1)$-valued mutually independent
$Q_1$ and $Q_2$-Wiener processes.
%$\{W^{Q_1}_t\}_{t\geq 0}$
%and $\{W^{Q_2}_t\}_{t\geq 0}$ are $L^2(0,1)$-valued mutually independent
%$Q_1$ and $Q_2$-Wiener processes.
%on complete probability space $(\Omega,\mathscr{F},\mathbb{P})$ with a filtration
%$\{\mathscr{F}_{t},t\geq0\}$ satisfying the usual conditions.

For any $t \in [0,T]$,
as $\varepsilon \to 0$, % goes to $0$,
the slow component $X^{\varepsilon}_t$ in \eqref{Equation} converges to $\bar{X}_t$,
which is the solution of the averaged equation:
\begin{equation}\left\{\begin{array}{l}
\displaystyle d\bar{X}_{t}=\Delta\bar{X}_{t}dt+\frac{1}{2}\frac{\partial}{\partial \xi}(\bar{X}_{t})^2 dt+\bar{f}(\bar{X}_{t})dt+dW^{Q_{1}}(t),\\
\bar{X}_{0}=x.\end{array}\right. \label{1.3}
\end{equation}
with the average
$$\bar{f}(x)=\int_{L^2(0,1)}f(x,y)\mu^{x}(dy), $$
where $\mu^{x}$ denotes the unique invariant measure for the fast component equation
with frozen slow component variable $x$ (see the equation \eref{FEQ} for details).
%when we fix slow variable $x\in L^2(0,1)$ ( see equation \eref{FEQ} below ).

%\vskip 0.3cm
We aim to study the rate of convergences of the process $X^{\varepsilon}$ to $\bar{X}$,
both in the strong convergence sense and in the weak convergence sense.
%Denote by $\|\cdot\|$ the norm of $L^2(0,1)$.
%in the strong and weak sense respectively.
Under some appropriate conditions, the result of strong convergence is stated as follows:
%\begin{theorem}(\textbf{Strong convergence})
%If assumptions $\textbf{(H1)}-\textbf{(H3)}$ hold, then for any $x\in H^{\alpha}, y\in H$, $p, T>0$ we have
%\begin{align*}
%\mathbb{E}\sup_{0\leq t\leq T}|X_{t}^{\vare}-\bar{X}_{t}|^{2p}\leq C_{p,T}\Big(\frac{1}{-\log\vare}\Big)^{\frac{1}{4p}}\rightarrow0~~(\vare\rightarrow0).
%\end{align*}
%\end{theorem}

%\begin{theorem} (\textbf{Weak convergence 1})
%Assume \textbf{(H1)}, \textbf{(H2)}, \textbf{(H4)} and \textbf{(H5)} hold, then for any $\theta\in(0,1]$, $r\in(0,1)$, $x\in H^{\theta}$, $y\in H$, $\phi\in C_{b}^{2}$, $t\in(0,T]$, %$\delta\in(0,\frac{1}{2})$, and for any small enough $\vare\in(0,1)$, we have
%\begin{align}
%\big|\mathbb{E}[\phi(X^{\vare}(t))]-\mathbb{E}[\phi(\bar{X}(t))]\big|\leq C(1+t^{-\theta+\frac{\theta^2}{1+\delta}})\vare^{1-r}, \nonumber
%\end{align}
%where $C$ is a positive constant which only depends on $T$, $|x|_{\theta}, |y|, \delta$ and $\phi$.
%\end{theorem}

\begin{itemize}
\item{
For any $x\in H^{\alpha}(0,1)$ with $\alpha\in(1,\frac{3}{2}]$
and $y\in L^2(0,1)$, $p>0$, $T>0$,
there exists a positive constant $C$ which is independent of $\varepsilon$ such that
\begin{align} \label{main result}
\mathbb{E} \left( \sup_{0\leq t\leq T} \| X^{\varepsilon}_t-\bar{X}_t \|^{2p} \right)
\leq C\Big(\frac{1}{-\log\vare}\Big)^{\frac{1}{4p}}\to 0, \quad \mbox{as} \ \varepsilon \to 0.
\end{align}
Here we denote by $\|\cdot\|$ the norm of $L^2(0,1)$.
%where $C$ is a positive constant which only depends on $T$, $p$, $|x|_{\alpha}$ and $|y|$.
}
\end{itemize}

If $Q_{1}=0$ in the system \eref{Equation}, then
under some conditions, the result of weak convergence is stated as follows:
\begin{itemize}
\item{
For any $x\in H^{\theta}(0,1)$ with $\theta\in(0, 1]$,
$y\in L^2(0,1)$, $\phi\in C_{b}^{2}(L^2(0,1))$, $r\in(0,1)$, $\delta\in (0,\frac{1}{2})$,
$t \in [0,T]$, there exists a positive constant $C$ which is independent of $\varepsilon$ such that
\begin{align}
\big|\mathbb{E} \phi(X^{\vare}_{t}) -\mathbb{E} \phi(\bar{X}_{t}) \big|\leq C(1+t^{-\theta+\frac{\theta^2}{1+\delta}})\vare^{1-r}.   \label{weak convergence}
\end{align}
%where $C$ is a positive constant which only depends on $T$, $|x|_{\theta}, |y|, \delta$ and $\phi$.
}

\end{itemize}

%As a consequence, we can see that the weak order is close to $1$. However, we can not obtain the strong order. This result is a little different from the

%\vskip 0.3cm

%Another part of our paper is to prove a weak convergence when slow equation without noise, i.e.,
%\begin{align}
%\big|\mathbb{E}[\phi(X^{\vare}(t))]-\mathbb{E}[\phi(\bar{X}(t))]\big|\leq C\vare^{1-r}. \label{weak convergence}
%\end{align}
%where $x\in H^{\theta}$ with $\theta\in(\frac{1}{2},1]$, $y\in H$, $\phi\in C_{b}^{2}$, $r\in(0,1)$ and $\bar{X}_t$ is the corresponding averaged equation.

\vskip -0.6cm
Comparing with the strong convergence, while the requirement on the regularity of initial value
$x$ in weak convergence is weaker, the rate of the convergence is pleasant in this case. The idea
of the proof follows a procedure inspired by \cite{B1}, in which the authors considered a relative simple
framework (without the nonlinear term and with $f$ being bounded ). In our case, to deal with the
nonlinear term and unbounded $f$ is a nontrivial task.

%Compared with the strong convergence, the regularity of initial value $x$ in weak convergence is weaker, but the rate of the convergence is pleasant in this case. The idea of the proof follows the procedure inspired by \cite{B1}, in which the authors consider a relative simple framework (without the nonlinear term and $f$ is bounded ). In our case, it is quite non-trivial
%to deal with the nonlinear term and unbounded $f$.

\vskip 0.1cm
%\vskip 0.3cm
The proof of \eqref{main result} is based on the Khasminskii argument introduced in \cite{K1}, but it is clearly more involved than in
\cite{K1}, as it concerns the
nonlinear term in the Burgers' equation and unbounded $f$.
To be precise, we split the interval
$[0,T]$ into some subintervals of size $\delta>0$,
then on each interval $[k\delta, (k+1)\delta))$, $k\geq 0$,
we construct an auxiliary process $(\hat{X}_t^\vare, \hat{Y}_t^\vare)$, $t\in [k\delta, (k+1)\delta))$,
associated with the system \eqref{Equation}.
Based on the exponential ergodicity of the fast component equation with frozen slow component $x$
in the system \eqref{Equation}, by controlling the error between
$\hat{X}_t^\vare $ and $X_t^\vare$, allows us to deduce  \eqref{main result}.
The biggest challenge in studying the strong convergence \eqref{main result} is to deal with the
nonlinear term.
%The most challenge in studying the strong convergence is how to deal with the nonlinear term in Burger's equation.
To overcome this difficulty, 
we first  give some estimates of the slow component $X^{\varepsilon}_t$ 
and fast component $Y^{\varepsilon}_t$ in $L^2(0,1)$. 
Secondly, by using the smoothness of semigroup $e^{t\Delta}$ and the interpolation inequality, 
we can further obtain 
$\sup_{\vare\in(0,1)}\mathbb{E}\sup_{t\in[0,T]}|X_{t}^{\varepsilon}|_{\alpha}^{p}\leq C_{p,T}$, 
which is a key step for proving \eref{main result}, 
where $|\cdot|_{\alpha}$ is the Sobolev norm. Finally, we obtain the result by
applying the skill of stopping time and following the procedure inspired by \cite{FL}.
%, we can obtain the main result.

To obtain the weak convergence \eqref{weak convergence},
we use the asymptotic expansion with respect to $\varepsilon$
of the solution to the Kolmogorov equation
corresponding to the system \eqref{Equation}.
However, some problems appear since the operator $\Delta$ is unbounded in the Kolmogorov equation.
To overcome this difficulty, following the approach used in \cite{B1}, we first
use the Galerkin approximation to reduce the infinite dimensional problem to a finite dimensional one,
then the remaining part is to establish the rate of convergence with some bounds
which is independent of the dimension.
Note that instead of using
the asymptotic expansion of the solution to the Kolmogorov equation,
an alternative martingale approach was applied to prove the weak convergence
for stochastic reaction-diffusion equations with unbounded multiplicative noise \cite{C1}.

%due to the unboundedness of the operator $\Delta$,
%we first apply
%we first use Galerkin approximation to
%the asymptotic expansion of

%\vskip 0.3cm
Finally, we refer that, in recent years, there are many interesting results for stochastic Burger's equation \cite{DD,Daz1,DaZ,DDT,DX,EWKM,HSX,LZ06, LZ09, TZ96, TZ98,TW}.

%\vskip 0.3cm
The rest of the paper is organized as follows.
In Section \ref{Sec Main Result}, %we introduce the necessary notations,
under some suitable assumptions,
we formulate our main results.
Section \ref{Sec Proof of Thm1} and Section \ref{Sec Proof of Thm2 3}
are devoted to proving the strong convergence and weak convergence, respectively.
In the Appendix \ref{Sec appendix},  we recall some useful inequalities.
% that we use throughout the paper.
%Finally in the Appendix, we prove some necessary estimates for proving the weak convergence.

%\vskip 0.3cm
Throughout the paper, $C$, $C_p$ and $C_{p,T}$ will denote positive constants which may change from line to line, where $C_p$ depends on $p$, $C_{p,T}$ depends on $p, T$.

%\section{Preliminaries}\label{sec.prelim}

\section{Notations and main results} \label{Sec Main Result}

%Denote by $|\cdot|_{L^p}$ the usual norm of the space $L^p(0,1),p\geq1$, and by $|\cdot|_{\infty}$ the usual
%supremum norm of the space $L^\infty(0,1)$. We consider the separable Hilbert space $H=L^2(0,1)$ (the inner product
%denoted $\langle\cdot,\cdot\rangle$). As usual, for $k\in\mathbb{N}, p\geq 1$, $W^{k, p}(0,1)$ is the Sobolev space
%of all functions in $L^p(0,1)$ whose differentials belong to $L^p(0,1)$ up to the order $k$. Recall that
%the usual Sobolev space $W^{k, p}(0,1)$ can be extended to
%the $W^{s, p}(0,1)$, for $s\in\mathbb{R}$. Set $H^k(0,1)\hat=W^{k, 2}(0,1)$
%and $H^1_0(0,1)$ is the subspace of $H^1(0,1)$ of all functions
%whose trace at $0$ and $1$ vanishes. We define the unbounded self-adjoint operator $A$ by
%$$Ax=\Delta x=\frac{\partial^2}{\partial \xi^2}x, \quad x\in \mathscr{D}(A)=H^2(0,1) \cap H^1_0(0,1).$$

%For $1 \leq p \leq \infty$,
%The square integrable functions on $[0,1]$ is denoted by $L^2: = L^2[0,1]$,
Let $L^2: = L^2(0,1)$ be the space of square integrable real-valued functions on the interval $[0,1]$.
The norm and the inner product on $L^2$ are
denoted by $\|\cdot\|$ and $\langle\cdot,\cdot\rangle$, respectively. The space $C^2_b(L^2)$ is the functions from $L^2$ to $\mathbb{R}$, which are twice continuously differentiable with first and second bounded derivative.
%which is equipped with the usual norm $|\cdot|_{L^p}$.
%For $p =2$, denote the inner product on the separable Hilbert space $L^2(0,1)$
%by $\langle\cdot,\cdot\rangle$.
For $k\in\mathbb{N}$, $W^{k, 2}(0,1)$ is the Sobolev space
of all functions in $L^2$ whose differentials belong to $L^2$ up to the order $k$.
The usual Sobolev space $W^{k, 2}(0,1)$ can be extended to
the $W^{s, 2}(0,1)$, for $s\in\mathbb{R}$. Set $H^k \hat=W^{k, 2}(0,1)$
and denote by $H^1_0$ the subspace of $H^1$ of all functions
whose trace at $0$ and $1$ vanishes.
The Laplacian operator $\Delta$ is given by
\begin{align*}
Ax = \Delta x = \frac{\partial^2}{\partial \xi^2}x,
\quad  x\in \mathscr{D}(A)=H^2 \cap H^1_0.
\end{align*}
%The Laplacian operator $\Delta$ is given by
%$\Delta x=\frac{\partial^2}{\partial \xi^2}x $,
%for any $ x\in \mathscr{D}(A)=H^2(0,1) \cap H^1_0(0,1).$
It is well known that $\Delta$ is the infinitesimal generator of a strongly continuous semigroup
$\{e^{t \Delta}\}_{t\geq0}$. The eigenfunctions of $\Delta$ are given by
$e_k(\xi)=\sqrt{2}\sin(k\pi\xi)$, $\xi\in[0,1],k\in \mathbb{N}$,
with the corresponding eigenvalues $\lambda_k = -k^2\pi^2$.
%$$e_k(\xi)=\sqrt{2}\sin(k\pi\xi),~~~\xi\in[0,1],k\in \mathbb{N}$$
%are eigenfunctions of $-A$ with eigenvalue $\lambda_k=k^2\pi^2$.
%Note that the operator $A$ is the infinitesimal generator of a strongly continuous semigroup in $H$,
%which we denote by $e^{tA},t\geq0$.
%Moreover, the semigroup $e^{tA}$ can be extended to $L^p(0,1)$ ($p>1$),
%\begin{equation}
%|e^{tA}x|_{L^p}\leq{e^{\gamma_pt}|x|_{L^p}},\quad x\in L^p(0,1),
%\end{equation}
%where $\gamma_p=2p^{-1}(p-1)\pi^2$.
%It is well known that  $e^{tA}$ ($t\geq0$) have smoothing properties, that is, for any
%$s_1,s_2\in\mathbb{R}$ with $s_1\leq s_2$, $r\geq1$, $e^{tA}:W^{s_1,r}(0,1)\to W^{s_2,r}(0,1)$ and there exists a
%constant $C$ which depending on $s_1,s_2,r$ such that
The operator $\Delta$ satisfies the smoothing property:
for any $s_1,s_2\in\mathbb{R}$ with $s_1\leq s_2$,
\begin{equation} \label{PSG}
|e^{t \Delta}z|_{H^{s_2}}\leq C\left(1+t^{(s_1-s_2)/2}\right)|z|_{H^{s_1}},
\quad \mbox{for} \  z\in  H^{s_1}.
\end{equation}
For any $\alpha \in \mathbb R$,
%Denote by $|\cdot|_{H^\alpha}$ the norm of the operator $(-A)^{\alpha/2}$.
$(-A)^\alpha$ is the power of the operator $-A$, and $|\cdot|_\alpha$ is the norm of
$\mathscr{D}((-A)^{\alpha/2})$ which is equivalent to the norm of $H^{\alpha}$.
%We have $|\cdot|_0=|\cdot|_{L^2}$, and denote it by $|\cdot|$ for simplicity.

Define the bilinear operator $B(x,y): L^2 \times H^1_0 \rightarrow H^{-1}_0$ by
$$ B(x,y)=x\cdot\partial_{\xi} y,$$
and the trilinear operator $ b(x,y,z): L^2 \times H^1_0 \times L^2 \rightarrow \RR$ by
$$ b(x,y,z)
=\int_0^1x(\xi) \partial_\xi y(\xi)z(\xi) d\xi.$$
For convenience, set $B(x)=B(x,x)$, for $x\in H^1_0$.

%\vskip 0.3cm

With the above notations, the system \eref{Equation} can be rewritten as: % the following form:
\begin{equation}\left\{\begin{array}{l}\label{main equation}
\displaystyle
dX^{\vare}_t=[AX^{\vare}_t+B(X^{\vare}_t)+f(X^{\vare}_t, Y^{\vare}_t)]dt+dW^{Q_1}_t,\quad X^{\vare}_0=x\\
dY^{\vare}_t=\frac{1}{\vare}[AY^{\vare}_t+g(X^{\vare}_t, Y^{\vare}_t)]dt+\frac{1}{\sqrt{\vare}}dW^{Q_2}_t,\quad Y^{\vare}_0=y\\
X^{\vare}_t(0)=X^{\vare}_t(1)=Y^{\vare}_t(0)=Y^{\vare}_t(1)=0.\end{array}\right.
\end{equation}
Here, the $Q_1$-Wiener process $W^{Q_1}_t$ is  given by
\begin{align} \label{Q1WienerPro}
W^{Q_1}_t=\sum^\infty_{k=1}\sqrt{\alpha_k}\beta^k_te_k, \quad t\geq0,
\end{align}
where $\alpha_k\geq 0$ satisfies  % $Q_1 e_k=\alpha_ke_k$ with
$Tr Q_1: = \sum_{k=1}^{\infty}\alpha_{k}<+\infty$, 
and $\{\beta^k\}_{k\in \mathbb N}$ is a sequence of mutually independent standard Brownian motions.
Throughout this paper, we assume that
 $W^{Q_2}_t$ also has a similar decomposition as in \eqref{Q1WienerPro} with $\text{Tr}Q_2<\infty$.
 Note that $W^{Q_1}_t$ and $W^{Q_2}_t$ are independent.

%A similar decomposition as in \eqref{Q1WienerPro} also holds for $W^{Q_2}_t$.
%Similarly, we can also write $W^{Q_2}_t=\sum^\infty_{k=1}\sqrt{\alpha_k'} \tilde{\beta}^k_te_k$
%In addition, we assume that $\text{Tr}Q_2<\infty$.

%\textcolor{red}{Now we make the following basic assumptions on the coefficients $f$ and $g$  throughout this paper.}

We impose the global Lipschitz condition
on the functions $f, g: L^2\times L^2 \rightarrow L^2$ in \eqref{main equation}.

%\textcolor{red}{The functions $f, g: H\times H \rightarrow H$ satisfy the global Lipschitz condition, i.e., there exist positive constants $L_{f}$ and $L_{g}$ such that for any $x_1,x_2,y_1,y_2\in H$,}

\begin{conditionA}\label{A1}
There exist two constants $L_{f}, L_{g} >0$ %and $L_{g}>0$ 
such that for any $x_1,x_2,y_1,y_2\in L^2$,
\begin{align*}
\|f(x_1, y_1)-f(x_2, y_2)\|\leq L_{f}(\|x_1-x_2\| + \|y_1-y_2\|),
\end{align*}
%and
\begin{align*}
\|g(x_1, y_1)-g(x_2, y_2)\|\leq L_{g}(\|x_1-x_2\| + \|y_1-y_2\|).
\end{align*}
\end{conditionA}

Following the standard approach developed in \cite{DaZ}, one can verify that under
the condition \ref{A1},
there exists a unique mild solution to the system \eqref{main equation}.
More specifically,
for any given initial value $x, y\in L^2$, and $T>0$,
there exist a unique
$X^{\varepsilon} \in C([0,T]; L^2) \cap L^2(0, T; H_0^1)$
and a unique $Y^{\varepsilon} \in C([0,T]; L^2) \cap L^2(0, T; H_0^1)$
%mild solution $\{(X^{\varepsilon}_t,Y^{\varepsilon}_t), t\geq 0\}$ to
%the system \eref{main equation} and for all $T>0$,
%$(X^{\varepsilon},Y^{\varepsilon})\in C([0,T]; L^2(0,1))\times C([0,T]; L^2(0,1))$
satisfying
%\textcolor{blue}{For the starting point in $L^2$, is it true that
%$X^{\varepsilon} \in C([0,T]; L^2) \cap L^2(0, T; H_0^1)$
%and $Y^{\varepsilon} \in C([0,T]; L^2) \cap L^2(0, T; H_0^1)$ ? }
%\PP-a.s.$.
\begin{equation}\left\{\begin{array}{l}\label{mild solution}
\displaystyle
X^{\varepsilon}_t=e^{tA}x+\int^t_0e^{(t-s)A}B(X^{\varepsilon}_s)ds+\int^t_0e^{(t-s)A}f(X^{\varepsilon}_s, Y^{\varepsilon}_s)ds+\int^t_0 e^{(t-s)A}dW^{Q_1}_s,\\
Y^{\varepsilon}_t=e^{tA/\varepsilon}y+\frac{1}{\varepsilon}\int^t_0e^{(t-s)A/\varepsilon}g(X^{\varepsilon}_s,Y^{\varepsilon}_s)ds
+\frac{1}{\sqrt{\varepsilon}}\int^t_0 e^{(t-s)A/\varepsilon}dW^{Q_2}_s.
\end{array}\right.
\end{equation}

We need the following dissipative condition,
which allows us to obtain the exponential ergodicity property of the fast component
with every fixed slow component in the second equation of \eqref{mild solution}.

%means that the growth rate of the function $g$ in fast component equation
%is smaller than the decay rate of operator $\Delta$.
%which is important to ....???

% \smallskip
%\noindent
%\textbf{(H2)}
\begin{conditionA}\label{A2}
$
\eta:=\lambda_{1}-L_{g}>0.
$
\end{conditionA}

The following condition on the $Q_1$-Wiener process $W^{Q_1}_t$ is used to establish the strong convergence
of $X^{\varepsilon}$ to $\bar{X}$.

% \smallskip
%\noindent
%\textbf{(H3)}

\begin{conditionA}\label{A3}
There exist constants $\alpha\in(1,\frac{3}{2})$ and $\beta\in(0,\frac{1}{2})$ such that
\begin{align*}
\sum_{k=1}^{\infty}\alpha_{k}\lambda_{k}^{\alpha+2\beta-1}<+\infty.
\end{align*}
\end{conditionA}

%Denote by $C^1(L^2, \mathbb{R})$, $C^2(L^2, \mathbb{R})$
%the spaces of real-valued functions $$
For any $x \in L^2$,
denote by $D\varphi(x)$, $D^2\varphi(x)$ the first and the second Frech\'{e}t derivatives
of the function $\varphi: L^2 \to \mathbb{R}$, respectively. By Riesz representation theorem,
we have
\begin{align*}
D\varphi(x) \cdot h = \langle D\varphi(x), h \rangle,
\quad
D^2\varphi(x) \cdot (h,k) = \langle D^2\varphi(x)h, k \rangle,  \quad  h,k \in L^2.
\end{align*}
%Introduce the following additional assumption:
The following condition is used to establish the weak convergence
of $X^{\varepsilon}$ to $\bar{X}$.

%\smallskip
%\noindent
%\textbf{(H4)}
\begin{conditionA}\label{A4}
Assume that  $f$ and $g$ are twice differentiable with respect to
the first and the second variable, respectively,
and that there exists a constant $C>0$ such that
for any $x,y, h, k \in L^2 $,  the following inequalities hold:
\begin{align}
& \|D_{xx}^{2}f(x,y)(h,k)\|\leq C\|h\| \|k\|  \label{SecondDer xx};  \\
& \|D_{y}g(x,y)\cdot h\|\leq C \|h\|, \quad   |D^2_{yy}g(x,y)(h,k)|\leq C\|h\| \|k\| \label{SecondDer yy}; \\
&  |\langle f(x,y),x\rangle| \leq C(1+ \|x\|^{2}).  \label{WeakDissIne}
\end{align}
%(1)For any $x,y, h, k \in L^2(0,1) $,  % and $h,k\in L^2(0,1) $,
%\begin{align}
%|D_{xx}^{2}f(x,y)(h,k)|\leq C\|h\| \|k\|. \nonumber
%\end{align}
%(2)For any $x,y, h, k \in L^2(0,1) $,
%\begin{align}
%|D_{y}g(x,y)\cdot h|\leq C \|h\|,
%\quad \quad |D^2_{yy}g(x,y)(h,k)|\leq C\|h\| \|k\|. \nonumber
%\end{align}
%(3) For any $x,y\in L^2(0,1)$,  %there exists a constant $C>0$ such that
%\begin{align}
%|\langle f(x,y),x\rangle|\leq C(1+ \|x\|^{2}). \nonumber
%\end{align}
\end{conditionA}

%The condition \eqref{WeakDissIne} is the weak dissipativity condition,
%which  allows us to
%\vskip 0.3cm

Now we are going to  formulate our main results.
The first result gives the convergence rate
in the sense of the trajectory distance
between the slow component $X_{t}^{\vare}$
and the averaged component $\bar{X}_{t}$,  as $\varepsilon \to 0$,
uniformly with respect to $t \in [0,T]$.

\begin{theorem}[Strong convergence]  \label{main result 1}
Assume the conditions \ref{A1}, \ref{A2} and \ref{A3} hold.
Then, for any $x\in H^{\alpha}$ with $\alpha$ given in \ref{A3},
$y\in L^2$, $p>0$ and $T>0$,
there exists a constant $C:=C_{x,y,T,p, \alpha}>0$ such that
\begin{align*}
\mathbb{E} \left(\sup_{0\leq t\leq T} \|X_{t}^{\vare}-\bar{X}_{t}\|^{2p} \right)
\leq C\Big(\frac{1}{-\log\vare}\Big)^{\frac{1}{4p}}\to 0,
\quad \mbox{as}  \  \varepsilon \to 0.
\end{align*}
%$where $C$ is a positive constant which only depends on $T$, $p$, $|x|_{\alpha}$ and $|y|$.
\end{theorem}

%Note that the convergence rate in Theorem \ref{main result 1} is slower than
%$\varepsilon^{1/2-r}$, for any small $r>0$, in \cite{B1},
%where the author considers the reaction-diffusion equations.
%However, it seems hopeless to improve the convergence rate in Theorem \ref{main result 1}
%to $\varepsilon^{1/2-r}$, due to the highly nonlinearity in the system \eqref{main equation}.

With some regularity for the initial value $x$,
the following result describes the convergence rate
in the sense of the law distance
between the slow component $X_{t}^{\vare}$
and the averaged component $\bar{X}_{t}$,  as $\varepsilon \to 0$.

\begin{theorem} [Weak convergence $1$]  \label{main result 2}
Assume the conditions \ref{A1}, \ref{A2} and \ref{A4} hold,
and that $Q_1=0$ in \eqref{main equation}.
Then for any $x\in H^{\theta}$ with $\theta\in(0,1]$, $y\in L^2$,
$\phi\in C_{b}^{2}(L^2)$, $t\in(0,T]$, $\delta\in(0,\frac{1}{2})$,
there exists a constant $C:=C_{x,y,T,\phi,\delta}>0$ such that for any $\vare\in(0,1)$, we have
%$\theta\in(0,1]$, $r\in(0,1)$, $x\in H^{\theta}$, $y\in H$,
%$\phi\in C_{b}^{2}$, $t\in(0,T]$, $\delta\in(0,\frac{1}{2})$,
%and for any small enough $\vare\in(0,1)$, we have
\begin{align} \label{MainThmInequa 01}
\big|\mathbb{E}[\phi(X^{\vare}_t)]-\mathbb{E}[\phi(\bar{X}_t)]\big|\leq C(1+t^{-\theta+\frac{\theta^2}{1+\delta}})\vare^{1-r}.
\end{align}
%where $C$ is a positive constant which only depends on $T$, $|x|_{\theta}, |y|, \delta$ and $\phi$.
\end{theorem}

The constant appeared in \eqref{MainThmInequa 01} can become quite terse at the cost of the
higher regularity condition on the initial value $x$.
More precisely, we have the following result:

\begin{theorem} [Weak convergence $2$] \label{main result 3}
Assume the conditions of Theorem \ref{main result 2} hold.
Then, for any $x\in H^{\theta}$ with $\theta\in(1,\frac{3}{2})$, $y\in L^2$,
$t\in(0,T]$ and $r\in (0,1)$, there exists a constant $C:=C_{x,y,T,\phi}>0$ such that
%Under the same assumptions in Theorem \ref{main result 2} with $\theta\in(1,\frac{3}{2})$,
%then for any $t\in(0,T]$, $r\in (0,1)$, we can obtain
\begin{align}
\big|\mathbb{E}[\phi(X^{\vare}_t)]-\mathbb{E}[\phi(\bar{X}_t)]\big|\leq C\vare^{1-r}. \nonumber
\end{align}
%where $C$ is a positive constant which only depends on $T$, $|x|_{\theta}, |y|$ and $\phi$.
\end{theorem}

%\begin{remark}
%Theorem \ref{main result 3} shows that if the initial value $x$ has higher regularity, then the control constant becomes quite terse.
%\end{remark}

\begin{remark}
In Theorems \ref{main result 2} and \ref{main result 3},
if we don't impose the condition $Q_{1}=0$ in the system \eqref{main equation},
there exist some essential difficulties.
For instance, it is not clear how to establish Lemma \ref{highorder of X},
which plays a crucial role in proving Theorems \ref{main result 2} and \ref{main result 3}.
%We assume (3) of \textbf{(H4)} holds and $Q_{1}=0$ in studying the weak convergence for several technical difficulty, for example, see Lemma \ref{highorder of X}, that for any $x\in H$ there exists a positive constant $C$ such that $\sup_{t\in[0,T]}|X^{\varepsilon}(t)| \leq C(1+|x|),$
%and then $\sup_{t\in[0,T]}e^{|X^{\varepsilon}(t)|^k}\leq C$, which plays an important role in proving Theorem \ref{main result 2} and Theorem \ref{main result 3}, see Subsection \ref{Subsection 5.5}. For the case of $Q_1\neq 0$, these estimates are failed and we do not know how to deal with this case recently.
\end{remark}

\section{Proof of Theorem \ref{main result 1}} \label{Sec Proof of Thm1}

In this section, we are devoted to proving Theorem \ref{main result 1}.
The proof consists of the following several steps.
In the first step,
we  give some priori estimates of the solution $(X^{\varepsilon}_t, Y^{\varepsilon}_t)$
to the system \eqref{main equation}.
In the second step,
following the idea inspired by Khasminskii in \cite{K1},
we introduce an auxiliary process
$(\hat{X}_{t}^{\varepsilon},\hat{Y}_{t}^{\varepsilon})\in L^2 \times L^2$
and also give the uniform bounds.
Meanwhile, we deduce an estimate of the process $X^{\varepsilon}_t-\hat{X}_{t}^{\varepsilon}$
in the space $L^{2p}(\Omega, C([0,T], L^2))$.
%Meanwhile, we show the error of $X^{\varepsilon}_t-\hat{X}_{t}^{\varepsilon}$.
%by using the smoothness of semigroup $e^{t\Delta}$ and interpolation inequality, we can further obtain $\sup_{\vare\in(0,1)}\mathbb{E}\sup_{t\in[0,T]}|X_{t}^{\varepsilon}|_{\alpha}^{p}\leq C_{p,T}$, which is a key step for proving the \eref{main result}, where $|\cdot|_{\alpha}$ is Sobolev norm.
In the last step, based on the ergodicity property of the averaged equation (see \eqref{FEQ}),
we make use of the skill of the stopping time  and some approximation techniques to
give a control of
$|\hat{X}_{t}^{\varepsilon}-\bar{X}_{t}^{\varepsilon}|_{L^{2p}(\Omega, C([0,T], L^2))}$.
Consequently, we deduce the convergence rate in Theorem \ref{main result 1}.

Recalling that $V:=H_0^1$ is continuously and densely embedded in $L^2$,
consider the Gelfand triple: $V \subset L^2 \subset V^*$, where
$V^*$ is the dual space of $V$. According to the Poincar\'{e} inequality,
we have that for any $x \in V$,
\begin{align} \label{Gelfand Ine}
_{V^*}\langle Ax, x \rangle_{V} = - \|\nabla x\|^2 \leq -\lambda_1 \|x\|^2,
\end{align}
where $_{V^*}\langle \cdot, \cdot \rangle_{V}$ denotes the dualization between $V^*$ and $V$.

%we study the average equation and apply the skill of stopping time and following the procedure inspired by \cite{FL},  the error of $\hat{X}_{t}^{\varepsilon}-\bar{X}_{t}^{\varepsilon}$ is obtained. Hence, the strong convergence is easily proved.
%Notice that we always assume the condition \ref{A3} in this section.

 \subsection{Some priori estimates of \texorpdfstring{$(X^{\varepsilon}_t, Y^{\varepsilon}_t)$} {Lg}}

We first prove the uniform bounds, with respect to $\varepsilon \in (0,1)$
and $t \in [0,T]$, for $p$-moments of the solutions to 
%$X_{t}^{\varepsilon}$ and $Y_{t}^{\varepsilon}$
the system \eref{main equation}.
The main proof follows the techniques in \cite{Mattin99} and \cite{LZ06, LZ09},
where the authors deal with the 2D stochastic Navier-Stokes equation and 1D stochastic
Burgers' equation, respectively.

\begin{lemma} \label{PMY}
Under  conditions \ref{A1} and \ref{A2},
for any $x,y\in L^2$, $p\geq2$ and $T>0$, there exists a constant $C_{p,T}>0$ such that
\begin{align} \label{Control X Ito 01}
\sup_{\vare\in(0,1)}\sup_{0\leq t\leq T}
\mathbb{E}\|X_{t}^{\varepsilon} \|^{2p}
\leq  C_{p,T}(1+ \|x\|^{2p} + \|y\|^{2p}),
\end{align}
%and
\begin{align} \label{Control Y Ito 01}
\sup_{\vare\in(0,1)}\sup_{0\leq t\leq T}
\mathbb{E} \|Y_{t}^{\varepsilon} \|^{2p}
\leq C_{p,T}(1+ \|x\|^{2p} + \|y\|^{2p}).
\end{align}
\end{lemma}

\begin{proof}
According to It\^{o}'s formula, we have
\begin{align} \label{ItoFormu 01}
\frac{d}{dt}\mathbb{E} \|Y_{t}^{\vare} \|^{2p}
= &\  \frac{2p\lambda_{1}}{\varepsilon}
\mathbb{E} \left[ \|Y_{t}^{\varepsilon}\|^{2p-2} (-|Y_{t}^{\varepsilon}|_1^2) \right]
%\frac{2p}{\varepsilon}
%\mathbb{E} \Big( \| Y_{t}^{\varepsilon} \|^{2p-2}
%\langle AY_{t}^{\varepsilon},Y_{t}^{\varepsilon}\rangle\Big)
+\frac{2p}{\varepsilon}\mathbb{E}\Big[\| Y_{t}^{\varepsilon} \|^{2p-2}\langle
g(X_{t}^{\varepsilon},Y_{t}^{\varepsilon}),Y_{t}^{\varepsilon}\rangle\Big] \nonumber\\
&\  + \frac{p}{\varepsilon}\mathbb{E} \| Y_{t}^{\varepsilon} \|^{2p-2}\text{Tr}Q_{2}
 +\frac{2p(p-1)}{\varepsilon}\mathbb{E} \| Y_{t}^{\varepsilon} \|^{2p-2}\text{Tr}Q_{2},
\end{align}
where It\^{o}'s formula can be understood in the way that
 we first use the Galerkin approximation to get  \eqref{ItoFormu 01} in the finite dimensional setting,
 then we take the limit of the dimension to obtain \eqref{ItoFormu 01}
 in the infinite dimensional setting.
%According to It\^{o}'s formula and
Using \eqref{Gelfand Ine}
%we get
%\begin{align*}
%\frac{d}{dt}\mathbb{E} \|Y_{t}^{\vare} \|^{2p}
%\leq &\   -\frac{2p\lambda_{1}}{\varepsilon}\mathbb{E} \|Y_{t}^{\varepsilon}\|^{2p}
%%\frac{2p}{\varepsilon}
%%\mathbb{E} \Big( \| Y_{t}^{\varepsilon} \|^{2p-2}
%%\langle AY_{t}^{\varepsilon},Y_{t}^{\varepsilon}\rangle\Big)
%+\frac{2p}{\varepsilon}\mathbb{E}\Big[\| Y_{t}^{\varepsilon} \|^{2p-2}\langle
%g(X_{t}^{\varepsilon},Y_{t}^{\varepsilon}),Y_{t}^{\varepsilon}\rangle\Big] \nonumber\\
%&\  + \frac{p}{\varepsilon}\mathbb{E} \| Y_{t}^{\varepsilon} \|^{2p-2}\text{Tr}Q_{2}
% +\frac{2p(p-1)}{\varepsilon}\mathbb{E} \| Y_{t}^{\varepsilon} \|^{2p-2}\text{Tr}Q_{2}.
%\end{align*}
and condition \ref{A1}, it follows from \eqref{ItoFormu 01} that
%there exists a constant $\gamma>0$ such that
\begin{eqnarray} \label{ItoFormu 001}
\frac{d}{dt}\mathbb{E} \| Y_{t}^{\vare} \|^{2p}
\leq \!\!\!\!\!\!\!\!&&
-\frac{2p\lambda_{1}}{\varepsilon} \mathbb{E} \|Y_{t}^{\varepsilon}\|^{2p}+
\frac{2p}{\varepsilon}\mathbb{E}\Big\{ \|Y_{t}^{\varepsilon} \|^{2p-2}
\left[C \|Y_{t}^{\varepsilon} \| + L_{g}( \| X_{t}^{\varepsilon} \|  \|Y_{t}^{\varepsilon} \|
+ \|Y_{t}^{\varepsilon} \|^2)\right] \Big\}
\nonumber \\
 \!\!\!\!\!\!\!\!&&
 + \frac{p}{\varepsilon}\mathbb{E} \| Y_{t}^{\varepsilon} \|^{2p-2}\text{Tr}Q_{2}
 +\frac{2p(p-1)}{\varepsilon}\mathbb{E} \| Y_{t}^{\varepsilon} \|^{2p-2}\text{Tr}Q_{2}.
\end{eqnarray}
From \eqref{ItoFormu 001},
using condition \ref{A2} and the Young inequality,
we deduce that
there exists a constant $\gamma>0$ such that
\begin{align}  \label{4.4.2}
\frac{d}{dt}\mathbb{E} \|Y_{t}^{\vare} \|^{2p}
\leq -\frac{p\gamma}{\varepsilon}\mathbb{E} \|Y_{t}^{\varepsilon} \|^{2p}
+\frac{C_{p}}{\varepsilon}\EE \| X_{t}^{\varepsilon} \|^{2p}+\frac{C_{p}}{\varepsilon}.
\end{align}
%where the last inequality by the fact of $\lambda_1-L_g>0$ in \textbf{(H2)} and Young inequality.\\
Applying the comparison theorem gives
\begin{eqnarray} \label{4.4.3}
\mathbb{E} \| Y_{t}^{\varepsilon} \|^{2p}
\leq \!\!\!\!\!\!\!\!&&
\|y\|^{2p} e^{-\frac{p\gamma}{\varepsilon}t}
+ \frac{C_{p}}{\varepsilon}\int^t_0 e^{-\frac{p\gamma}{\varepsilon}(t-s)}
\Big(1+\EE \|X_{s}^{\varepsilon} \|^{2p}\Big)ds.
\end{eqnarray}
For $X_{t}^{\varepsilon}$,
note that by Lemma \ref{Property B0}, $b(x,y,y)=0$ for any $x,y \in H_0^1$.
Similarly to \eqref{ItoFormu 01}, applying It\^{o}'s formula, we have
\begin{eqnarray*}
\frac{d}{dt}\mathbb{E} \| X_{t}^{\vare} \|^{2p}
= \!\!\!\!\!\!\!\!&&
2p \lambda_1 \mathbb{E} \left[ \|X_{t}^{\varepsilon}\|^{2p-2} |X_{t}^{\varepsilon}|_1^2 \right]
%\mathbb{E}\Big(\| X_{t}^{\varepsilon} \|^{2p-2}
%\langle AX_{t}^{\varepsilon},X_{t}^{\varepsilon}\rangle\Big)
+2p\mathbb{E}
\Big[\| X_{t}^{\varepsilon} \|^{2p-2}
\langle f(X_{t}^{\varepsilon},Y_{t}^{\varepsilon}),X_{t}^{\varepsilon}\rangle\Big] \nonumber \\
 \!\!\!\!\!\!\!\!&& + p\mathbb{E} \| X_{t}^{\varepsilon} \|^{2p-2}\text{Tr}Q_{1}
 +2p(p-1)\mathbb{E}\Big\{ \| X_{t}^{\varepsilon} \|^{2p-4} \|\sqrt{Q_1}X^\varepsilon_t \|^2\Big\}.
\end{eqnarray*}
%using It\^{o}'s formula,
Using \eqref{Gelfand Ine},
%and Lemma \ref{Property B0},
we obtain
\begin{eqnarray*}
\frac{d}{dt}\mathbb{E} \| X_{t}^{\vare} \|^{2p}
\leq \!\!\!\!\!\!\!\!&&
-2p \lambda_1 \mathbb{E} \|X_{t}^{\varepsilon}\|^{2p}
%\mathbb{E}\Big(\| X_{t}^{\varepsilon} \|^{2p-2}
%\langle AX_{t}^{\varepsilon},X_{t}^{\varepsilon}\rangle\Big)
+2p\mathbb{E}
\Big[\| X_{t}^{\varepsilon} \|^{2p-2}
\langle f(X_{t}^{\varepsilon},Y_{t}^{\varepsilon}),X_{t}^{\varepsilon}\rangle\Big] \nonumber \\
 \!\!\!\!\!\!\!\!&& + p\mathbb{E} \| X_{t}^{\varepsilon} \|^{2p-2}\text{Tr}Q_{1}
 +2p(p-1)\mathbb{E}\Big\{ \| X_{t}^{\varepsilon} \|^{2p-4} \|\sqrt{Q_1}X^\varepsilon_t \|^2\Big\}.
\end{eqnarray*}
%\begin{eqnarray*}
%|X_{t}^{\varepsilon}|^{2p}=\!\!\!\!\!\!\!\!&&|x|^{2p}+2p\int_{0} ^{t}| X_{s}^{\varepsilon}|^{2p-2}\langle AX_{s}^{\varepsilon},X_{s}^{\varepsilon}\rangle ds+2p\int_{0} ^{t}| X_{s}^{\varepsilon}|^{2p-2}\langle B(X_{s}^{\varepsilon}),X_{s}^{\varepsilon}\rangle ds \\
% \!\!\!\!\!\!\!\!&&+2p\int_{0} ^{t}|X_{s}^{\varepsilon}|^{2p-2}\langle f(X_{s}^{\varepsilon},Y_{s}^{\varepsilon}),X_{s}^{\varepsilon}\rangle ds +2p\int_{0} ^{t}| X_{s}^{\varepsilon}|^{2p-2}\langle X_{s}^{\varepsilon}, dW^{Q_1}_s\rangle\\
% \!\!\!\!\!\!\!\!&& +p\int_{0} ^{t}|X_{s}^{\varepsilon}|^{2p-2}\text{Tr}Q_{1}ds
% +2p(p-1)\int_{0} ^{t}|X_{s}^{\varepsilon}|^{2p-4}|\sqrt{Q_1}X^\varepsilon_s|^2ds.
%\end{eqnarray*}
In the same way as in \eref{4.4.2} and \eqref{4.4.3}, one can verify that
%\begin{align*}
%\frac{d}{dt}\mathbb{E}|X_{t}^{\vare}|^{2p}
%\leq C_p\mathbb{E}|X_{t}^{\varepsilon}|^{2p}+C_p\EE|Y_{t}^{\varepsilon}|^{2p}+C_p.
%\end{align*}
%Hence, by comparison theorem
\begin{eqnarray} \label{4.4.4}
\mathbb{E} \| X_{t}^{\varepsilon} \|^{2p}
\leq\!\!\!\!\!\!\!\!&&
\| x \|^{2p}e^{C_p t}
+ C_p\int^t_0 e^{C_p(t-s)}\left(1+\EE \| Y_{s}^{\varepsilon} \|^{2p}\right)ds.
\end{eqnarray}
Combining \eref{4.4.3} and \eref{4.4.4}, we get that, for any $t\in [0, T]$,
\begin{align*}
\mathbb{E} \| Y_{t}^{\varepsilon} \|^{2p}
\leq
C_{p,T}(1+ \| x \|^{2p}+ \|y\|^{2p})+\frac{C_{p}}{\varepsilon}
\int^t_0 e^{-\frac{p\gamma}{\varepsilon}(t-s)}\int^s_0 \mathbb{E} \|Y_{r}^{\varepsilon} \|^{2p}drds
+\frac{C_{p}}{\varepsilon}\int^t_0 e^{-\frac{p\gamma}{\varepsilon}(t-s)}ds,
\end{align*}
which implies
%With a change of variable, we have
\begin{eqnarray*}
\mathbb{E} \| Y_{t}^{\varepsilon} \|^{2p}
%\leq\!\!\!\!\!\!\!\!&&
%C_{p,T}(1+ \|x \|^{2p} + \|y \|^{2p})
%+C_p\int^t_0 \Big[\int^{\frac{t-r}{\varepsilon}}_0 e^{-p\gamma s}ds\Big]
%\mathbb{E} \| Y_{r}^{\varepsilon} \|^{2p}dr
%+ C_p\int^{t/\varepsilon}_0 e^{-p\gamma s}ds  \nonumber\\
\leq \!\!\!\!\!\!\!\!&&  C_{p,T}(1+ \|x \|^{2p} + \| y \|^{2p})
+C_p \int^t_0 \mathbb{E} \| Y_{r}^{\varepsilon} \|^{2p}dr.
\end{eqnarray*}
Using Gronwall's inequality, we get \eqref{Control Y Ito 01}.
The inequality \eqref{Control X Ito 01} follows by combining \eqref{Control Y Ito 01} and \eqref{4.4.4}.
The proof is complete.
%Using Grownall inequality, it follows that
%\begin{eqnarray*}
%\mathbb{E}|Y_{t}^{\varepsilon}|^{2p}\leq C_{p,T}(1+|x|^{2p}+|y|^{2p}),\label{4.4.5}
%\end{eqnarray*}
%which also gives
%\begin{eqnarray*}
%\mathbb{E}|X_{t}^{\varepsilon}|^{2p}\leq C_{p,T}(1+|x|^{2p}+|y|^{2p}).\label{4.4.6}
%\end{eqnarray*}
%The proof is complete.
%\hspace{\fill}$\square$
\end{proof}

In order to estimate the high-order norm of $|X^{\varepsilon}_t|_{\alpha}$,
with $\alpha \in (1,\frac{3}{2})$,  we first give a control of the stochastic convolution
$W_{A}(t):=\int_{0}^{t}e^{(t-s)A}dW^{Q_{1}}_s.$
%Then we have the following result:
\begin{lemma} \label{stochastic convolution}
Under the condition \ref{A3},
for any $p, T>0$ and $\alpha \in (1, \frac{3}{2})$ which is given in \ref{A3}, 
there exists a positive constant $C_{p,T}$ such that
\begin{align} \label{4.2}
\EE\sup_{0\leq t\leq T}|W_{A}(t)|^{2p}_{\alpha}\leq C_{p,T}.
\end{align}
%where $\alpha$ is the one in \textbf{(H3)}.
\end{lemma}
\begin{proof}
By the H\"{o}lder inequality, it suffices to prove \eref{4.2} for large enough $p$.
Using the factorization method, for $\beta \in (0, \frac{1}{2})$ in \ref{A3}, we write
$$
W_{A}(t)=\frac{\sin \pi\beta}{\pi}\int^t_0 e^{(t-s)A}(t-s)^{\beta-1}Z_sds,
$$
where
%$\beta$ is the one in \textbf{(H3)} and
$$
Z_s=\int^s_0 e^{(s-r)A}(s-r)^{-\beta}dW^{Q_1}_r.
$$
%For any $T>0$, $t\in[0,T]$, $p$ is large enough such that $\frac{2p(1-\beta)}{2p-1}<1$, we have
Choosing $p>1$ large enough such that $\frac{2p(1-\beta)}{2p-1}<1$, we get
\begin{eqnarray*}
|W_{A}(t)|_{\alpha}\leq C\left(\int^t_0 (t-s)^{-\frac{2p(1-\beta)}{2p-1}}ds\right)^{\frac{2p-1}{2p}}|Z|_{L^{2p}(0,T; H^{\alpha})}
\leq C_{p}t^{\beta-\frac{1}{2p}}|Z|_{L^{2p}(0,T; H^{\alpha})},
\end{eqnarray*}
which implies
\begin{eqnarray}
\sup_{0 \leq t\leq T}|W_{A}(t)|^{2p}_{\alpha}\leq\!\!\!\!\!\!\!\!&&C_{p,T}|Z|^{2p}_{L^{2p}(0,T; H^{\alpha})}.\label{4.3}
\end{eqnarray}
Notice that $(-A)^{\alpha/2}Z_s\sim N(0, \tilde{Q}_s)$,  
which is a Gaussian random variable with mean zero and the covariance operator given by
$$
\tilde{Q}_s x=\int^s_0 r^{-2\beta}e^{rA}(-A)^{\alpha}Q_1 e^{rA^{*}}xdr.
$$
For any $p\geq1$, $s>0$, we follow the proof of \cite[Corollary 2.17]{Daz1} to obtain
\begin{align}
\EE|(-A)^{\alpha/2}Z_s|^{2p}\leq &\   C_{p}[\text{Tr}(\tilde{Q}_s)]^p
= C_p\left( \sum_{k=1}^{\infty} \int^s_0 r^{-2\beta}e^{-2r\lambda_k}\lambda^{\alpha}_k\alpha_k dr\right)^p \nonumber\\
= &  \ C_p\left(\sum_{k=1}^{\infty} \lambda^{\alpha+2\beta-1}_k\alpha_k
\int^{2s\lambda_k}_0 r^{-2\beta}e^{-r}dr\right)^p
\leq C_{p} \big(\sum_{k=1}^{\infty} \lambda^{\alpha+2\beta-1}_k\alpha_k \big)^p <\infty ,\label{4.4}
\end{align}
where in the last two inequalities we use the fact
$
\int^{2s\lambda_k}_0 r^{-2\beta}e^{-r}dr\leq\int^{\infty}_0 r^{-2\beta}e^{-r}dr,
$
and the condition \ref{A3}.
We conclude the proof of Lemma \ref{stochastic convolution}
by combining \eref{4.3} and \eref{4.4}.
%Consequently, \eref{4.3} and \eref{4.4} imply
%\begin{eqnarray*}
%\EE\sup_{0\leq t\leq T}|W_{A}(t)|^{2p}_{\alpha}\leq\!\!\!\!\!\!\!\!&&C_{p,T}\int^T_0\EE|Z_s|^{2p}_{\alpha}ds\leq C_{p,T}.
%\end{eqnarray*}
%\hspace{\fill}$\square$
\end{proof}

\begin{lemma} \label{SOX}
Under the conditions \ref{A1}-\ref{A3},
for $T>0$ and $p>0$, there exists a constant $C_{p,T}>0$ such that
for any $x\in H^{\alpha}$ with $\alpha$ given in \ref{A3}, 
and for any $y\in L^2$,  we have
%$T>0$ and $p>0$, there exists a positive constant $C_{p,T}$ independent of $\varepsilon$, such that
\begin{align}
\sup_{\varepsilon \in (0,1)} \mathbb{E}\big(\sup_{t\in[0,T]}|X_{t}^{\varepsilon}|_{\alpha}^{2p}\big)
\leq C_{p,T}(1+|x|^{2p}_{\alpha}+ \|y \|^{2p}). \label{4.5}
\end{align}
%where $\alpha$ is given in \ref{A3}.
\end{lemma}

\begin{proof}
Using the H\"{o}lder inequality,
it suffices to prove \eref{4.5} for large enough $p$. Recall that
\begin{align*}
X^{\varepsilon}_t=e^{tA}x+\int^t_0e^{(t-s)A}B(X^{\varepsilon}_s)ds+\int^t_0e^{(t-s)A}f(X^{\varepsilon}_s, Y^{\varepsilon}_s)ds+\int^t_0 e^{(t-s)A}dW^{Q_1}_s.
\end{align*}
For the first term, it is clear that $|e^{At}x|_{\alpha}^{2p}\leq|x|_{\alpha}^{2p}$.
For the second term, according to $\eref{PSG}$ and Lemma \ref{Property B1}, we have %the following estimate
\begin{eqnarray*}
\Big|\int^t_0e^{(t-s)A}B(X^{\varepsilon}_s)ds\Big|_{\alpha}\leq \!\!\!\!\!\!\!\!&& \int^t_0\big|e^{(t-s)A}B(X^{\varepsilon}_s)\big|_{\alpha}ds \nonumber\\
\leq \!\!\!\!\!\!\!\!&& C\int^t_0 \big(1+(t-s)^{\frac{-\alpha_{3}-\alpha}{2}}\big)|B(X^{\varepsilon}_s)|_{-\alpha_{3}}ds  \nonumber\\
\leq \!\!\!\!\!\!\!\!&&
C\int^t_0\big(1+(t-s)^{\frac{-\alpha_{3}-\alpha}{2}}\big)|X^{\varepsilon}_s|_{\alpha_{1}}|X^{\varepsilon}_s|_{\alpha_{2}+1}ds,
\end{eqnarray*}
where $\alpha_{1}+\alpha_{2}+\alpha_{3}>\frac{1}{2}$ and $\alpha_{i}>0, i=1,2,3$.
Using  the interpolation inequality, we have that
\begin{align} \label{IEB1}
|X^{\varepsilon}_s|_{\alpha_{1}}
\leq C \| X^{\varepsilon}_s \|^{\frac{\alpha-\alpha_{1}}{\alpha}}
|X^{\varepsilon}_s|_{\alpha}^{\frac{\alpha_{1}}{\alpha}},
\end{align}
for any $0<\alpha_{1}<\alpha$, and that
\begin{align} \label{IEB2}
|X^{\varepsilon}_s|_{\alpha_{2}+1}
\leq C \| X^{\varepsilon}_s \|^{\frac{\alpha-\alpha_{2}-1}{\alpha}}
|X^{\varepsilon}_s|_{\alpha}^{\frac{\alpha_{2}+1}{\alpha}},
\end{align}
for any $0<\alpha_{2}+1<\alpha$.
%Besides, we also assume
Let $\alpha_{1}$ and $\alpha_{2}$  be small enough such that
$1+\alpha_{1}+\alpha_{2} \in (1,\alpha)$.
It follows from $\eref{IEB1}$ and $\eref{IEB2}$ that
%Then, combining $\eref{IEB1}$ and $\eref{IEB2}$, we obtain
\begin{eqnarray}
\!\!\!\!\!\!\!\!&&\mathbb{E}\sup_{t\in[0,T]}
\Big|\int^t_0e^{(t-s)A}B(X^{\varepsilon}_s)ds\Big|_{\alpha}^{2p} \nonumber\\
\leq\!\!\!\!\!\!\!\!&& C\mathbb{E}\Big(\sup_{t\in[0,T]}
\int^t_0 \big(1+(t-s)^{\frac{-\alpha_{3}-\alpha}{2}}\big)
\| X^{\varepsilon}_s \|^{\frac{2\alpha-\alpha_{1}-\alpha_{2}-1}{\alpha}}
|X^{\varepsilon}_s|_{\alpha}^{\frac{\alpha_{1}+\alpha_{2}+1}{\alpha}}ds\Big)^{2p} \nonumber\\
\leq\!\!\!\!\!\!\!\!&&
C_{p,T}\mathbb{E}
\Big[\Big(\sup_{t\in[0,T]}
\int^t_0 \big(1+(t-s)^{\frac{-\alpha_{3}-\alpha}{2}}\big)^{\frac{2p}{2p-1}}ds\Big)^{2p-1}
\nonumber\\
\!\!\!\!\!\!\!\!&&
\times \Big(\sup_{t\in[0,T]}\int^t_0 \|X^{\varepsilon}_s \|^{2p\cdot\frac{2\alpha-\alpha_{1}
-\alpha_{2}-1}{\alpha}}
|X^{\varepsilon}_s|_{\alpha}^{2p\cdot\frac{\alpha_{1}+\alpha_{2}+1}{\alpha}}ds\Big)\Big]  \nonumber\\
\leq\!\!\!\!\!\!\!\!&&C_{p,T}\Big(1+\int^T_0 s^{-\frac{\alpha_{3}+\alpha}{2}\cdot\frac{2p}{2p-1}}ds\Big)^{2p-1}
\Big(\int^T_0\mathbb{E} \|X^{\varepsilon}_s \|^{2p\cdot\frac{2\alpha-\alpha_{1}-\alpha_{2}-1}{\alpha-\alpha_{1}-\alpha_{2}-1}}ds
+\int^T_0\mathbb{E}|X^{\varepsilon}_s|_{\alpha}^{2p}ds\Big).  \nonumber
\end{eqnarray}
By Lemma \ref{PMY}, we obtain
\begin{align} \label{HolderNorm B}
\mathbb{E}\sup_{t\in[0,T]}\Big|\int^t_0e^{(t-s)A}B(X^{\varepsilon}_s)ds\Big|_{\alpha}^{2p}
\leq C_{p,T}\Big(1+\int^T_0 s^{-\frac{\alpha_{3}+\alpha}{2}\cdot\frac{2p}{2p-1}}ds\Big)^{2p-1}
\Big(1+\int^T_0\mathbb{E}|X^{\varepsilon}_s|_{\alpha}^{2p}ds\Big).
\end{align}
We choose positive constants $\alpha_{1}, \alpha_{2}, \alpha_{3}$ such that
$
0<1+\alpha_{1}+\alpha_{2}<\alpha
$ and $ \alpha_{1}+\alpha_{2}+\alpha_{3}>\frac{1}{2} $.
Let $p$ be large enough such that
$
\frac{\alpha+\alpha_{3}}{2}\cdot\frac{2p}{2p-1}<1
$
(for instance, $\alpha_{3}=\frac{1}{2}$, $\alpha_{1}=\alpha_{2}=\frac{\alpha-1}{4}$, and $p>\frac{2}{3-2\alpha}$).
Consequently, from \eqref{HolderNorm B}, we get
\begin{eqnarray}\label{IE351}
     \mathbb{E}\Big(\sup_{t\in[0,T]}\Big|\int^t_0e^{(t-s)A}B(X^{\varepsilon}_s)ds\Big|_{\alpha}^{2p}\Big)
\leq C_{p,T}
     \Big(1+\int^T_0\mathbb{E}|X^{\varepsilon}_s|_{\alpha}^{2p}ds\Big).
\end{eqnarray}
For the third term, according to $\eref{PSG}$, we obtain
\begin{eqnarray}
\!\!\!\!\!\!\!\!&&\mathbb{E}\Big(\sup_{t\in[0,T]}\Big|\int^t_0e^{(t-s)A}f(X^{\varepsilon}_s, Y^{\varepsilon}_s)ds\Big|_{\alpha}^{2p}\Big) \nonumber\\
%\leq\!\!\!\!\!\!\!\!&&
%C\mathbb{E}\Big[\sup_{t\in[0,T]}
%\int^t_0 \big(1+(t-s)^{-\frac{\alpha}{2}}\big)\|f(X^{\varepsilon}_s, Y^{\varepsilon}_s)\|ds\Big]^{2p}
%\nonumber\\
\leq\!\!\!\!\!\!\!\!&&
C\mathbb{E}\Big[\sup_{t\in[0,T]}\int^t_0 \big(1+(t-s)^{-\frac{\alpha}{2}}\big)
\big(1 + \|X^{\varepsilon}_s \| + \|Y^{\varepsilon}_s \| \big)ds\Big]^{2p}  \nonumber\\
%\leq\!\!\!\!\!\!\!\!&&
%C_{p}\Big(\sup_{t\in[0,T]}\int^t_0\big(1+(t-s)^{-\frac{\alpha}{2}}\big)^{\frac{2p}{2p-1}}ds\Big)^{2p-1}
%\mathbb{E}\Big(\sup_{t\in[0,T]}
%\int^t_0 \big(1+ \|X^{\varepsilon}_s \| + \| Y^{\varepsilon}_s \| \big)^{2p}ds\Big)  \nonumber\\
\leq\!\!\!\!\!\!\!\!&&C_{p,T}\Big(1+\int^T_0 s^{-\frac{\alpha p}{2p-1}}ds\Big)^{2p-1}
\int^T_0\big(1+\mathbb{E} \| X^{\varepsilon}_s \|^{2p} + \mathbb{E} \| Y^{\varepsilon}_s \|^{2p}\big)ds.  \nonumber
\end{eqnarray}
Taking $p$ large enough such that $\frac{\alpha p}{2p-1}<1$,
it follows from  Lemma \ref{PMY} that
\begin{align}
\mathbb{E}\Big(\sup_{t\in[0,T]}\Big|\int^t_0e^{(t-s)A}f(X^{\varepsilon}_s, Y^{\varepsilon}_s)ds\Big|_{\alpha}^{2p}\Big)\leq C_{p,T}(1+ \| x \|^{2p}+ \| y \|^{2p}).  \label{IE352}
\end{align}
%For the last term, by Lemma \ref{stochastic convolution} we have
%\begin{eqnarray} \label{IE353}
%\mathbb{E}\sup_{t\in[0,T]}\Big|\int^t_0 e^{(t-s)A}dW^{Q_1}_s\Big|_{\alpha}^{2p}
%\mathbb{E}\Big|\sum^{\infty}_{k=1}\sqrt{\alpha_{k}^{1}}(-A)^{\frac{\alpha}{2}}\int_{0}^{t}e^{A(t-s)}e_{k}(\xi)d\beta_{k}(s)\Big|^{2p}  \nonumber\\
%=\!\!\!\!\!\!\!\!&&
%\mathbb{E}\Big|\sum^{\infty}_{k=1}\sqrt{\alpha_{k}^{1}}\lambda_{k}^{\frac{\alpha}{2}}\int_{0}^{t}e^{-\lambda_{k}(t-s)}e_{k}(\xi)d\beta_{k}(s)\Big|^{2p}  \nonumber\\
%\leq\!\!\!\!\!\!\!\!&&
%\mathbb{E}\Big(\sum^{\infty}_{k=1}\sqrt{\alpha_{k}^{1}}\lambda_{k}^{\frac{\alpha}{2}}\sup_{t\in[0,T]}\int_{0}^{t}e^{-\lambda_{k}(t-s)}d\beta_{k}(s)\Big)^{2p} \nonumber\\
%\leq C_{p,T}\big(\sum^{\infty}_{k=1}\alpha_{k}^{1}\lambda_{k}^{\alpha-1}\big)^{p}
%\leq C_{p,T}.
%\end{eqnarray}
We conclude  the proof of Lemma \ref{SOX} by
combining $\eref{IE351}, \eref{IE352}$, Lemma \ref{stochastic convolution} and Gronwall's inequality.
%\begin{align*}
%\sup_{\vare\in(0,1)}\mathbb{E}\big(\sup_{t\in[0,T]}|X_{t}^{\varepsilon}|_{\alpha}^{2p}\big)\leq C_{p,T}(1+|x|^{2p}_{\alpha}+|y|^{2p}),
%\end{align*}
%which completes the proof.
%\hspace{\fill}$\square$
\end{proof}

%\vskip 0.3cm

%\begin{Rem} \label{EXtp}
%Provided that the initial data $x\in D(-A)^{\theta}$, $\theta\in(0,\frac{1}{2})$ and we don't take supremum, according to $|e^{At}x|_{\alpha}^{2p}\leq %Ct^{-p(\alpha-2\theta)}|x|_{2\theta}^{2p}$, then we have
%\begin{align}
%\mathbb{E}|X_{t}^{\varepsilon}|_{\alpha}^{2p}\leq Ct^{-p(\alpha-2\theta)}|x|_{2\theta}^{2p}+C. \nonumber
%\end{align}
%Besides, the above result still holds for $Q_{1}=0$.
%\hspace{\fill}$\square$
%\end{Rem}

Now we are equipped to prove the H\"{o}lder continuity of $t \mapsto X_{t}^{\varepsilon}$,
which holds uniformly with respect to $\varepsilon \in (0,1)$.
\begin{lemma} \label{COX}
Under the conditions \ref{A1}-\ref{A3},
for any $x\in H^{\alpha}, y\in L^2$, $T>0$, $0<t\leq t+h\leq T$,
there exists a constant $C_{p,T}>0$ such that
\begin{align*}
\sup_{\vare\in(0,1)}\mathbb{E} \| X_{t+h}^{\varepsilon}-X_{t}^{\varepsilon} \|^{2p}
\leq C_{p,T}h^{p}(1+|x|^{4p}_{\alpha} + \| y \|^{4p}).
\end{align*}
\end{lemma}

\begin{proof}
After simple calculations, we have
\begin{eqnarray}
X_{t+h}^{\varepsilon}-X_{t}^{\varepsilon}=\!\!\!\!\!\!\!\!&&(e^{Ah}-I)X_{t}^{\varepsilon}+\int_{t}^{t+h}e^{(t+h-s)A}B(X^{\varepsilon}_s)ds \nonumber\\
\!\!\!\!\!\!\!\!&&+\int_{t}^{t+h}e^{(t+h-s)A}f(X^{\varepsilon}_s, Y^{\varepsilon}_s)ds+\int_{t}^{t+h}e^{(t+h-s)A}dW^{Q_1}_s  \nonumber\\
:=\!\!\!\!\!\!\!\!&&I_{1}+I_{2}+I_{3}+I_{4}.  \nonumber
\end{eqnarray}
For $I_1$, note that
for $\alpha$ given in \ref{A3}, there exists a constant $C_{\alpha}>0$ such that for any $x\in \mathscr{D}((-A)^{\frac{\alpha}{2}})$,
$
\| e^{Ah}x-x \| \leq C_{\alpha}h^{\frac{\alpha}{2}}|x|_{\alpha}.
$
Then using Lemma \ref{SOX}, we get
\begin{align}  \label{REGX1}
\mathbb{E} \| I_{1} \|^{2p} \leq
C_{\alpha}h^{\alpha p}\mathbb{E}|X^{\varepsilon}_t|_{\alpha}^{2p}
\leq  C_{p,T}h^{\alpha p}(1+|x|^{2p}_{\alpha} + \| y \|^{2p}).
\end{align}
For $I_{2}$, using the contractive property of the semigroup $e^{tA}$,
Corollary \ref{Property B3} and Lemma \ref{SOX},
we obtain
\begin{align}\label{REGX2}
\mathbb{E} \|I_{2} \|^{2p}
\leq & \
%\mathbb{E}\Big(\int_{t}^{t+h}
%\big \| e^{(t+h-s)A}B(X^{\varepsilon}_s) \big\| ds\Big)^{2p} %\nonumber\\
%\leq
%\!\!\!\!\!\!\!\!&&
\mathbb{E}\Big(\int_{t}^{t+h} \|B(X^{\varepsilon}_s) \| ds\Big)^{2p} 
\leq   C\mathbb{E}\Big(\int_{t}^{t+h} |X^{\varepsilon}_s|^2_1 ds\Big)^{2p}  \nonumber\\
\leq & \  C_{p,T}h^{2p}\mathbb{E}\sup_{s\in [0, T] }|X^{\varepsilon}_s|_1^{4p}
\leq C_{p, T}h^{2p}(1+|x|^{4p}_{\alpha} + \| y\|^{4p}).
\end{align}
For $I_{3}$, applying condition \ref{A1} and Lemma \ref{PMY}, we get
\begin{align}  \label{REGX3}
\mathbb{E} \|I_{3} \|^{2p}\leq
& \ 
%\mathbb{E}\Big(\int_{t}^{t+h}\|f(X^{\varepsilon}_s, Y^{\varepsilon}_s)\|ds\Big)^{2p}
%\leq  
h^{2p-1}\mathbb{E}\!\!\int_{t}^{t+h}\|f(X^{\varepsilon}_s, Y^{\varepsilon}_s)\|^{2p}ds \nonumber\\
\leq & \  Ch^{2p-1}\mathbb{E}
\int_{t}^{t+h}\big(1+ \| X^{\varepsilon}_s \| + \| Y^{\varepsilon}_s \| \big)^{2p}ds
\leq  C_{p,T}h^{2p}(1 + \|x \|^{2p} + \|y \|^{2p}).
\end{align}
%For $I_{4}$, notice that $I_4\sim N(0, S_h)$ is a Gaussian random variable, where
%$$
%S_h x=\int^{h}_0 e^{(h-r)A}Q_1 e^{(h-r)A^{*}}xdr.
%$$
For $I_{4}$, note that $I_4$ is the centered Gaussian random variable with the variance given by
$S_h x=\int^{h}_0 e^{(h-r)A}Q_1 e^{(h-r)A^{*}}xdr.$
Then, for any $p\geq1$, we get
%we following the proof of \cite[Corollary 2.17]{Daz1}
\begin{align}\label{REGX4}
\EE \|I_4 \|^{2p}\leq  C_{p}[\text{Tr}(S_h)]^p
= C_p\left( \sum_{k=1}^{\infty} \int^h_0 e^{-2(h-r)\lambda_k}\alpha_k dr\right)^p
\leq   C_{p}(\sum_{k=1}^{\infty} \alpha_k)^p h^p.
\end{align}
Putting \eref{REGX1}-\eref{REGX4} together, the result follows.
%Hence,  $\eref{REGX1}-\eref{REGX4}$ imply
%\begin{align*}
%\sup_{\vare\in(0,1)}\mathbb{E}|X_{t+h}^{\varepsilon}-X_{t}^{\varepsilon}|^{2p}\leq C_{p,T}h^{p}(1+|x|^{2p}_{\alpha}+|y|^{2p}).
%\end{align*}
%The proof is complete.
%\hspace{\fill}$\square$
\end{proof}

%\begin{Rem} \label{COXQ}
%The above result still holds if $Q_{1}=0$. \hspace{\fill}$\square$
%\end{Rem}

%\vskip 0.3cm

\subsection{ Estimates of the auxiliary process
\texorpdfstring{ $(\hat{X}_{t}^{\varepsilon},\hat{Y}_{t}^{\varepsilon})$} {Lg} }

Following the idea inspired by Khasminskii \cite{K1},
we introduce an auxiliary process
$(\hat{X}_{t}^{\varepsilon},\hat{Y}_{t}^{\varepsilon})\in L^2 \times L^2$.
Specifically,
we split the interval $[0,T]$ into some subintervals of size $\delta>0$.
%and divide $[0,T]$ into intervals of size $\delta$,
%where $\delta$ is a fixed positive number.
With the initial value $\hat{Y}_{0}^{\varepsilon}=Y^{\varepsilon}_{0}=y$,
for any $t\in[k\delta,\min((k+1)\delta,T))$, $k \in \mathbb{N}$,
we construct the process $\hat{Y}_{t}^{\varepsilon}$ as follows:
%with initial value $\hat{Y}_{0}^{\varepsilon}=Y^{\varepsilon}_{0}=y$,
%and
\begin{align} \label{AuxiliaryPro Y 01}
\hat{Y}_{t}^{\varepsilon}= Y_{k\delta}^{\varepsilon}+\frac{1}{\varepsilon}\int_{k\delta}^{t}A\hat{Y}_{s}^{\varepsilon}ds+\frac{1}{\varepsilon}\int_{k\delta}^{t}
g(X_{k\delta}^{\varepsilon},\hat{Y}_{s}^{\varepsilon})ds+\frac{1}{\sqrt{\varepsilon}}\int_{k\delta}^{t}dW^{Q_{2}}_s,
\end{align}
where $(X^{\varepsilon}_{s}, Y^{\varepsilon}_{s})$
is the solution to the system \eref{main equation}. % at time $k\delta$.
Then, for any $t\in[0,T]$, we construct the process $\hat{X}_{t}^{\varepsilon}$
as follows:
%which is determined by the following equation:
%Also, we define the process $\hat{X}_{t}^{\varepsilon}$ by integral
\begin{align} \label{AuxiliaryPro X 01}
\hat{X}_{t}^{\varepsilon}=x+\int_{0}^{t}A\hat{X}_{s}^{\varepsilon}ds+\int_{0}^{t}B(X_{s(\delta)}^{\varepsilon})ds+\int_{0}^{t}
f(X_{s(\delta)}^{\varepsilon},\hat{Y}_{s}^{\varepsilon})ds+W^{Q_{1}}_t,
\end{align}
where $s(\delta)=[\frac{s}{\delta}]\delta$
is the nearest breakpoint proceeding $s$.
Note that for any $t \in [k\delta, (k+1)\delta)$,
the fast component $\hat{Y}_{t}^{\varepsilon}$
does not depend on the slow component $\hat{X}_{t}^{\varepsilon}$.
%but only on the value of $X_{t}^{\vare}$ at the first point of the interval.
The following result gives a control of the auxiliary process
$(\hat{X}_{t}^{\varepsilon},\hat{Y}_{t}^{\varepsilon})$.

\begin{lemma} \label{MDY}
Under the conditions \ref{A1}-\ref{A3},
for any $x,y\in L^2$, $p \geq 2$ and $T>0$, there exists a constant $C_{p,T}>0$ such that
$$
\sup_{\vare\in(0,1)}\sup_{t\in[0,T]}
\mathbb{E} \| \hat{Y}_{t}^{\vare} \|^{2p}\leq C_{p,T}(1+ \|x \|^{2p}+ \|y\|^{2p}).
$$
In addition, for any $x\in H^{\alpha}$, $y\in L^2$, $p \geq 2$ and $T>0$, we have
%there exists a constant $C_{p,T}>0$ such that
\begin{align} \label{hatXHolderalpha}
\sup_{\varepsilon\in(0,1)}
\mathbb{E}\big(\sup_{t\in[0,T]}|\hat{X}_{t}^{\varepsilon}|_{\alpha}^{2p}\big)
\leq C_{p,T}(1+|x|^{2p}_{\alpha}+ \|y\|^{2p}).
\end{align}
\end{lemma}

From the construction of $(\hat{X}_{t}^{\varepsilon}, \hat{Y}_{t}^{\varepsilon})$,
since the proof of Lemma \ref{MDY} can be carried out in the same way as in the proof of
Lemmas \ref{PMY} and \ref{SOX}, we omit the details. 
%proof of Lemma \ref{MDY}.

%we can easily obtain the following estimates which will be used below. Because the proof almost follows the steps in Lemmas \ref{PMY} and \ref{SOX}, we omit the proof here.

%\begin{proof}
%For $t\in[0,T]$ with $t\in[k\delta,(k+1)\delta)$, by It\^{o}'s formula we have
%\begin{eqnarray}
%|\hat{Y}_{t}^{\vare}|^{2p}\leq\!\!\!\!\!\!\!\!&&|Y_{k\delta}^{\vare}|^{2p}+\frac{2p}{\vare}\int_{k\delta}^{t}|\hat{Y}_{s}^{\vare}|^{2p-2}\langle %A\hat{Y}_{s}^{\vare},\hat{Y}_{s}^{\vare}\rangle ds   \nonumber\\
%\!\!\!\!\!\!\!\!&&+\frac{2p}{\vare}\int_{k\delta}^{t}|\hat{Y}_{s}^{\vare}|^{2p-2}\langle g(X_{k\delta}^{\vare},\hat{Y}_{s}^{\vare}),\hat{Y}_{s}^{\vare}\rangle ds  %+\frac{2}{\sqrt{\vare}}\int_{k\delta}^{t}|\hat{Y}_{s}^{\vare}|^{2p-2}\langle\hat{Y}_{s}^{\vare},dW^{Q_{2}}(s)\rangle     \nonumber\\
%\!\!\!\!\!\!\!\!&&+\frac{p}{\vare}\int_{k\delta}^{t}|\hat{Y}_{s}^{\vare}|^{2p-2}\text{Tr}Q_{2}ds
%+\frac{2p(p-1)}{\vare}\int_{k\delta}^{t}|\hat{Y}_{s}^{\vare}|^{2p-2}\text{Tr}Q_{2}ds. \nonumber
%\end{eqnarray}
%Similar to the proof in Lemma \ref{PMY}, we have
%\begin{align*}
%\mathbb{E}|\hat{Y}_{t}^{\varepsilon}|^{2p}\leq\mathbb{E}|Y_{k\delta}^{\varepsilon}|^{2p}-\frac{p(\lambda_{1}
%-\beta_{1})}{\varepsilon}\int_{k\delta}^{t}\mathbb{E}|\hat{Y}_{s}^{\varepsilon}|^{2p}ds+\frac{C_{p}}{\varepsilon}(t-k\delta).
%\end{align*}
%Thanks to Gronwall's inequality, we finally get
%\begin{eqnarray*}
%\sup_{\vare\in(0,1)}\sup_{t\in[0,T]}\mathbb{E}|\hat{Y}_{t}^{\varepsilon}|^{2p}\leq C_{p}.
%\end{eqnarray*}
%\hspace{\fill}$\square$
%\end{proof}

%\vskip 0.3cm
We now give a control of $Y_{t}^{\varepsilon}-\hat{Y}_{t}^{\varepsilon}$.
%establish the the error of $Y_{t}^{\varepsilon}-\hat{Y}_{t}^{\varepsilon}$, and furthermore the error of $X^{\varepsilon}_t-\hat{X}_{t}^{\varepsilon}$ .
%Now we will establish convergence of the auxiliary processes $\hat{Y}_{t}^{\varepsilon}$ to the fast solution process $Y_{t}^{\varepsilon}$ and $\hat{X}_{t}^{\varepsilon}$ to the slow solution %process $X_{t}^{\varepsilon}$ respectively.
\begin{lemma} \label{DEY}
Under the conditions \ref{A1}-\ref{A3},
for any $x\in H^{\alpha}, y\in L^2$, $p\geq2$, $T>0$ and $\vare\in(0,1)$, there exists a constant $C_{p,T}>0$ such that
$$
\sup_{0\leq t\leq T}\mathbb{E} \|Y_{t}^{\varepsilon}-\hat{Y}_{t}^{\varepsilon} \|^{2p}
\leq C_{p,T}(1+|x|^{4p}_{\alpha}+ \|y\|^{4p})\frac{\delta^{p+1}}{\varepsilon}.
$$
\end{lemma}

\begin{proof}
For $t\in[0,T]$ with $t\in[k\delta,(k+1)\delta)$, by It\^{o}'s formula and Lemma \ref{COX},
similarly to \eqref{ItoFormu 01}, we have
\begin{eqnarray}
\frac{d}{dt}\mathbb{E} \| Y_{t}^{\varepsilon}-\hat{Y}_{t}^{\varepsilon} \|^{2p}
=\!\!\!\!\!\!\!\!&&
\frac{2p}{\varepsilon}
\mathbb{E}\Big[ \| Y_{t}^{\varepsilon}
-\hat{Y}_{t}^{\varepsilon} \|^{2p-2}
(-|Y_{t}^{\varepsilon} -\hat{Y}_{t}^{\varepsilon}|_1^2)
%\langle A(Y_{t}^{\varepsilon}-\hat{Y}_{t}^{\varepsilon}),
%(Y_{t}^{\varepsilon}-\hat{Y}_{t}^{\varepsilon})\rangle
\Big]\nonumber\\
\!\!\!\!\!\!\!\!&&
+\frac{2p}{\varepsilon}\mathbb{E}\Big[ \| Y_{t}^{\varepsilon}-\hat{Y}_{t}^{\varepsilon} \|^{2p-2}\langle
g(X_{t}^{\varepsilon},Y_{t}^{\varepsilon})-g(X_{k\delta}^{\varepsilon},\hat{Y}_{t}^{\varepsilon}),(Y_{t}^{\varepsilon}-\hat{Y}_{t}^{\varepsilon})\rangle\Big] \nonumber\\
\leq\!\!\!\!\!\!\!\!&&
-\frac{2p}{\varepsilon}(\lambda_{1}-L_{g})
\mathbb{E}\Big[ \| Y_{t}^{\varepsilon}-\hat{Y}_{t}^{\varepsilon} \|^{2p}\Big]\nonumber\\
\!\!\!\!\!\!\!\!&&
+\frac{p}{\varepsilon}(\lambda_{1}-L_{g})
\mathbb{E}\Big[ \|Y_{t}^{\varepsilon}-\hat{Y}_{t}^{\varepsilon} \|^{2p}\Big]
+\frac{C_{p}}{\varepsilon}\mathbb{E}\Big[ \| X_{t}^{\varepsilon}-X_{k\delta}^{\varepsilon} \|^{2p}\Big]
\nonumber\\
\leq\!\!\!\!\!\!\!\!&&
-\frac{p}{\varepsilon}(\lambda_{1}-L_{g})
\mathbb{E} \|Y_{t}^{\varepsilon}-\hat{Y}_{t}^{\varepsilon} \|^{2p}
+C_{p,T}(1+|x|^{4p}_{\alpha}+ \| y \|^{4p})\frac{(t-k\delta)^{p}}{\varepsilon}.
\nonumber
\end{eqnarray}
Gronwall's inequality yields that
\begin{eqnarray*}
\mathbb{E} \| Y_{t}^{\vare}-\hat{Y}_{t}^{\vare} \|^{2p}
\leq\!\!\!\!\!\!\!\!&&
\frac{C_{p,T}}{\varepsilon}(1+|x|^{4p}_{\alpha}+ \|y \|^{4p})
\int^{t}_{k\delta}e^{-\frac{p}{\varepsilon}(\lambda_{1}-L_{g})(t-s)}(s-k\delta)^p ds\nonumber\\
\leq\!\!\!\!\!\!\!\!&&
C_{p,T}(1+|x|^{4p}_{\alpha}+ \| y \|^{4p})\frac{\delta^{p+1}}{\varepsilon},
\end{eqnarray*}
which ends the proof. %\hspace{\fill}$\square$
\end{proof}

\begin{lemma} \label{DEX}
Under the conditions \ref{A1}-\ref{A3},
for any $x\in H^{\alpha}, y\in L^2$, $p\geq 2$, $T>0$ and $\vare\in(0,1)$,
there exists a constant $C_{p,T}>0$ such that
\begin{align*}
\mathbb{E}\Big(\sup_{0\leq t\leq T} \|X_{t}^{\vare}-\hat{X}_{t}^{\vare} \|^{2p}\Big)
\leq C_{p,T}(\delta^{p}+\frac{\delta^{p+1}}{\vare})(1+|x|^{6p}_{\alpha}+ \|y \|^{6p}).
\end{align*}
\end{lemma}

\begin{proof}
%Recall that
%\begin{align*}
%X^{\vare}_t=e^{tA}x+\int^t_0e^{(t-s)A}B(X^{\vare}_s)ds+\int^t_0e^{(t-s)A}f(X^{\vare}_s, Y^{\vare}_s)ds+\int^t_0 e^{(t-s)A}dW^{Q_1}_s
%\end{align*}
%and
%\begin{align*}
%\hat{X}^{\vare}_t=e^{tA}x+\int^t_0e^{(t-s)A}B(X^{\vare}_{s(\delta)})ds+\int^t_0e^{(t-s)A}f(X^{\vare}_{s(\delta)}, \hat{Y}^{\vare}_s)ds+\int^t_0 e^{(t-s)A}dW^{Q_1}_s.
%\end{align*}
In view of  \eqref{mild solution} and  \eqref{AuxiliaryPro X 01}, we write
\begin{align*}
X^{\vare}_t-\hat{X}^{\vare}_t=\int^t_0e^{(t-s)A}\big[B(X^{\vare}_s)-B(X^{\vare}_{s(\delta)})\big]ds
+\int^t_0e^{(t-s)A}\big[f(X^{\vare}_s, Y^{\vare}_s)-f(X^{\vare}_{s(\delta)}, \hat{Y}^{\vare}_s)\big]ds.
\end{align*}
Using $\eref{PSG}$, condition \ref{A1} and Lemma \ref{Property B2}, we get
\begin{eqnarray}
\| X_{t}^{\vare}-\hat{X}_{t}^{\vare} \|^{2p}
\leq\!\!\!\!\!\!\!\!&&C_{p}\Big\{\int_{0}^{t}\big[1+(t-s)^{-\frac{1}{2}}\big]\big|B(X^{\vare}_s)-B(X^{\vare}_{s(\delta)})\big|_{-1}ds\Big\}^{2p}\nonumber\\
\!\!\!\!\!\!\!\!&&
+C_{p}\Big\{\int_{0}^{t}\big(\big\|X^{\vare}_s-X^{\vare}_{s(\delta)}\big\|
 + \big\|Y^{\vare}_s-\hat{Y}^{\vare}_s \big\| \big)ds\Big\}^{2p}\nonumber\\
\leq\!\!\!\!\!\!\!\!&&C_{p}\Big\{\int_{0}^{t}\big[1+(t-s)^{-\frac{1}{2}}\big]\big\|X^{\vare}_s-X^{\vare}_{s(\delta)}\big\|\big(|X^{\vare}_s|_{1}
+|X^{\vare}_{s(\delta)}|_{1}\big)ds\Big\}^{2p}\nonumber\\
\!\!\!\!\!\!\!\!&&
+C_{p,T}\int_{0}^{t}\big(\big \|X^{\vare}_s-X^{\vare}_{s(\delta)}\big \|^{2p}
+\big \|Y^{\vare}_s-\hat{Y}^{\vare}_s\big \|^{2p}\big)ds\nonumber\\
\leq\!\!\!\!\!\!\!\!&&C_{p}\Big\{\int_{0}^{t}\big[1+(t-s)^{-\frac{1}{2}}\big]^{\frac{2p}{2p-1}}ds\Big\}^{2p-1}
\cdot\Big\{\int_{0}^{t}\big \|X^{\vare}_s-X^{\vare}_{s(\delta)}\big\|^{4p}ds\Big\}^{\frac{1}{2}}
\nonumber\\
\!\!\!\!\!\!\!\!&&
\ \   \times\Big(\int_{0}^{t}\big(|X^{\vare}_s|_{1}+|X^{\vare}_{s(\delta)}|_{1}\big)^{4p}ds\Big)^{\frac{1}{2}}   \nonumber\\
\!\!\!\!\!\!\!\!&&
+ C_{p,T}\int_{0}^{t}\big(\big\|X^{\vare}_s-X^{\vare}_{s(\delta)}\big\|^{2p}
+ \big\|Y^{\vare}_s-\hat{Y}^{\vare}_s\big\|^{2p}\big)ds\nonumber
\end{eqnarray}
According to Lemmas \ref{SOX}, \ref{COX} and \ref{DEY}, we obtain
\begin{eqnarray}
\mathbb{E}\Big(\sup_{0\leq t\leq T} \|X_{t}^{\vare}-\hat{X}_{t}^{\vare} \|^{2p}\Big)
\leq\!\!\!\!\!\!\!\!&&
C_{p,T}\Big(\int_{0}^{T}\mathbb{E} \big \|X^{\vare}_s-X^{\vare}_{s(\delta)}\big\|^{4p}ds\Big)
^{\frac{1}{2}}  \nonumber\\
\!\!\!\!\!\!\!\!&& \ \   \times\Big(\int_{0}^{T}\big(\mathbb{E}|X^{\vare}_s|^{4p}_{1}+\mathbb{E}|X^{\vare}_{s(\delta)}|^{4p}_{1}\big)ds\Big)^{\frac{1}{2}}   \nonumber\\
\!\!\!\!\!\!\!\!&&
+C_{p,T}\int_{0}^{T}\mathbb{E} \big\|X^{\vare}_s-X^{\vare}_{s(\delta)}\big\|^{2p}
+\mathbb{E} \big\|Y^{\vare}_s-\hat{Y}^{\vare}_s\big\|^{2p}ds\nonumber\\
\leq\!\!\!\!\!\!\!\!&&
C_{p,T}(\delta^{p}+\frac{\delta^{p+1}}{\vare})(1+|x|^{6p}_{\alpha} + \|y\|^{6p}).  \nonumber
\end{eqnarray}
The proof is complete.  %\hspace{\fill}$\square$
\end{proof}

\subsection{The averaged equation}
For any fixed $x\in L^2$, we consider the following frozen equation
associated with the fast component:
%We consider the frozen equation associate to fast motion for fixed slow component $x\in L^2$.
 \begin{equation}\left\{\begin{array}{l}\label{FEQ}
\displaystyle
\frac{\partial Y_{t}(\xi)}{\partial t}=AY_{t}(\xi)+g(x,Y_{t}(\xi))
 + \frac{\partial W^{Q_2}}{\partial t}(t,\xi),  \   Y_{0}(\xi)=y, \\
Y_{t}(0)=Y_{t}(1)=0, \  \  t\in[0,\infty). 
\end{array}\right.
\end{equation}
Since $g(x,\cdot)$ is Lipshcitz continuous,
it is easy to prove that for any fixed $x, y \in L^2$,  %$y\in L^2$,
the equation $\eref{FEQ}$ has a unique mild solution denoted by $Y_{t}^{x,y}$.
For any $x \in L^2$,
let $P^x_t$ be the transition semigroup of $Y_{t}^{x,y}$,
that is, for any bounded measurable function $\varphi$ on $L^2$ and $t \geq 0$,
\begin{align*}
P^x_t \varphi(y)= \mathbb{E} \varphi(Y_{t}^{x,y}), \quad y \in L^2.
\end{align*}
The asymptotic behavior of $P^x_t$ has been studied in many literatures.
The following result shows the existence and uniqueness of the invariant measure and
gives the exponential convergence to the equilibrium (see \cite[Theorem 3.5]{CF}).
%now we state the ergodicity for \eref{FEQ} (see \cite[Theorem 3.5]{CF}).
\begin{proposition}\label{ergodicity}
For any $x, y\in L^2$,
there exists a unique invariant measure $\mu^x$ for $\eref{FEQ}$. Moreover, there exists $C>0$ such that for any bounded measurable function $\varphi: L^2 \to \mathbb{R}$,
$$
\left| P^x_t\varphi(y)-\int_{L^2}\varphi(z)\mu^x(dz)\right|
\leq C(1+ \|x \| + \|y \|)e^{-\frac{(\lambda_1-L_g)t}{2}}(t\wedge1)^{-1/2}|\varphi|_{\infty},
$$
where $|\varphi|_{\infty}=\sup_{x\in L^2 } |\varphi(x)|$.
\end{proposition}

%\vskip 0.3cm
Furthermore,  we have the following result, whose proof can refer to \cite[Remark 3.6]{{CF}}.
\begin{proposition}\label{Rem 4.1}
For any $x, y\in L^2$,
there exists $C>0$ such that for any Lipschitz function $\varphi: L^2 \to \mathbb{R}$,
$$
\left| P^x_t\varphi(y)-\int_{L^2}\varphi(z)\mu^x(dz)\right|
\leq C(1+ \|x\|+ \|y \|)e^{-\frac{(\lambda_1-L_g)t}{2}}|\varphi|_{Lip},
$$
where $|\varphi|_{Lip}=\sup_{x,y\in L^2, x \neq y}\frac{|\varphi(x)-\varphi(y)|}{\|x-y\|}$.
\end{proposition}

 %In particular, for any $\varphi\in \text{Lip}(H)$, we obtain
 %\begin{eqnarray}
% \Big|\mathbb{E}\varphi(Y_{t}^{x,y})-\int_{H}\varphi(x,z)\mu^{x}(dz)\Big|^{2}
% =\!\!\!\!\!\!\!\!&&\Big|\int_{H}\mathbb{E}\big(\varphi(Y_{t}^{x,y})-\varphi(Y_{t}^{x,z})\big)\mu^{x}(dz)\Big|^{2}  \nonumber \\
% \leq\!\!\!\!\!\!\!\!&&\int_{H}\mathbb{E}\big|Y_{t}^{x,y}-Y_{t}^{x,z}\big|^{2}\mu^{x}(dz)  \nonumber \\
% \leq\!\!\!\!\!\!\!\!&&e^{-(\lambda_{1}-L_{g})t}\int_{H}\big|y-z\big|^{2}\mu^{x}(dz)   \nonumber \\
% \leq\!\!\!\!\!\!\!\!&&Ce^{-(\lambda_{1}-L_{g})t}(1+|x|^{2}+|y|^{2})   \label{AIM}
 % \end{eqnarray}

% \vskip 0.5cm

In the sequel we shall prove that the slow component $X_{t}^{\varepsilon}$
in the system \eqref{main equation}
converges strongly to $\bar{X}_{t}$,
which is the solution of the averaged equation:
%the averaging principle occurs in the sense that the slow component process $X_{t}^{\varepsilon}$ converges strongly to the solution $\bar{X}_{t}$ of the averaged equation
\begin{equation}
\left\{\begin{array}{l}
\displaystyle d\bar{X}_{t}=A \bar{X}_{t}dt+B(\bar{X}_{t})dt+\bar{f}(\bar{X}_{t})dt+dW^{Q_{1}}_t,\\
\bar{X}_{0}=x.\end{array}\right. \label{3.1}
\end{equation}
where
\begin{align*}
\bar{f}(x)=\int_{L^2}f(x,y)\mu^{x}(dy), \quad x\in L^2.
\end{align*}
%where $\mu^{x}$ is the unique invariant measure for the equation $\eref{FEQ}$.

%\vskip 0.3cm

The following result gives a control of $|\hat{X}_{t}^{\vare}-\bar{X}_{t}|$.
%implies the error between auxiliary process $\hat{X}_{t}^{\vare}$ and the averaging solution $\bar{X}_{t}$.
\begin{lemma} \label{ESX}
Under the conditions \ref{A1}-\ref{A3},
for any $x\in H^{\alpha}$, $y\in L^2$, $p\geq 1$, $T>0$ and $\vare\in(0,1)$,
there exists a constant $C_{p,T}>0$ such that
\begin{align*}
\mathbb{E}\sup_{0\leq t\leq T}
\|\hat{X}_{t}^{\vare}-\bar{X}_{t} \|^{2p}
\leq C_{p,T}(1+|x|^{6p}_{\alpha}+ \|y \|^{6p})
\Big(\frac{1}{-\log\vare}\Big)^{\frac{1}{4p}}.
\end{align*}
\end{lemma}

\begin{proof}
From \eqref{AuxiliaryPro X 01} and \eqref{3.1}, we have
%After simple calculations, we have
\begin{eqnarray}
\hat{X}_{t}^{\vare}-\bar{X}_{t}
%\!\!\!\!\!\!\!\!&&\int_{0}^{t}e^{(t-s)A}\left[B(X_{s(\delta)}^{\vare})-B(\bar{X}_{s})\right]ds
%+\int_{0}^{t}e^{(t-s)A}\left[f(X_{s(\delta)}^{\vare},\hat{Y}_{s}^{\vare})-\bar{f}(\bar{X}_{s})\right]ds    \nonumber\\
=\!\!\!\!\!\!\!\!&&
\int_{0}^{t}e^{(t-s)A}\left[B(X_{s(\delta)}^{\vare})-B(X_{s}^{\vare})\right]ds
+\int_{0}^{t}e^{(t-s)A}\left[B(X_{s}^{\vare})-B(\hat{X}_{s}^{\vare})\right]ds   \nonumber\\
\!\!\!\!\!\!\!\!&&
+\int_{0}^{t}e^{(t-s)A}\left[B(\hat{X}_{s}^{\vare})-B(\bar{X}_{s})\right]ds
+\int_{0}^{t}e^{(t-s)A}
\left[f(X_{s(\delta)}^{\vare},\hat{Y}_{s}^{\vare})-\bar{f}(X_{s}^{\vare})\right]ds \nonumber\\
\!\!\!\!\!\!\!\!&&
+\int_{0}^{t}e^{(t-s)A}\left[\bar{f}(X_{s}^{\vare})-\bar{f}(\hat{X}_{s}^{\vare})\right]ds
+\int_{0}^{t}e^{(t-s)A}\left[\bar{f}(\hat{X}_{s}^{\vare})-\bar{f}(\bar{X}_{s})\right]ds   \nonumber\\
:=\!\!\!\!\!\!\!\!&& \sum_{k=1}^{6}J_{k}(t).   \nonumber
\end{eqnarray}
For $J_{1}(t)$, in the same way as in the proof of Lemma \ref{DEX}, we get
%just as the techniques in the proof of Lemma \ref{DEX}, we have
\begin{eqnarray}  \label{J1}
\mathbb{E}\sup_{0\leq t\leq T} \| J_{1}(t) \|^{2p}
\leq\!\!\!\!\!\!\!\!&&
C_{p,T}\left[\int_{0}^{T}
\mathbb{E}\big \|X^{\vare}_s-X^{\vare}_{s(\delta)}\big\|^{4p}ds\right]^{\frac{1}{2}}  \left[\int_{0}^{T}\left(\mathbb{E}|X^{\vare}_s|_{1}
+\mathbb{E}|X^{\vare}_{s(\delta)}|_{1}\right)^{4p}ds\right]^{\frac{1}{2}} \nonumber\\
\leq\!\!\!\!\!\!\!\!&&C_{p,T}\delta^{p}(1+|x|^{6p}_{\alpha} + \| y \|^{6p}).
\end{eqnarray}
For $J_{2}(t)$, using Lemmas \ref{SOX}, \ref{MDY} and \ref{DEX} gives
\begin{eqnarray}    \label{J2}
\mathbb{E}\sup_{0\leq t\leq T} \| J_{2}(t) \|^{2p}
\leq\!\!\!\!\!\!\!\!&&
C_{p,T}\left[\int_{0}^{T}\mathbb{E}\big \|X^{\vare}_s-\hat{X}^{\vare}_{s}\big \|^{4p}ds\right]^{\frac{1}{2}}  \left[\int_{0}^{T}\left(\mathbb{E}|X^{\vare}_s|^{4p}_{1}
+\mathbb{E}|\hat{X}^{\vare}_{s}|^{4p}_{1}\right)ds\right]^{\frac{1}{2}} \nonumber\\
\leq\!\!\!\!\!\!\!\!&&
C_{p,T}(\delta^{p}+\frac{\delta^{p+\frac{1}{2}}}{\sqrt{\vare}})(1+|x|^{6p}_{\alpha}+ \|y \|^{6p}).
\end{eqnarray}
For $J_{3}(t)$, according to $\eref{PSG}$ and Lemma \ref{Property B2}, we have
\begin{eqnarray}  \label{J3INE}
\sup_{0\leq t\leq T} \|J_{3}(t) \|^{2p}
\leq\!\!\!\!\!\!\!\!&&
C_{p}\left\{\sup_{0\leq t\leq T}
\int_{0}^{t}\left[1+(t-s)^{-\frac{1}{2}}\right]\big|B(\hat{X}^{\vare}_s)-B(\bar{X}_{s})\big|_{-1}ds\right\}^{2p}   \nonumber\\
\leq\!\!\!\!\!\!\!\!&&C_{p}\left\{\sup_{0\leq t\leq T}
\int_{0}^{t}\left[1+(t-s)^{-\frac{1}{2}}\right]\big \|\hat{X}^{\vare}_s-\bar{X}_{s}\big\|
\left(|\hat{X}^{\vare}_s|_{1}
+|\bar{X}_{s}|_{1}\right)ds\right\}^{2p}.
\end{eqnarray}
In order to control $J_3(t)$, we make a use of the skill of stopping times.
%deal with the above estimate, we will use the skill of stopping times, i.e.,
For any fixed $n\geq1$ and $\vare>0$, define the stopping time:
\begin{align} \label{StoppingTime 01}
\tau_{n}^{\vare}=\inf\left\{t>0: |\hat{X}^{\vare}_t|_{1}+|\bar{X}_{t}|_{1}>n\right\}.
\end{align}
It follows from  \eqref{J3INE} and \eqref{StoppingTime 01} that
\begin{eqnarray}   \label{J31}
\mathbb{E}\sup_{0\leq t\leq T\wedge\tau_{n}^{\vare}} \| J_{3}(t) \|^{2p}
\leq\!\!\!\!\!\!\!\!&&
C_{p}\mathbb{E}\Big(\sup_{0\leq t\leq T\wedge\tau_{n}^{\vare}}
\int_{0}^{t}(1+(t-r)^{-\frac{1}{2}})\big\|\hat{X}^{\vare}_r-\bar{X}_{r}\big\|
\big(|\hat{X}^{\vare}_r|_{1}
+|\bar{X}_{r}|_{1}\big)dr\Big)^{2p}    \nonumber\\
\leq\!\!\!\!\!\!\!\!&&
C_{p}n^{2p}\mathbb{E}\Big(\sup_{0\leq t\leq T\wedge\tau_{n}^{\vare}}
\int_{0}^{t}(1+(t-r)^{-\frac{1}{2}})
\big\|\hat{X}^{\vare}_r-\bar{X}_{r}\big\| dr\Big)^{2p}    \nonumber\\
\leq\!\!\!\!\!\!\!\!&&
C_{p}n^{2p}\Big(\sup_{0\leq t\leq T}\int_{0}^{t}(1+(t-r)^{-\frac{1}{2}})^{\frac{2p}{2p-1}}dr\Big)^{2p-1}
\mathbb{E}
\int_{0}^{T\wedge\tau_{n}^{\vare}}\big\|\hat{X}^{\vare}_r-\bar{X}_{r}\big\|^{2p}dr  \nonumber\\
\leq\!\!\!\!\!\!\!\!&&
C_{p,T}n^{2p}
\int_{0}^{T}
  \mathbb{E}\sup_{0\leq r\leq s\wedge\tau_{n}^{\vare}}
  \big\|\hat{X}^{\vare}_{r}-\bar{X}_{r}\big\|^{2p}ds.
\end{eqnarray}
For $J_{5}(t)$, using the contractive property of the semigroup $e^{tA}$, $t \geq 0$,
Lipschitz continuity of $\bar{f}$, and Lemma \ref{DEX}, we obtain
\begin{align}  \label{J5}
\mathbb{E}\sup_{0\leq t\leq T}|J_{5}(t)|^{2p}
\leq C_{p,T}\mathbb{E}
\int_{0}^{T} \| X_{s}^{\varepsilon}-\hat{X}_{s}\|^{2p}ds   
\leq
C_{p,T}(\delta^{p}+\frac{\delta^{p+1}}{\varepsilon})(1+|x|^{2p}_{\alpha}+\|y\|^{2p}).
\end{align}
For $J_{6}(t)$, similarly to the estimate of $J_{5}(t)$, we get
\begin{eqnarray}  \label{J6}
\mathbb{E}\sup_{0\leq t\leq {T\wedge \tau_n^\varepsilon}}|J_{6}(t)|^{2p}\leq
C_{p,T}\int_{0}^T\mathbb{E}\sup_{0\leq r\leq s\wedge\tau_{n}^{\vare}}\big\|\hat{X}^{\vare}_{r}-\bar{X}_{r}\big\|^{2p}ds.
\end{eqnarray}
For $J_{4}(t)$, set $n_{t}=[\frac{t}{\delta}]$, where $t\in[0,T)$ and $\delta>0$.
%which is given in the construction of $X$ in
%Now we are going to estimate $J_{4}(t)$. For any $t\in[0,T)$, set $n_{t}=[\frac{t}{\delta}]$,
%we have $t\in[n_{t}\delta,(n_{t}+1)\delta\wedge T)$. Therefore, we have representation in the form
We write
\begin{align*}
J_{4}(t)=J_{4}^{1}(t)+J_{4}^{2}(t)+J_{4}^{3}(t),
\end{align*}
where
\begin{align*}
J_{4}^{1}(t)=\sum_{k=0}^{n_{t}-1}
\int_{k\delta}^{(k+1)\delta}e^{(t-s)A}\left[f(X_{k\delta}^{\varepsilon},\hat{Y}_{s}^{\varepsilon})-\bar{f}(X_{k\delta}^{\varepsilon})\right]ds,
\end{align*}
\begin{align*}
J_{4}^{2}(t)=\sum_{k=0}^{n_{t}-1}
\int_{k\delta}^{(k+1)\delta}e^{(t-s)A}\left[\bar{f}(X_{k\delta}^{\varepsilon})-\bar{f}(X_{s}^{\varepsilon})\right]ds,
\end{align*}
\begin{align*}
J_{4}^{3}(t)=
\int_{n_{t}\delta}^{t}e^{(t-s)A}\left[f(X_{n_{t}\delta}^{\varepsilon},\hat{Y}_{s}^{\varepsilon})-\bar{f}(X_{s}^{\varepsilon})\right]ds.
\end{align*}
For $J_{4}^{2}(t)$, we have
\begin{align}   \label{J42}
\mathbb{E}\sup_{0\leq t\leq T} \|J_{4}^{2}(t) \|^{2p}
\leq C_{p,T}\int_{0}^{T}\mathbb{E} \|X_{s(\delta)}^{\varepsilon}-X_{s}^{\varepsilon} \|^{2p}ds
\leq
C_{p,T}\delta^{p}(1+|x|^{2p}_{\alpha}+ \|y\|^{2p}).
\end{align}
For $J_{4}^{3}(t)$, it follows from Lemmas \ref{PMY} and \ref{MDY} that
\begin{eqnarray}  \label{J43}
\mathbb{E}\sup_{0\leq t\leq T} \| J_{4}^{3}(t) \|^{2p}
\leq\!\!\!\!\!\!\!\!&&
C_{p}\delta^{2p-1}\mathbb{E}\left[\sup_{0\leq t\leq T}\int_{n_{t}\delta}^{t}
\left(1+ \| X_{n_{t}\delta}^{\varepsilon} \|^{2p}
+ \| \hat{Y}_{s}^{\varepsilon} \|^{2p}
+ \| X_{s}^{\varepsilon} \|^{2p}\right)ds\right]  \nonumber\\
%\leq\!\!\!\!\!\!\!\!&&
%C_{p}\delta^{2p-1}\mathbb{E}\int_{0}^{T}\left(1+|X_{n_{t}\delta}^{\varepsilon}|_{L^2}^{2p}
% +|\hat{Y}_{s}^{\varepsilon}|_{L^2}^{2p}+|X_{s}^{\varepsilon}|_{L^2}^{2p}\right)ds  \nonumber\\
\leq\!\!\!\!\!\!\!\!&& C_{p,T}\delta^{2p-1}(1+ \| x \|^{2p}+ \|y \|^{2p}).
\end{eqnarray}
For $J_{4}^{1}(t)$,
from the construction of $\hat{Y}_{t}^{\varepsilon}$,
we obtain that, for any $k$ and $s\in[0,\delta)$,
% and a time shift transformation, for any fixed $k$ and $s\in[0,\delta)$,  we have the equalities
\begin{align*}
\hat{Y}_{s+k\delta}^{\varepsilon}
%=\!\!\!\!\!\!\!\!&&Y_{k\delta}^{\varepsilon}+\frac{1}{\varepsilon}\int_{k\delta}^{k\delta+s}A\hat{Y}_{r}^{\varepsilon}dr
%+\frac{1}{\varepsilon}\int_{k\delta}^{k\delta+s}g(X_{k\delta}^{\varepsilon},\hat{Y}_{r}^{\varepsilon})dr
%+\frac{1}{\sqrt{\varepsilon}}\int_{k\delta}^{k\delta+s}dW^{Q_{2}}(r)   \nonumber\\
= Y_{k\delta}^{\varepsilon}+\frac{1}{\varepsilon}\int_{0}^{s}A\hat{Y}_{r+k\delta}^{\varepsilon}dr
+\frac{1}{\varepsilon}\int_{0}^{s}g(X_{k\delta}^{\varepsilon},\hat{Y}_{r+k\delta}^{\varepsilon})dr
+\frac{1}{\sqrt{\varepsilon}}\int_{0}^{s}d\tilde{W}^{Q_{2}}(r),
\end{align*}
where $\tilde{W}^{Q_{2}}(t):=W^{Q_{2}}(t+k\delta)-W^{Q_{2}}(k\delta)$
is the shift version of $W^{Q_{2}}(t)$.
%and hence they have the same distribution.
Let $\bar{W}^{{Q_2}}(t)$ be a ${Q_2}$-Wiener process
%defined on the same stochastic basis
which is independent of $W^{Q_{1}}(t)$ and $W^{Q_{2}}(t)$.
Denote by $\bar{\bar{W}}^{{Q_2}}(t)=\sqrt{\vare}\bar{W}^{{Q_2}}(\frac{t}{\vare})$.
%where $\bar{\bar{W}}^{{Q_2}}(t)=\sqrt{\vare}\bar{W}^{{Q_2}}(\frac{t}{\vare})$ is the scaled version of $\bar{W}^{{Q_2}}(t)$.
We construct a process $Y^{X_{k\delta}^{\varepsilon},Y_{k\delta}^{\varepsilon}}$ by means of
\begin{eqnarray}
Y_{\frac{s}{\varepsilon}}^{X_{k\delta}^{\varepsilon},Y_{k\delta}^{\varepsilon}}=\!\!\!\!\!\!\!\!&&Y_{k\delta}^{\varepsilon}
+\int_{0}^{\frac{s}{\varepsilon}}AY_{r}^{X_{k\delta}^{\varepsilon},Y_{k\delta}^{\varepsilon}}dr
+\int_{0}^{\frac{s}{\varepsilon}}g(X_{k\delta}^{\varepsilon},Y_{r}^{X_{k\delta}^{\varepsilon},Y_{k\delta}^{\varepsilon}})dr
+\int_{0}^{\frac{s}{\varepsilon}}d\bar{W}^{{Q_2}}(r)   \nonumber\\
=\!\!\!\!\!\!\!\!&&Y_{k\delta}^{\varepsilon}
+\frac{1}{\varepsilon}\int_{0}^{s}AY_{\frac{r}{\varepsilon}}^{X_{k\delta}^{\varepsilon},Y_{k\delta}^{\varepsilon}}dr
+\frac{1}{\varepsilon}\int_{0}^{s}g(X_{k\delta}^{\varepsilon},Y_{\frac{r}{\varepsilon}}^{X_{k\delta}^{\varepsilon},Y_{k\delta}^{\varepsilon}})dr
+\frac{1}{\sqrt{\varepsilon}}\int_{0}^{s}d\bar{\bar{W}}^{{Q_2}}(r).  \nonumber
\end{eqnarray}
This, together with the uniqueness of the solution to the equation \eqref{AuxiliaryPro Y 01},
implies
that the distribution of $(X_{k\delta}^{\varepsilon},\hat{Y}^{\vare}_{s+k\delta})$
coincides with the distribution of
$(X_{k\delta}^{\varepsilon},
Y_{\frac{s}{\varepsilon}}^{X_{k\delta}^{\varepsilon},Y_{k\delta}^{\varepsilon}})$.
%By the uniqueness of the solution, we have
%$$
%(X_{k\delta}^{\varepsilon},\hat{Y}^{\vare}_{s+k\delta})\simeq (X_{k\delta}^{\varepsilon},Y_{\frac{s}{\varepsilon}}^{X_{k\delta}^{\varepsilon},Y_{k\delta}^{\varepsilon}}),
%$$
%where $\simeq$ denotes a coincidence in distribution sense.

In order to estimate $\mathbb{E}\sup_{0\leq t\leq T} \|J_{4}^{1}(t) \|^{2p}$,
we first give a control of $\mathbb{E}\sup_{0\leq t\leq T} \|J_{4}^{1}(t)\|^{2}$:
\begin{eqnarray} \label{Esti J41}
\!\!\!\!\!\!\!\!&& \mathbb{E}\sup_{0\leq t\leq T} \| J_{4}^{1}(t) \|^{2}  \nonumber\\
=\!\!\!\!\!\!\!\!&&
\mathbb{E}\sup_{0\leq t\leq T}\Big \|\sum_{k=0}^{n_{t}-1}e^{(t-(k+1)\delta)A}
\int_{k\delta}^{(k+1)\delta}e^{((k+1)\delta-s)A}
\left[f(X_{k\delta}^{\varepsilon},\hat{Y}_{s}^{\varepsilon})
- \bar{f}(X_{k\delta}^{\varepsilon})\right]ds\Big \|^{2}
\nonumber\\
\leq\!\!\!\!\!\!\!\!&& \mathbb{E} \sup_{0\leq t\leq T}
\left\{n_{t}\sum_{k=0}^{n_{t}-1} \Big \|\int_{k\delta}^{(k+1)\delta}
e^{((k+1)\delta-s)A}\left[f(X_{k\delta}^{\varepsilon},\hat{Y}_{s}^{\varepsilon})-\bar{f}(X_{k\delta}^{\varepsilon})\right]ds\Big \|^{2}\right\}
\nonumber\\
\leq\!\!\!\!\!\!\!\!&&
[\frac{T}{\delta}]
\sum_{k=0}^{[\frac{T}{\delta}]-1}
\mathbb{E} \Big \|\int_{k\delta}^{(k+1)\delta}
e^{((k+1)\delta-s)A}\left[f(X_{k\delta}^{\varepsilon},\hat{Y}_{s}^{\varepsilon})-\bar{f}(X_{k\delta}^{\varepsilon})\right]ds\Big\|^{2}
\nonumber\\
\leq\!\!\!\!\!\!\!\!&&
\frac{C_{T}}{\delta^{2}}\max_{0\leq k\leq[\frac{T}{\delta}]-1}\mathbb{E}
\Big \| \int_{k\delta}^{(k+1)\delta}
e^{((k+1)\delta-s)A}\left[f(X_{k\delta}^{\varepsilon},\hat{Y}_{s}^{\varepsilon})-\bar{f}(X_{k\delta}^{\varepsilon})\right]ds \Big \|^{2}  \nonumber\\
=\!\!\!\!\!\!\!\!&&
C_{T}\frac{\vare^{2}}{\delta^{2}}\max_{0\leq k\leq[\frac{T}{\delta}]-1}
\mathbb{E}
\Big\| \int_{0}^{\frac{\delta}{\varepsilon}}
e^{(\delta-s\varepsilon)A}
\left[f(X_{k\delta}^{\varepsilon},\hat{Y}_{s\varepsilon+k\delta}^{\varepsilon})-\bar{f}(X_{k\delta}^{\varepsilon})\right]ds\Big\|^{2}  \nonumber\\
%=\!\!\!\!\!\!\!\!&&
%C_{T}\frac{\varepsilon^{2}}{\delta^{2}}
%\max_{0\leq k\leq[\frac{T}{\delta}]-1}\int_{0}^{1}\mathbb{E}
%\Big \|\int_{0}^{\frac{\delta}{\varepsilon}}e^{(\delta-s\varepsilon)A}
%\big(f(X_{k\delta}^{\varepsilon},\hat{Y}_{s\varepsilon+k\delta}^{\varepsilon})-\bar{f}(X_{k\delta}^{\varepsilon})\big)ds\Big\|^{2}d\xi   \nonumber\\
=\!\!\!\!\!\!\!\!&&C_{T}\frac{\varepsilon^{2}}{\delta^{2}}\max_{0\leq k\leq[\frac{T}{\delta}]-1}\int_{0}^{\frac{\delta}{\varepsilon}}
\int_{\tau}^{\frac{\delta}{\varepsilon}}\Psi_{k}(s,\tau)dsd\tau,  
\end{eqnarray}
where
\begin{eqnarray}
\Psi_{k}(s,\tau)=\!\!\!\!\!\!\!\!&&\mathbb{E}\left\langle e^{(\delta-s\varepsilon)A}
\big(f(X_{k\delta}^{\varepsilon},\hat{Y}_{s\varepsilon+k\delta}^{\varepsilon})-\bar{f}(X_{k\delta}^{\varepsilon})\big), e^{(\delta-\tau\varepsilon)A}
\big(f(X_{k\delta}^{\varepsilon},\hat{Y}_{\tau\varepsilon+k\delta}^{\varepsilon})-\bar{f}(X_{k\delta}^{\varepsilon})\big)\right\rangle  \nonumber\\
=\!\!\!\!\!\!\!\!&&\mathbb{E}\left\langle e^{(\delta-s\varepsilon)A}
\big(f(X_{k\delta}^{\varepsilon},Y_{s}^{X_{k\delta}^{\varepsilon},Y_{k\delta}^{\varepsilon}})-\bar{f}(X_{k\delta}^{\varepsilon})\big), e^{(\delta-\tau\varepsilon)A}\big(f(X_{k\delta}^{\varepsilon},Y_{\tau}^{X_{k\delta}^{\varepsilon},Y_{k\delta}^{\varepsilon}})-\bar{f}(X_{k\delta}^{\varepsilon})\big)\right\rangle.  \nonumber
\end{eqnarray}
Similar as the argument in \cite[appendix A]{FL}, using Lemma \ref{PMY},
one can verify that 
%there exists a constant $C>0$ such that
\begin{align}\label{EstiPsi}
\Psi_{k}(s,\tau)\leq
C\mathbb{E}\left(1 + \| X_{k\delta}^{\varepsilon} \|^{2}
+ \| Y_{k\delta}^{\varepsilon} \|^{2}\right)e^{-\frac{1}{2}(s-\tau)\eta} 
\leq C_{T}(1+ \| x \|^{2} + \| y \|^{2})e^{-\frac{1}{2}(s-\tau)\eta}.
\end{align}
%where $\eta=\lambda_{1}-L_{g}>0$.\\
%Keep in mind that we will set $\delta=\varepsilon^{1/2}$.
%We shall choose $\delta=\varepsilon^{1/2}$.
Combining \eqref{Esti J41} and \eqref{EstiPsi}, 
we get that for any $\varepsilon\in(0,1)$
\begin{align}  \label{J412}
\mathbb{E}\sup_{0\leq t\leq T} \| J_{4}^{1}(t) \|^{2}
%\leq\!\!\!\!\!\!\!\!&&
%C_{T}\frac{\varepsilon^{2}}{\delta^{2}}(1 + \| x \|^{2} + \| y \|^{2})
%\int_{0}^{\frac{\delta}{\varepsilon}}\int_{\tau}^{\frac{\delta}{\varepsilon}}e^{-\frac{1}{2}(s-\tau)\eta}dsd\tau    \nonumber\\
%=\!\!\!\!\!\!\!\!&&
%C_{T}\frac{\varepsilon^{2}}{\delta^{2}}(1+|x|_{L^2}^{2}+|y|_{L^2}^{2})
%\Big(\frac{2}{\eta}\cdot\frac{\delta}{\varepsilon}-\frac{4}{\eta^{2}}
%+e^{-\frac{\eta}{2}\cdot\frac{\delta}{\varepsilon}}\cdot \frac{4}{\eta^2}\Big)   \nonumber\\
\leq C_{T}\frac{\varepsilon}{\delta}(1+ \| x \|^{2}+ \| y \|^{2}).
\end{align}
By Lemmas \ref{SOX} and \ref{MDY}, we have
\begin{eqnarray}  \label{J4U}
 \mathbb{E}\sup_{0\leq t\leq T} \| J_{4}^{1}(t) \|^{2p}
&\leq &
  \mathbb{E}
       \Big(
        \int_0^T |f(X^\varepsilon_{[\frac{s}{\delta}] \delta}, \hat{Y}^\varepsilon_s)|+\bar{f}(X^\varepsilon_{[\frac{s}{\delta}] \delta})ds
       \Big)^{2p}\nonumber\\
&\leq & C_{p,T}\Big[1+\sup_{s\in[0,T]}\mathbb{E}\Big( \| X_s^\varepsilon \|^{2p}\Big)
                     +\sup_{s\in[0,T]}\mathbb{E}\Big( \| \hat{Y}_{s}^{\varepsilon} \|^{2p}\Big)
               \Big]\nonumber\\
\leq\!\!\!\!\!\!\!\!&&
C_{p,T}(1+ \| x \|^{2p} + \| y \|^{2p}).
\end{eqnarray}
This, together with \eref{J412}, implies
\begin{eqnarray} \label{J41}
\mathbb{E}\sup_{0\leq t\leq T} \| J_{4}^{1}(t) \|^{2p}
\leq\!\!\!\!\!\!\!\!&&
\left(\mathbb{E}\sup_{0\leq t\leq T}
\| J_{4}^{1}(t) \|^{2(2p-1)}\right)^{\frac{1}{2}}
\left(\mathbb{E}\sup_{0\leq t\leq T} \| J_{4}^{1}(t) \|^{2}\right)^{\frac{1}{2}} \nonumber\\
\leq\!\!\!\!\!\!\!\!&&C_{p,T}(1+ \| x \|^{2p} + \| y \|^{2p})\sqrt{\frac{\vare}{\delta}}.
\end{eqnarray}
Consequently, combining \eref{J42}, \eref{J43} and \eref{J41}, we get
\begin{align}
\mathbb{E}\sup_{0\leq t\leq T} \| J_{4}(t) \|^{2p}
\leq C_{p,T}(1+|x|^{2p}_{\alpha}+ \|y\|^{2p})\Big(\delta^p+\delta^{2p-1}
+\sqrt{\frac{\vare}{\delta}}\Big).  \label{J4}
\end{align}
According to the estimates $\eref{J1}$-$\eref{J2}$, $\eref{J31}$-$\eref{J6}$, $\eref{J4}$, we obtain
\begin{eqnarray*}
\mathbb{E}\Big(\sup_{0\leq t\leq T\wedge\tau_{n}^{\vare}}
\| \hat{X}_{t}^{\vare}-\bar{X}_{t}\|^{2p}\Big)
\leq \!\!\!\!\!\!\!\!&&
C_{p,T}(1+|x|^{6p}_{\alpha}+ \|y \|^{6p})
\Big(\delta^p+\frac{\delta^{p+\frac{1}{2}}}{\sqrt{\vare}}+\frac{\delta^{p+1}}{\vare}+\delta^{2p-1}+\sqrt{\frac{\vare}{\delta}}\Big)  \\
&&+C_{p,T}n^{2p}\int_{0}^{T}\mathbb{E}\sup_{0\leq r\leq s\wedge\tau_{n}^{\vare}}
\big \|\hat{X}^{\vare}_{r}-\bar{X}_{r}\big \|^{2p}ds.
\end{eqnarray*}
Using Gronwall's inequality, we get
\begin{align*}
&\  \mathbb{E}
\Big(\sup_{0\leq t\leq T\wedge\tau_{n}^{\vare}}
\|\hat{X}_{t}^{\vare}-\bar{X}_{t} \|^{2p}\Big)  \nonumber\\
\leq  & \
C_{p,T}(1+|x|^{6p}_{\alpha}+ \| y \|^{6p})
\big(\delta^p+\frac{\delta^{p+\frac{1}{2}}}{\sqrt{\vare}}+\frac{\delta^{p+1}}{\vare}+\delta^{2p-1}+\sqrt{\frac{\vare}{\delta}}\big)e^{C_{p,T}n^{2p}},
\end{align*}
which implies
\begin{eqnarray*}
&&\mathbb{E}\Big(\sup_{0\leq t\leq T}
\|\hat{X}_{t}^{\vare}-\bar{X}_{t}\|^{2p}
\cdot  1_{\{T\leq\tau_{n}^{\vare}\}}\Big)\\
\leq\!\!\!\!\!\!\!\!&&
C_{p,T}(1+|x|^{2p}_{\alpha}+ \|y \|^{2p})
\big(\delta^p+\frac{\delta^{p+\frac{1}{2}}}{\sqrt{\vare}}
+\frac{\delta^{p+1}}{\vare}+\delta^{2p-1}+\sqrt{\frac{\vare}{\delta}}\big)e^{C_{p,T}n^{2p}}.
\end{eqnarray*}
Taking $n=\sqrt[2p]{-\frac{1}{8C_{p,T}}\log\vare}$, $\delta=\vare^{\frac{1}{2}}$, we get
\begin{eqnarray}
\mathbb{E}\Big(\sup_{0\leq t\leq T}
\| \hat{X}_{t}^{\vare}-\bar{X}_{t} \|^{2p}
\cdot   1_{\{T\leq\tau_{n}^{\vare}\}}\Big)
\leq C_{p,T}\vare^{\frac{1}{8}}(1+|x|^{6p}_{\alpha}
+ \| y \|^{6p}).\label{befor tau_n}
\end{eqnarray}
Note that,  similarly to the proof of Lemma \ref{SOX}, one can check that
uniformly in $\vare\in(0,1)$, 
$\mathbb{E} (\sup_{0\leq t\leq T}|\bar{X}^{\vare}_t|_{1} )
\leq C_{T}(1+|x|_{\alpha}).$
%where we use the fact of  $\sup_{\vare\in(0,1)}\mathbb{E}\Big(\sup_{0\leq t\leq T}|\bar{X}^{\vare}_t|_{1}\Big)\leq C_{T}(1+|x|_{\alpha})$,
%which can be proved by the method similar to the proof in Lemma \ref{SOX}.
Combining this with \eqref{hatXHolderalpha}, we deduce that
\begin{align}
& \  \mathbb{E}\Big(\sup_{0\leq t\leq T}
\|\hat{X}_{t}^{\vare}-\bar{X}_{t} \|^{2p}
\cdot1_{\{T>\tau_{n}^{\vare}\}}\Big)
\leq
\Big(\mathbb{E}\sup_{0\leq t\leq T}
\| \hat{X}_{t}^{\vare}-\bar{X}_{t} \|^{4p}\Big)^{\frac{1}{2}}
\cdot\big[\mathbb{P}(T>\tau_{n}^{\vare})\big]^{\frac{1}{2}} \nonumber\\
\leq &  \
C_{p}\Big(\mathbb{E}\sup_{0\leq t\leq T}
\| \hat{X}_{t}^{\vare} \|^{4p}
+\mathbb{E} \sup_{0\leq t\leq T} \| \bar{X}_{t} \|^{4p}\Big)^{\frac{1}{2}}
\frac{1}{\sqrt{n}}\big(\sup_{\vare\in(0,1)}\mathbb{E}\sup_{0\leq t\leq T}|\hat{X}^{\vare}_t|_{1}+\mathbb{E}\sup_{0\leq t\leq T}|\bar{X}_{t}|_{1}\big)^{\frac{1}{2}} \nonumber\\
\leq & \
\frac{C_{p,T}}{\sqrt[4p]{-\log\vare}}
(1+|x|^{2p+\frac{1}{2}}_{\alpha}+ \|y\|^{2p+\frac{1}{2}}). \label{after tau_n}
\end{align}
Putting together (\ref{befor tau_n}) and (\ref{after tau_n}), we obtain
\begin{align*}
\mathbb{E}\sup_{0\leq t\leq T}
\| \hat{X}_{t}^{\vare}-\bar{X}_{t} \|^{2p}
%\leq\!\!\!\!\!\!\!\!&& C_{p,T}(1+|x|^{2p+\frac{1}{2}}_{\alpha}+|y|^{2p+\frac{1}{2}})
%\left\{\vare^{\frac{1}{8}}+\Big(\frac{1}{-\log\vare}\Big)^{\frac{1}{4p}}\right\} \nonumber\\
\leq C_{p,T}
(1+|x|^{6p}_{\alpha} + \| y \|^{6p})
\Big(\frac{1}{-\log\vare}\Big)^{\frac{1}{4p}}.
\end{align*}
The proof is complete.
%\hspace{\fill}$\square$
\end{proof}

\subsection{Proof of Theorem \ref{main result 1}}
%\textbf{Proof of Theorem \ref{main result 1}:}
Taking $\delta=\vare^{\frac{1}{2}}$, Lemma $\ref{DEX}$ implies
\begin{align*}
\mathbb{E}\sup_{0\leq t\leq T}
\| X_{t}^{\vare}-\hat{X}_{t}^{\vare} \|^{2p}
%\leq \!\!\!\!\!\!\!\!&&
%C_{p,T}(1+|x|^{2p}_{\alpha}+ \| y \|^{2p})\big(\vare^{\frac{p}{2}}
%+\vare^{\frac{p}{2}-\frac{1}{2}}\big)  \nonumber\\
\leq
%\!\!\!\!\!\!\!\!&&
C_{p,T}(1+|x|^{2p}_{\alpha}+ \|y\|^{2p})\vare^{\frac{p}{2}-\frac{1}{2}}.
\end{align*}
Combining this with Lemma \ref{ESX}, we obtain
\begin{eqnarray*}
\mathbb{E}\sup_{0\leq t\leq T} \| X_{t}^{\vare}-\bar{X}_{t}\|^{2p}
\leq\!\!\!\!\!\!\!\!&&
\mathbb{E}\sup_{0\leq t\leq T} \| X_{t}^{\vare}-\hat{X}_{t} \|^{2p}
+\mathbb{E}
\sup_{0\leq t\leq T} \|\hat{X}_{t}^{\vare}-\bar{X}_{t}\|^{2p}  \nonumber\\
\leq\!\!\!\!\!\!\!\!&&
C_{p,T}(1+ |x|^{6p}_{\alpha}+\|y\|^{6p})
\Big(\frac{1}{-\log\vare}\Big)^{\frac{1}{4p}}
\longrightarrow 0\quad (\vare\rightarrow0),
\end{eqnarray*}
which concludes the proof of Theorem \ref{main result 1}.\hspace{\fill}$\square$

\section{Proofs of Theorems \ref{main result 2} and \ref{main result 3}} \label{Sec Proof of Thm2 3}

This section is devoted to proving Theorems \ref{main result 2} and \ref{main result 3}.
Following the procedure inspired by \cite{B1},
the proofs are based on the Galerkin approximation and
the asymptotic expansion with respect to $\varepsilon$
of the solution to the Kolmogorov equation corresponding to
\eqref{main equation} with $Q_1=0$.
%The idea of the proofs follows the procedure inspired by \cite{B1}.
Since the proofs are tediously long and technical,
we first give a brief summary of the main ideas and steps in the proofs of
Theorems \ref{main result 2} and \ref{main result 3}.
%a brief summary of the main ideas and steps will be provided at first.
%And then, every subsection will be easier to be understood.
%Recall that we always assume there is no noise ($Q_{1}=0$)
%about the slow equation in system \eref{main equation}
%and condition \ref{A4} holds in this section.
%\vskip 0.3cm

%\subsection{The main ideas and steps of the proofs}\label{Weak Convergence Sub 01}
%Recall the assumptions in Theorem \ref{main result 2} and Theorem \ref{main result 3}.
%\vskip 0.2cm
${\mathbf{Step \ 1.}}$
Due to the unboundedness of operator $\Delta$, we use the Galerkin approximation
to reduce the infinite dimensional problem to a finite dimensional one as follows.

%\vskip 0.3cm
Let $H_{N}=\text{span}\{e_{k};1\leq k \leq N\}$. % and
Denote by $P_{N}$ the orthogonal projection of $L^2$ onto $H_{N}$.
Set $f_{N}(x,y)=P_{N}(f(x,y))$, $g_{N}(x,y)=P_{N}(g(x,y))$, $B_{N}(x)=P_{N}(B(x))$, $W^{Q_2}_{N}(t)=P_{N}W^{Q_2}(t)$ for $x,y\in H_{N}$.
The following equation is the finite dimensional projection of the system \eref{main equation}
with $Q_1=0$:
%Consider
%the following approximations of system \eref{main equation} and the averaged equation \eref{1.3}:
\begin{equation} \label{main finite equation}
\left\{\begin{array}{l}
\displaystyle
dX^{\vare}_{N}(t)=[AX^{\vare}_{N}(t)+B_{N}(X^{\vare}_{N}(t))+f_{N}(X^{\vare}_{N}(t),Y^{\vare}_{N}(t))]dt, 
\ \ X^{\vare}_{N}(0)=P_{N}x, \\
dY^{\vare}_{N}(t)=\frac{1}{\vare}[AY^{\vare}_{N}(t)+g_{N}(X^{\vare}_{N}(t),Y^{\vare}_{N}(t))]dt+\frac{1}{\sqrt{\vare}}dW_{N}^{Q_2}(t), \  \  Y^{\vare}_{N}(0)=P_{N}y.\\
%X^{\vare}_{N}(0)=P_{N}x,  Y^{\vare}_{N}(0)=P_{N}y.
\end{array}\right.
\end{equation}
Similarly, we consider the finite dimensional projection of the equation \eref{1.3} with $Q_1=0$:
\begin{equation}\left\{\begin{array}{l}\label{finite averaged equation}
\displaystyle
d\bar{X}_{N}(t)=[A\bar{X}_{N}(t)+B_{N}(\bar{X}_{N}(t))+\bar{f}_{N}(\bar{X}_{N}(t))]dt, \\
\bar{X}_{N}(0)=P_{N}x,
\end{array}\right.
\end{equation}
where $\bar{f}_{N}(x)=\int_{H_{N}}P_{N}f(x,y)\mu_{N}^{x}(dy)$,
and $\mu_{N}^{x}(dy)$ is the unique invariant measure for
$$dY_{N}(t)=[AY_{N}(t)+g_{N}(x,Y_{N}(t))]dt+dW_{N}^{Q_2}(t).$$
%\newline

%It is easily to see
%\begin{lemma} \label{Finite approximate}
%For any $\vare>0$, $t\geq0$ and $x,y\in H$,
%\begin{eqnarray*}
%\mathbb{E}\left|X^{\vare}(t)-X^{\vare}_{N}(t)\right|^{2}+\mathbb{E}\left|Y^{\vare}(t)-Y^{\vare}_{N}(t)\right|^{2}
%+\left|\bar{X}(t)-\bar{X}_{N}(t)\right|^{2}\rightarrow0 ~~as~~N\rightarrow+\infty.
%\end{eqnarray*}
%\end{lemma}

For the test function $\phi \in C_{b}^{2}(L^2)$, we have
\begin{align} \label{ephix}
& \  \mathbb{E}\left[\phi\left(X^{\vare}(t)\right)\right]
-\phi(\bar{X}(t) )  \nonumber\\
= & \  
\{\mathbb{E}\left[\phi\left(X^{\vare}(t)\right)\right]
-  \mathbb{E}\left[\phi\left(X_{N}^{\vare}(t)\right)\right] \}
+ \{ \mathbb{E}\left[\phi\left(X_{N}^{\vare}(t)\right)\right]-\phi (\bar{X}_{N}(t)) \}
+ \{  \phi (\bar{X}_{N}(t))-\phi(\bar{X}(t)) \}.
\end{align}
It is not difficult to show that
the first term and the third term in \eqref{ephix} converge to $0$,
%$|\mathbb{E}\left[\phi\left(X^{\vare}(t)\right)\right]
%-\mathbb{E}\left[\phi\left(X_{N}^{\vare}(t)\right)\right]|
%+
%|\phi\left(\bar{X}_{N}(t)\right)-\phi\left(\bar{X}(t)\right)| \to 0
%$
as
$N\to \infty.$
Therefore, in order to establish
Theorems \ref{main result 2} and \ref{main result 3},
it remains to show that the second term in \eqref{ephix} converges to $0$ as $N \to \infty$.
%Hence, to establish
%Theorem \ref{main result 2} and Theorem \ref{main result 3},
%we only need to deal with the second term of Eq. (\ref{ephix}).
We will give the main idea in the next step.

%\vskip 0.3cm
%{\bf Step 2.}, we establish some properties of $\bar{X}_{N}$ and $(X^{\vare}_{N}, Y^{\vare}_{N})$ which are independent of the dimension $N$ in Subsection \ref{Subsection 5.1} and Subsection \ref{Subsection 5.2} respectively.

%\vskip 0.3cm
$\mathbf{Step \ 2.}$ Inspired by \cite{B1},
we construct an asymptotic expansion of $\mathbb{E}\left[\phi\left(X_{N}^{\vare}(t)\right)\right]$.
Roughly speaking,
it has an expansion with respect to the small parameter $\vare$:
\begin{equation*}
\mathbb{E}\left[\phi\left(X_{N}^{\vare}(t)\right)\right]
=\phi (\bar{X}_{N}(t))+\vare u_{1}+v^{\vare}.
\end{equation*}
In order to control the second term in \eqref{ephix},
the main task is to analyze $u_1$ and $v^{\vare}$.
Almost all the work in this section is to deal with this step.
\vskip 0.3cm

This section is organized as follows. Subsections \ref{Subsection 5.1} and \ref{Subsection 5.2}
are to establish some properties of $\bar{X}_N$ and  $(X^{\vare}_N, Y^{\vare}_N)$ respectively.
The asymptotic expansion of $\mathbb{E}\left[\phi\left(X_{N}^{\vare}(t)\right)\right]$
will be given in Subsection \ref{Subsection 5.3}.
Based on some results obtained in subsections \ref{Subsection 5.1} and \ref{Subsection 5.2},
subsection \ref{Subsection 5.4} is to study properties of $u_1$ and $v^{\vare}$.
Finally, we prove Theorems \ref{main result 2} and \ref{main result 3}
in subsections \ref{Subsection 5.5} and \ref{ProofThm3}, respectively.

\subsection{Properties of \texorpdfstring{$\bar{X}_N$ } {Lg} }\label{Subsection 5.1}

This subsection is to establish some properties of $\bar{X}_N$.
\textbf{For simplicity, we omit the index $N$.}

\begin{lemma} \label{BarXgamma}
Assume the conditions \ref{A1}, \ref{A2} and \ref{A4} hold.  \\
(1) For any $x\in L^2$, $T>0$, there exists a constant $C>0$ such that
\begin{align}
\sup_{0\leq t\leq T}\|\bar{X}_t\|\leq C(1+\|x\|).\label{A.1.1}
\end{align}
(2) Furthermore, for any $x\in H^{\theta}$ with $\theta\in(0,1)$,
$\gamma\in(1,\frac{3}{2})$, $t\in (0, T]$,
there exist $k\in \mathbb{N}$ and a constant $C=C_{\gamma, \theta, T}>0$ such that
\begin{align}
|\bar{X}_{t}|_{\gamma}\leq C(|x|_{\theta}+1)t^{-\frac{\gamma-\theta}{2}}e^{C\|x\|^{k}}. \label{A.1.2}
\end{align}
%where $C$ is a positive constant depending on $\gamma, \theta, T$.
\end{lemma}

\begin{proof}
%Recall that
%\begin{equation}\left\{\begin{array}{l}
%\displaystyle \frac{d}{dt}\bar{X}_{t}=A\bar{X}_{t}+B(\bar{X}_{t})+\bar{f}(\bar{X}_{t}),\\
%\bar{X}_{0}=x.\end{array}\right. \label{A.1.3}
%\end{equation}
Multiplying both sides of the equation \eref{finite averaged equation} by $2\bar{X_t}$
and integrating with respect to $\xi$, we get
\begin{align*}
\frac{d}{dt} \|\bar{X}_t \|^{2}
=  2\langle A\bar{X_t}, \bar{X_t}\rangle+2\langle \bar{f}(\bar{X}_{t}), \bar{X}_t\rangle
\leq  C(1+ \|\bar{X}_t \|^{2}),
\end{align*}
which implies \eqref{A.1.1} by applying Gronwall's inequality.
%implies
%%\begin{eqnarray*}
%$\sup_{0\leq t\leq T}|\bar{X}_t|\leq C(1+|x|). $  %\label{A.1.4}
%%\end{eqnarray*}

To prove \eqref{A.1.2},
note that
\begin{align*}
\bar{X}_{t}
=e^{tA}x+\int_{0}^{t}e^{(t-s)A}B(\bar{X}_{s})ds+\int_{0}^{t}e^{(t-s)A}\bar{f}(\bar{X}_{s})ds.
\end{align*}
For the first term, we have
\begin{eqnarray}
|e^{tA}x|_{\gamma}\leq Ct^{-\frac{\gamma-\theta}{2}}|x|_{\theta}.\label{BXG1}
\end{eqnarray}
For the second term, it follows from $\eref{PSG}$ and Lemma \ref{Property B1} that
\begin{eqnarray}
\Big|\int^t_0e^{(t-s)A}B(\bar{X}_{s})ds\Big|_{\gamma}
\leq \!\!\!\!\!\!\!\!&&
C\int^t_0 \left[1+(t-s)^{-\frac{1+2\gamma}{4}}\right]|B(\bar{X}_{s})|_{-\frac{1}{2}}ds  \nonumber\\
\leq \!\!\!\!\!\!\!\!&&
C\int^t_0(t-s)^{-\frac{1+2\gamma}{4}} \| \bar{X}_{s} \| |\bar{X}_{s}|_{\gamma}ds \nonumber\\
\leq \!\!\!\!\!\!\!\!&&
C\int^t_0(t-s)^{-\frac{1+2\gamma}{4}}(1 + \| x \| )|\bar{X}_{s}|_{\gamma}ds.
\end{eqnarray}
For the last term, using \eqref{PSG}  and \eqref{A.1.1}, we get
\begin{align}
\Big|\int_{0}^{t}e^{(t-s)A}\bar{f}(\bar{X}_{s})ds\Big|_{\gamma}\leq
C\int^t_0\left[1+(t-s)^{-\frac{\gamma}{2}}\right](1+ \|\bar{X}_{s} \|)ds
\leq  C(1+ \| x \|).\label{BXG3}
\end{align}
Consequently, combining \eref{BXG1}-\eref{BXG3},
we conclude the proof of \eqref{A.1.2}  by using Lemma \ref{Gronwall 2}.
%and Lemma \ref{Gronwall 2}, the conclusion follows.
%\hspace{\fill}$\square$
\end{proof}

Note that from Lemma \ref{BarXgamma}, using the interpolation inequality,
we get that for
any $\gamma\in (0, 1]$, $\theta\in (0,1)$, $\delta\in(0, \frac{1}{2})$, $t\in (0, T]$,
there exist $k\in \mathbb{N}$ and a constant $C=C_{\theta, \delta, T}>0$
%which only depends on $\theta, \delta, T$
such that
%\begin{remark}
%For any $\gamma\in (0, 1]$, $\theta\in (0,1)$, $\delta\in(0, \frac{1}{2})$, $t\in (0, T]$, then by interpolation inequality, there exists $k\in \mathbb{N}$ such that
\begin{align} \label{BarX1}
|\bar{X}_{t}|_{\gamma}\leq
C\|\bar{X}_{t}\|^{\frac{1+\delta-\gamma}{1+\delta}}
|\bar{X}_{t}|^{\frac{\gamma}{1+\delta}}_{1+\delta}
\leq Ct^{-\frac{1+\delta-\theta}{2}\frac{\gamma}{1+\delta}}(|x|_{\theta}+1)e^{C\|x\|^k}.
\end{align}

%where $C$ is a positive constant depending on $\theta, \delta, T$.
%\end{remark}

\begin{lemma} \label{COXT}
Under the conditions \ref{A1}, \ref{A2} and \ref{A4},
for any $\theta\in(0,1)$, $\alpha\in(0,\frac{1}{2})$, $x\in H^{\theta}$, $0<s<t\leq T$,
there exist $k\in \mathbb{N}$ and a constant $C=C_{\theta, \alpha, T}>0$ such that
\begin{align}
|\bar{X}(t,x)-\bar{X}(s,x)|_{1}
\leq C(t-s)^{\frac{\alpha}{2}}s^{-\frac{1+\alpha-\theta}{2}}(|x|_{\theta}+1)e^{C\|x\|^k}. \nonumber
\end{align}
%where $C$ is a constant depending on $\theta, \alpha, T$.
\end{lemma}

\begin{proof}
In view of \eqref{finite averaged equation}, we write
%It is easy to see that
\begin{align} \label{bar X Contin}
& \  \bar{X}(t,x)-\bar{X}(s,x)  \nonumber\\
= & \  \big (e^{A(t-s)}-I \big)\bar{X}(s,x)
+\int_{s}^{t}e^{(t-r)A}B(\bar{X}(r,x))dr
+\int_{s}^{t}e^{(t-r)A}\bar{f}(\bar{X}(r,x))dr.
\end{align}
For the first term, using the property $\|(e^{tA}-I)x \|\leq Ct^{\frac{\alpha}{2}}|x|_{\alpha}$
and Lemma \ref{BarXgamma}, we get that there exists $k \in \mathbb{N}$ such that
\begin{eqnarray}  \label{bar X Contin 1}
\left|\big(e^{A(t-s)}-I \big)\bar{X}(s,x)\right|_{1}
\leq\!\!\!\!\!\!\!\!&&
C(t-s)^{\frac{\alpha}{2}}|\bar{X}(s,x)|_{1+\alpha}  \nonumber\\
\leq\!\!\!\!\!\!\!\!&&
C(t-s)^{\frac{\alpha}{2}}s^{-\frac{1+\alpha-\theta}{2}}(|x|_{\theta}+1)e^{C\|x\|^k}.
\end{eqnarray}
%for some $k\in \mathbb{N}$.
For the second term, according to Lemma \ref{BarXgamma},
there exists some $k\in\mathbb{N}$ such that
\begin{eqnarray}  \label{bar X Contin 2}
\left|\int_{s}^{t}e^{(t-r)A}B(\bar{X}(r,x))dr\right|_{1}\leq\!\!\!\!\!\!\!\!&&
C\int_{s}^{t}[1+(t-r)^{-\frac{3}{4}}]\left|B(\bar{X}(r,x))\right|_{-\frac{1}{2}}dr  \nonumber\\
\leq\!\!\!\!\!\!\!\!&&
C\int_{s}^{t}[1+(t-r)^{-\frac{3}{4}}]\|\bar{X}(r,x)\||\bar{X}(r,x)|_{1+\alpha}dr  \nonumber\\
\leq\!\!\!\!\!\!\!\!&&
C\int_{s}^{t}[1+(t-r)^{-\frac{3}{4}}]r^{-\frac{1+\alpha-\theta}{2}}(|x|_{\theta}+1)e^{C\|x\|^k}dr  \nonumber\\
\leq\!\!\!\!\!\!\!\!&&
C(t-s)^{\frac{1}{4}}s^{-\frac{1+\alpha-\theta}{2}}(|x|_{\theta}+1)e^{C\|x\|^k}.
\end{eqnarray}
%for some $k\in\mathbb{N}$.
For the third term, using Lemma \ref{BarXgamma} again, we obtain
\begin{eqnarray} \label{bar X Contin 3}
\left|\int_{s}^{t}e^{(t-r)A}\bar{f}(\bar{X}(r,x))dr\right|_{1}
\leq\!\!\!\!\!\!\!\!&&
C\int_{s}^{t}[1+(t-r)^{-\frac{1}{2}}]\left\|\bar{f}(\bar{X}(r,x))\right\|dr
\nonumber\\
\leq\!\!\!\!\!\!\!\!&&
C\int_{s}^{t}[1+(t-r)^{-\frac{1}{2}}](1+ \| \bar{X}(r,x) \|)dr  \nonumber\\
\leq\!\!\!\!\!\!\!\!&&
C(t-s)^{\frac{1}{2}}(1 + \|x\|).
\end{eqnarray}
The result follows by combining \eqref{bar X Contin}-\eqref{bar X Contin 3}.
%Eventually, there exists a constant $k\in\mathbb{N}$ such that
%\begin{align}
%|\bar{X}(t,x)-\bar{X}(s,x)|_{1}\leq C(t-s)^{\frac{\alpha}{2}}s^{-\frac{1+\alpha-\theta}{2}}(|x|_{\theta}+1)e^{C|x|^k}. \nonumber
%\end{align} \hspace{\fill}$\square$
\end{proof}

\begin{lemma} \label{ESDX}
Under the conditions \ref{A1}, \ref{A2} and \ref{A4},
for any $x\in H^{\theta}$ with $\theta\in(0,1)$,   $0\leq t\leq T$,
there exist $k\in \mathbb{N}$ and a constant $C=C_{\theta, T}>0$ such that
\begin{align}
\big\| \frac{d}{dt}\bar{X}(t,x) \big\|
\leq Ct^{-1+\frac{\theta}{2}}(|x|^2_{\theta}+1)e^{C \|x\|^k}. \nonumber
\end{align}
%where $C$ is a constant depending on $\theta, T$.
\end{lemma}

\begin{proof}
Recall that
\begin{align} \label{EquabarX}
\frac{d}{dt}\bar{X}(t,x)=A\bar{X}(t,x)+B(\bar{X}(t,x))+\bar{f}(\bar{X}(t,x)).
\end{align}
Taking $\delta=\frac{1}{4}$ in \eqref{BarX1}, we get
$\| B(\bar{X}(t,x)) \|\leq C|\bar{X}(t,x)|_1^2
\leq C t^{-1+\frac{4\theta}{5}}(|x|^2_{\theta}+1)e^{C\|x\|^k}.$
It is easy to see that
$\| \bar{f}(\bar{X}(t,x)) \|\leq C(1+\|x\|)$.
Hence, to prove Lemma \ref{ESDX},
it remains to control the first term in \eqref{EquabarX}.
From \eqref{EquabarX}, we write
\begin{eqnarray}
\bar{X}(t,x)
%\!\!\!\!\!\!\!\!&&e^{tA}x+\int_{0}^{t}e^{(t-s)A}B(\bar{X}(s,x))ds+\int_{0}^{t}e^{(t-s)A}\bar{f}(\bar{X}(s,x))ds  \nonumber\\
=\!\!\!\!\!\!\!\!&&
e^{tA}x+
\int_{0}^{t}e^{(t-s)A}B(\bar{X}(t,x))ds
+ \int_{0}^{t}e^{(t-s)A}\left[B(\bar{X}(s,x))-B(\bar{X}(t,x))\right]ds  \nonumber\\
\!\!\!\!\!\!\!\!&&
+\int_{0}^{t}e^{(t-s)A}\bar{f}(\bar{X}(t,x))ds
+\int_{0}^{t}e^{(t-s)A}\left[\bar{f}(\bar{X}(s,x))-\bar{f}(\bar{X}(t,x))\right]ds \nonumber\\
:=\!\!\!\!\!\!\!\!&&I_{1}+I_{2}+I_{3}+I_{4}+I_{5}. \nonumber
\end{eqnarray}
For $I_{1}$, using \eqref{PSG}, we have
$$
\|AI_1\| \leq Ct^{-1+\frac{\theta}{2}}|x|_{\theta}.
$$
For $I_{2}$, we deduce from Corollary \ref{Property B3} and \eqref{BarX1} that
\begin{eqnarray*} %\label{AX2}
\|AI_{2}\|=\!\!\!\!\!\!\!\!&&
\left\|(e^{tA}-I)B(\bar{X}(t,x))\right\|
\leq  2\left\|B(\bar{X}(t,x))\right\|  \nonumber\\
\leq \!\!\!\!\!\!\!\!&& 2\left|\bar{X}(t,x)\right|_{1}^{2}
\leq C t^{-1+\frac{4\theta}{5}}(|x|^2_{\theta}+1)e^{C\|x\|^k}.
\end{eqnarray*}
For $I_{3}$, according to Lemma \ref{COXT} and  \eqref{BarX1}, we get
\begin{eqnarray*}
\| AI_{3} \|
\leq\!\!\!\!\!\!\!\!&&
C\int_{0}^{t}\frac{1}{t-s}\left\|B(\bar{X}(t,x))-B(\bar{X}(s,x))\right\|ds \nonumber\\
\leq\!\!\!\!\!\!\!\!&&
C\int_{0}^{t}\frac{1}{t-s}|\bar{X}(t,x)-\bar{X}(s,x)|_{1}(|\bar{X}(t,x)|_{1}
+ |\bar{X}(s,x)|_{1})ds  \nonumber\\
\leq\!\!\!\!\!\!\!\!&&
C\int_{0}^{t}\frac{1}{t-s}(t-s)^{\frac{\alpha}{2}}
s^{-\frac{1+\alpha-\theta}{2}}(t^{-\frac{1}{2}+\frac{2\theta}{5}}+
s^{-\frac{1}{2}+\frac{2\theta}{5}})(|x|^2_{\theta}+1)e^{C\|x\|^k}ds \nonumber\\
\leq\!\!\!\!\!\!\!\!&&
C t^{-1+\frac{9\theta}{10}}(|x|^2_{\theta}+1)e^{C\|x\|^k}. \nonumber
\end{eqnarray*}
For $I_{4}$, it follows from Lemma \ref{BarXgamma} that
\begin{align*} \label{AX4}
\| AI_{4} \|= \left\| (e^{tA}-I)\bar{f}(\bar{X}(t,x))\right\|
\leq  C(1+ \|\bar{X}(t,x)\|)
\leq  C(1+ \|x\|).
\end{align*}
For $I_{5}$, using Lemma \ref{COXT} gives
\begin{align*} %\label{AX5}
\|AI_{5}\|\leq  C\int_{0}^{t}\frac{1}{t-s}\left\| \bar{X}(t,x)-\bar{X}(s,x) \right\|ds
\leq  C t^{-\frac{1-\theta}{2}}(|x|_{\theta}+1)e^{C\|x \|^k}.
\end{align*}
The conclusion follows by the above estimates.
%\hspace{\fill}$\square$
\end{proof}

%\vskip 0.3cm

Denote by $\eta^{h}(t,x)$  the derivative of $\bar{X}(t,x)$
with respect to $x$ in the direction $h$.
%$\eta^{h}(t,x):= \langle D_{x}\bar{X}(t,x),h\rangle$,
$\eta^{h}(t,x)$ satisfies the following equation
\begin{equation}\left\{\begin{array}{l} \label{Equa derivative}
\displaystyle
\frac{d\eta^{h}(t,x)}{dt}
=A\eta^{h}(t,x)+D\bar{f}(\bar{X}(t,x))\cdot\eta^{h}(t,x)
+D_{\xi}\left[\bar{X}(t,x)\eta^{h}(t,x)\right]\\
\eta^{h}(0,x)=h.
\end{array}\right.
\end{equation}

The following three Lemmas give some bounds for $\eta^{h}(t,x)$.
%the derivative of $\bar{X}(t,x)$
%with respect to $x$ in the direction $h$.

\begin{lemma} \label{eta}
Assume the conditions \ref{A1}, \ref{A2} and \ref{A4} hold.  \\
(1) For any $t\in(0,T]$, $h\in L^2$, there exists a constant $C>0$ such that
\begin{align} \label{LemDeriEtah 01}
\| \eta^{h}(t,x) \|^2+\int^t_0|\eta^{h}(s,x)|_{1}^{2}ds\leq Ce^{C\|x\|^{5}}\|h\|^{2}.
\end{align}
(2) For any $x\in H^{\theta}$ with $\theta\in(0,1)$, $h\in L^2$, $\gamma \in(1,\frac{3}{2})$,
$t\in (0, T]$,
there exist $k\in \mathbb{N}$ and a constant $C=C_{\gamma, \theta, T}$ such that
\begin{align} \label{LemDeriEtah 02}
|\eta^{h}(t,x)|_{\gamma}
\leq Ct^{-\frac{\gamma}{2}}(|x|_{\theta}+1)e^{C \|x\|^{k}}\|h \|.
\end{align}
%where $C$ is a constant depending on $\gamma, \theta, T$.
\end{lemma}

\begin{proof}
Multiplying both sides of the equation \eref{Equa derivative}
by $\eta^{h}(t,x)$ and integrating with respect to $\xi$, we obtain
\begin{eqnarray}
\frac{1}{2}\frac{d}{dt} \| \eta^{h}(t,x)\|^{2}
+|\eta^{h}(t,x)|_{1}^{2}
=\!\!\!\!\!\!\!\!&&
\int_{0}^{1}\left[D\bar{f}(\bar{X}(t,x))\eta^{h}(t,x)\right]\eta^{h}(t,x) d\xi \nonumber\\
\!\!\!\!\!\!\!\!&&
+\int_{0}^{1}D_{\xi}\left[\bar{X}(t,x)\eta^{h}(t,x)\right]\eta^{h}(t,x)d\xi \nonumber\\
\leq\!\!\!\!\!\!\!\!&&
C \| \eta^{h}(t,x)\|^{2}-\int_{0}^{1}\bar{X}(t,x)\eta^{h}(t,x)D_{\xi}\eta^{h}(t,x)d\xi \nonumber\\
=\!\!\!\!\!\!\!\!&&
C \| \eta^{h}(t,x)\|^{2}-b(\bar{X}(t,x),\eta^{h}(t,x),\eta^{h}(t,x)). \nonumber
\end{eqnarray}
According to Lemma \ref{Property B1} and the interpolation inequality,
it follows that
\begin{eqnarray}
\frac{1}{2}\frac{d}{dt} \| \eta^{h}(t,x) \|^{2}
+|\eta^{h}(t,x)|_{1}^{2}
\leq\!\!\!\!\!\!\!\!&&
C \|\eta^{h}(t,x) \|^{2} + \| \bar{X}(t,x)\| |\eta^{h}(t,x)|_{1}|\eta^{h}(t,x)|_{\frac{3}{5}} \nonumber\\
\leq\!\!\!\!\!\!\!\!&&
C \| \eta^{h}(t,x)\|^{2}
+ \| \bar{X}(t,x)\| |\eta^{h}(t,x)|_{1}^{\frac{8}{5}} \| \eta^{h}(t,x)\|^{\frac{2}{5}}   \nonumber\\
\leq\!\!\!\!\!\!\!\!&&
C\| \eta^{h}(t,x)\|^{2} + \frac{1}{2}|\eta^{h}(t,x)|_{1}^{2}
+C \| \bar{X}(t,x)\|^{5} \|\eta^{h}(t,x)\|^{2}. \nonumber
\end{eqnarray}
%Then, we get
This implies
\begin{eqnarray}
\| \eta^{h}(t,x) \|^{2}
+ \int^t_0|\eta^{h}(s,x)|_{1}^{2}ds
\leq \| h\|^{2}+C\int^t_0\left(1+ \|\bar{X}(s,x)\|^{5}\right) \| \eta^{h}(s,x)\|^{2}ds. \nonumber
\end{eqnarray}
Then, \eqref{LemDeriEtah 01} follows by using Gronwall's inequality and Lemma \ref{BarXgamma}.

%Finally, Gronwall inequality and Lemma \ref{BarXgamma} imply
To show \eqref{LemDeriEtah 02}, notice that
%\begin{eqnarray}
%|\eta^{h}(t,x)|^2+\int^t_0|\eta^{h}(s,x)|_{1}^{2}ds\leq Ce^{C|x|^{5}}|h|^{2}. \nonumber
%\end{eqnarray}
%(2) Notice that
\begin{align}\label{mild eta}
\eta^{h}(t,x)= e^{tA}h+\int_{0}^{t}e^{(t-s)A}D\bar{f}(\bar{X}(s,x))\cdot\eta^{h}(s,x)ds
+\int_{0}^{t}e^{(t-s)A}D_{\xi}\left[\bar{X}(s,x)\eta^{h}(s,x)\right]ds.
\end{align}
Using \eqref{LemDeriEtah 01}, Lemmas \ref{Property B1} and  \ref{BarXgamma},
we obtain that for $\gamma\in(1,\frac{3}{2})$,
\begin{eqnarray}
|\eta^{h}(t,x)|_{\gamma}\leq\!\!\!\!\!\!\!\!&&
Ct^{-\frac{\gamma}{2}} \|h\|
+C\int_{0}^{t}(t-s)^{-\frac{\gamma}{2}}\left|D\bar{f}(\bar{X}(s,x))\cdot\eta^{h}(s,x)\right|ds \nonumber\\
\!\!\!\!\!\!\!\!&&
+C\int_{0}^{t}(t-s)^{-\frac{2\gamma+1}{4}}\left|D_{\xi}\left[\bar{X}(s,x)\eta^{h}(s,x)\right]\right|_{-\frac{1}{2}}ds \nonumber\\
\leq\!\!\!\!\!\!\!\!&&
Ct^{-\frac{\gamma}{2}}\|h\|
+C\int_{0}^{t}(t-s)^{-\frac{\gamma}{2}}\|\eta^{h}(s,x)\|ds  \nonumber\\
\!\!\!\!\!\!\!\!&&
+C\int_{0}^{t}(t-s)^{-\frac{2\gamma+1}{4}}\left|B(\bar{X}(s,x),\eta^{h}(s,x))+B(\eta^{h}(s,x),\bar{X}(s,x))\right|_{-\frac{1}{2}}ds \nonumber\\
\leq\!\!\!\!\!\!\!\!&&
Ct^{-\frac{\gamma}{2}}\|h\|+Ce^{C|x|^{5}}\|h\| \nonumber\\
\!\!\!\!\!\!\!\!&&
+C\int_{0}^{t}(t-s)^{-\frac{2\gamma+1}{4}}\left[\|\eta^{h}(s,x)\| |\bar{X}(s,x)|_{\gamma}
+\|\bar{X}(s,x)\||\eta^{h}(s,x)|_{\gamma}\right]ds \nonumber\\
%\leq\!\!\!\!\!\!\!\!&&
%Ct^{-\frac{\gamma}{2}}\|h\|+Ce^{C\|x\|^{5}}\|h\|
%+Ct^{-\frac{4\gamma-3-2\theta}{4}}(|x|_{\theta}+1)e^{C\|x\|^{k}}\|h\|  \nonumber\\
%\!\!\!\!\!\!\!\!&&
%+C\int_{0}^{t}(t-s)^{-\frac{2\gamma+1}{4}}|\eta^{h}(s,x)|_{\gamma}(1+\|x\|)ds  \nonumber\\
\leq\!\!\!\!\!\!\!\!&&Ct^{-\frac{\gamma}{2}}(|x|_{\theta}+1) e^{C\|x\|^{k}} \|h\|
+\int_{0}^{t}(t-s)^{-\frac{2\gamma+1}{4}}|\eta^{h}(s,x)|_{\gamma}(1+\|x\|)ds.  \nonumber
\end{eqnarray}
Consequently, by Lemma \ref{Gronwall 2}, we get \eqref{LemDeriEtah 02}.
The proof is complete.
%there exists some $k\in\mathbb{N}$ such that
%\begin{align}
%|\eta^{h}(t,x)|_{\gamma}\leq Ct^{-\frac{\gamma}{2}}(|x|_{\theta}+1)e^{C|x|^{k}}|h|, \nonumber
%\end{align}
%where $C$ is a constant depending on $\gamma, \theta, T$.
%\hspace{\fill}$\square$
\end{proof}

Note that from Lemma \ref{eta}, using the interpolation inequality,
we deduce that for
any $\gamma\in (0, 1]$, $\theta\in (0,1)$,  $t\in (0, T]$,
there exist $k\in \mathbb{N}$ and a constant $C=C_{\theta, T}>0$
such that
%\begin{remark} \label{eta1}
%Similar as the argument in \ref{BarX1}, for any $\gamma\in(0,1]$, $t\in (0, T]$, then there exists $k\in \mathbb{N}$ such that
\begin{align}  \label{eta1}
|\eta^h(t,x)|_{\gamma}
\leq Ct^{-\frac{\gamma}{2}}(|x|_{\theta}+1)e^{C\|x\|^{k}}\|h\|.
\end{align}
%where $C$ is a constant depending on $\theta, T$.
%\end{remark}

\begin{lemma}\label{LemmaA.5}
Under the conditions \ref{A1}, \ref{A2} and \ref{A4},
%(1) For any $T>0$, $\alpha\in(0,\frac{1}{2})$, $\theta\in(1/2,1]$, $0<s\leq t\leq T$, $x,h\in H^{\theta}$, there exists a positive constant $C_T$ such that
%\begin{align}
%|\eta^{h}(t,x)-\eta^{h}(s,x)|\leq C_T(t-s)^{\frac{\alpha}{2}}(s^{-\frac{\alpha(1-\theta)}{2}}+1)|h|_{\theta}. \nonumber
%\end{align}
for any $x\in H^{\theta}$ with $\theta\in(0,1)$,
$\alpha\in(0,\frac{1}{2})$,  $0<s<t\leq T$,
there exist $k\in\mathbb{N}$ and a constant $C=C_{\alpha, \theta, T}$ such that
\begin{align}
|\eta^{h}(t,x)-\eta^{h}(s,x)|_{1}
\leq
C(t-s)^{\frac{\alpha}{2}}s^{-\frac{1+\alpha}{2}}(1+|x|_{\theta})e^{C\|x\|^{k}}\|h\|. \nonumber
\end{align}
%where $C$ is a constant depending on $\alpha, \theta, T$.
\end{lemma}
\begin{proof}
From \eqref{Equa derivative}, we write
%It is easy to see that
\begin{eqnarray} \label{ContinI5}
\eta^{h}(t,x)-\eta^{h}(s,x)=\!\!\!\!\!\!\!\!&&(e^{(t-s)A}-I)\eta^{h}(s,x)+\int_{s}^{t}e^{(t-r)A}D\bar{f}(\bar{X}(r,x))\cdot\eta^{h}(r,x)dr \nonumber\\
\!\!\!\!\!\!\!\!&&+\int_{s}^{t}e^{(t-r)A}D_{\xi}\left[\bar{X}(r,x)\eta^{h}(r,x)\right]dr  \nonumber\\
=\!\!\!\!\!\!\!\!&&I_{1}+I_{2}+I_{3}. 
\end{eqnarray}
For $I_{1}$, using \eqref{PSG} and \eqref{LemDeriEtah 02}, we have
\begin{align}  \label{I51}
|I_{1}|_{1}\leq C(t-s)^{\frac{\alpha}{2}}|\eta^{h}(s,x)|_{1+\alpha}
\leq C(t-s)^{\frac{\alpha}{2}}s^{-\frac{1+\alpha}{2}}(1+|x|_{\theta})e^{C\|x\|^{k}} \|h\|.
\end{align}
For $I_{2}$, using \eqref{LemDeriEtah 01}, we get
\begin{align}   \label{I52}
|I_{2}|_{1}\leq
\int_{s}^{t}(t-r)^{-\frac{1}{2}}\left \| D\bar{f}(\bar{X}(r,x))\cdot \eta^{h}(r,x)\right\| dr
\leq  C(t-s)^{\frac{1}{2}}e^{C\|x\|^{5}}\|h\|.
\end{align}
For $I_{3}$, using \eqref{PSG}, Lemmas \ref{BarXgamma} and \ref{eta}, we obtain
\begin{eqnarray}  \label{I53}
|I_{3}|_{1}\leq\!\!\!\!\!\!\!\!&&
C\int_{s}^{t}\left[1+(t-r)^{-\frac{3}{4}}\right]
\left|D_{\xi}\left[\bar{X}(r,x)\eta^{h}(r,x)\right]\right|_{-\frac{1}{2}}dr  \nonumber\\
\leq\!\!\!\!\!\!\!\!&&C\int_{s}^{t}\left[1+(t-r)^{-\frac{3}{4}}\right]\left|B(\bar{X}(r,x),\eta^{h}(r,x))+B(\eta^{h}(r,x),\bar{X}(r,x))\right|_{-\frac{1}{2}}dr  \nonumber\\
\leq\!\!\!\!\!\!\!\!&&C\int_{s}^{t}\left[1+(t-r)^{-\frac{3}{4}}\right]\left(\left\|\bar{X}(r,x)\right\|\left|\eta^{h}(r,x)\right|_{1+\alpha}
+\left|\bar{X}(r,x)\right|_{1+\alpha} \left\|\eta^{h}(r,x)\right\| \right)dr  \nonumber\\
%\leq\!\!\!\!\!\!\!\!&&
%C(t-s)^{\frac{1}{4}}\left[s^{-\frac{1+\alpha}{2}} (1+|x|_{\theta})
%+ s^{-\frac{1+\alpha-\theta}{2}}(|x|_{\theta}+1)\right] e^{C\|x\|^{k}}\|h\|  \nonumber\\
\leq\!\!\!\!\!\!\!\!&&C(t-s)^{\frac{1}{4}}s^{-\frac{1+\alpha}{2}}(1+|x|_{\theta})e^{C\|x\|^{k}} \|h\|.
\end{eqnarray}
We conclude the proof by combining \eqref{ContinI5}-\eqref{I53}.
%Finally, we obtain
%\begin{eqnarray*} \label{I5}
%|\eta^{h}(t,x)-\eta^{h}(s,x)|_1\leq C(t-s)^{\frac{\alpha}{2}}s^{-\frac{1+\alpha}{2}}(1+|x|_{\theta})e^{C|x|^{k}}|h|.
%\end{eqnarray*}
%The proof is complete. \hspace{\fill}$\square$
\end{proof}

\begin{lemma} \label{ESDET}
Under the conditions \ref{A1}, \ref{A2} and \ref{A4},
for any $x\in H^{\theta}$ with $\theta\in(0,1)$,  $h\in L^2$,
$0<s<t\leq T$,  there exist $k\in\mathbb{N}$ and a constant $C=C_{\theta, T}$ such that
\begin{align}
\left\|\frac{d}{dt}\eta^{h}(t,x)\right\|\leq Ct^{-1}(|x|_{\theta}^{2}+1)e^{C\|x\|^{k}}\|h\|.  \nonumber
\end{align}
%where $C$ is a constant depending on $\theta, T$.
\end{lemma}
\begin{proof}
%Recall that
%\begin{eqnarray}
%\frac{d\eta^{h}(t,x)}{dt}
%=\!\!\!\!\!\!\!\!&&A\eta^{h}(t,x)+D\bar{f}(\bar{X}(t,x))\cdot\eta^{h}(t,x)+D_{\xi}\left[\bar{X}(t,x)\eta^{h}(t,x)\right].  \nonumber
%\end{eqnarray}
We first control the first term in \eqref{Equa derivative}.
Notice that
\begin{eqnarray}
\eta^{h}(t,x)
%=\!\!\!\!\!\!\!\!&&e^{tA}h+\int_{0}^{t}e^{(t-s)A}D\bar{f}(\bar{X}(s,x))\cdot\eta^{h}(s,x)ds
%+\int_{0}^{t}e^{(t-s)A}D_{\xi}\left[\bar{X}(s,x)\eta^{h}(s,x)\right]ds  \nonumber\\
=\!\!\!\!\!\!\!\!&&e^{tA}h+\int_{0}^{t}e^{(t-s)A}D\bar{f}(\bar{X}(t,x))\cdot\eta^{h}(t,x)ds \nonumber\\
\!\!\!\!\!\!\!\!&&+\int_{0}^{t}e^{(t-s)A}\left[D\bar{f}(\bar{X}(s,x))\cdot\eta^{h}(s,x)-D\bar{f}(\bar{X}(t,x))\cdot\eta^{h}(t,x)\right]ds  \nonumber\\
\!\!\!\!\!\!\!\!&&+\int_{0}^{t}e^{(t-s)A}D_{\xi}\left[\bar{X}(t,x)\eta^{h}(t,x)\right]ds \nonumber\\
\!\!\!\!\!\!\!\!&&+\int_{0}^{t}e^{(t-s)A}\left\{D_{\xi}\left[\bar{X}(s,x)\eta^{h}(s,x)\right]-D_{\xi}\left[\bar{X}(t,x)\eta^{h}(t,x)\right]\right\}ds \nonumber\\
=\!\!\!\!\!\!\!\!&&I_{1}+I_{2}+I_{3}+I_{4}+I_{5}. \nonumber
\end{eqnarray}
For $I_{1}$, we have
\begin{align} \label{AI1}
\|A I_1\| \leq C t^{-1} \|h\|.
\end{align}
For $I_{2}$, it follows from condition \ref{A4} that
\begin{align} \label{AI2}
\|AI_{2}\|= \left \|(e^{tA}-I)D\bar{f}(\bar{X}(t,x))\cdot\eta^{h}(t,x)\right\|
\leq  2\left\|D\bar{f}(\bar{X}(t,x))\cdot\eta^{h}(t,x)\right\|
\leq  Ce^{C\|x\|^{5}} \|h\|.
\end{align}
For $I_{3}$, according to Lemmas \ref{COXT} and \ref{LemmaA.5}, we obtain
\begin{eqnarray} \label{AI3}
\| AI_{3} \|
\leq\!\!\!\!\!\!\!\!&&
C\int_{0}^{t}\frac{1}{t-s}
\left\| D\bar{f}(\bar{X}(t,x))\cdot\eta^{h}(t,x)-D\bar{f}(\bar{X}(s,x))\cdot\eta^{h}(s,x)\right\|ds
\nonumber\\
\leq\!\!\!\!\!\!\!\!&&
C\int_{0}^{t}\frac{1}{t-s}\left\|[D\bar{f}(\bar{X}(t,x))-D\bar{f}(\bar{X}(s,x))]\cdot\eta^{h}(t,x)\right\|ds \nonumber\\
\!\!\!\!\!\!\!\!&&
+C\int_{0}^{t}\frac{1}{t-s}\left\|D\bar{f}(\bar{X}(s,x))\cdot
\left(\eta^{h}(t,x)-\eta^{h}(s,x)\right)\right\|ds \nonumber\\
\leq\!\!\!\!\!\!\!\!&&
C\int_{0}^{t}\frac{1}{t-s}\|\bar{X}(t,x)-\bar{X}(s,x)\| e^{C\|x\|^{5}}\|h\| ds
+ C\int_{0}^{t}\frac{1}{t-s} \left\|\eta^{h}(t,x)-\eta^{h}(s,x)\right\|ds
\nonumber\\
\leq\!\!\!\!\!\!\!\!&&
Ct^{-\frac{1}{2}}(1+|x|_{\theta})e^{C \|x\|^{k}} \|h\|.
\end{eqnarray}
For $I_{4}$, we deduce from Lemma \ref{Property B1}, \eqref{BarX1} and \eqref{eta1} that
\begin{eqnarray} \label{AI4}
\|AI_{4}\|=\!\!\!\!\!\!\!\!&&
\left\|(e^{tA}-I)D_{\xi}\left[\bar{X}(t,x)\eta^{h}(t,x)\right]\right\| \nonumber\\
\leq\!\!\!\!\!\!\!\!&&
2 \| B(\bar{X}(t,x),\eta^{h}(t,x))+B(\eta^{h}(t,x),\bar{X}(t,x))\|       \nonumber\\
\leq\!\!\!\!\!\!\!\!&&
C\left|\bar{X}(t,x)\right|_{1}\left|\eta^{h}(t,x)\right|_{1}  %\nonumber\\
\leq   %\!\!\!\!\!\!\!\!&&
Ct^{-1+\frac{2\theta}{5}}(|x|_{\theta}^{2}+1)e^{C \|x\|^{k}}\|h\|.
\end{eqnarray}
For $I_{5}$, using Lemmas \ref{COXT} and \ref{LemmaA.5}, \eqref{BarX1} and \eqref{eta1},
we get
\begin{eqnarray} \label{AI5}
\|AI_{5}\|\leq\!\!\!\!\!\!\!\!&&
C\int_{0}^{t}\frac{1}{t-s}\left\|B\left(\bar{X}(s,x), \eta^{h}(t,x)-\eta^{h}(s,x)\right)+B\left(\eta^{h}(t,x)-\eta^{h}(s,x), \bar{X}(s,x)\right)\right.\nonumber\\
\!\!\!\!\!\!\!\!&&
\left.+B\left(\bar{X}(t,x)-\bar{X}(s,x), \eta^{h}(t,x)\right)+B\left(\eta^{h}(t,x), \bar{X}(t,x)-\bar{X}(s,x)\right)\right\|ds \nonumber\\
\leq\!\!\!\!\!\!\!\!&&
C\int_{0}^{t}\frac{1}{t-s}\left(|\bar{X}(s,x)|_{1}|\eta^{h}(t,x)-\eta^{h}(s,x)|_1+|\eta^{h}(t,x)|_1|\bar{X}(t,x)-\bar{X}(s,x)|_1\right)ds \nonumber\\
\leq\!\!\!\!\!\!\!\!&&
Ct^{-1+\frac{2\theta}{5}}(|x|_{\theta}^{2}+1)e^{C\|x\|^{k}} \|h\|.
\end{eqnarray}
Putting \eref{AI1}-\eref{AI5} together, we get
%Then \eref{AI1}-\eref{AI5} imply
\begin{align}
\|A\eta^{h}(t,x)\| \leq Ct^{-1}(|x|_{\theta}^{2}+1)e^{C\|x\|^{k}} \|h\|.\label{eta2}
\end{align}
For the second term in \eqref{Equa derivative}, we have
\begin{eqnarray}\label{DF}
\| D\bar{f}(\bar{X}(t,x))\cdot\eta^{h}(t,x) \|\leq Ce^{C\|x\|^{5}} \|h\|.
\end{eqnarray}
For the third term in \eqref{Equa derivative}, it follows by \eref{AI4} that
\begin{eqnarray} \label{dx}
\left\|D_{\xi}\left[\bar{X}(t,x)\eta^{h}(t,x)\right]\right\|
\leq Ct^{-1+\frac{2\theta}{5}}(|x|_{\theta}^{2}+1)e^{C\|x\|^{k}} \|h\|.
\end{eqnarray}
The result follows by combining \eref{eta2}-\eref{dx}.
%Eventually, by \eref{eta2}-\eref{dx} we complete the proof. \hspace{\fill}$\square$
\end{proof}

Denote by $\zeta^{h,k}(t,x)$ the second derivative of $\bar{X}(t,x)$
with respect to $x$ towards the directions $h, k \in L^2$.
Then, $\zeta^{h,k}(t,x)$ satisfies
%  $D^2_{xx}\bar{X}(t,x)\cdot(h,k)$,
%which is the solution of the following equation:
\begin{eqnarray}\label{XI1}
\frac{d\zeta^{h,k}(t,x)}{dt}=\!\!\!\!\!\!\!\!&&A\zeta^{h,k}(t,x)+D^{2}\bar{f}(\bar{X}(t,x))\cdot\left(\eta^{h}(t,x),\eta^{k}(t,x)\right)+D\bar{f}(\bar{X}(t,x))\cdot \zeta^{h,k}(t,x)\nonumber\\
\!\!\!\!\!\!\!\!&&+D_{\xi}\left[\eta^{k}(t,x)\eta^{h}(t,x)\right]+D_{\xi}\left[\bar{X}(t,x)\zeta^{h,k}(t,x)\right].
\end{eqnarray}
The following lemma gives a control of $\zeta^{h,k}(t,x)$.
%deals with the second derivative of $\bar{X}(t,x)$
%with respect to $x$ towards directions $h,k\in L^2$.

\begin{lemma}  \label{zeta}
Under the conditions \ref{A1}, \ref{A2} and \ref{A4},
for any $x\in H^{\theta}$ with $\theta\in(0,1)$,  $h, k\in L^2$,  $t\in (0, T]$,
there exists a constant $C>0$ such that
\begin{align}
\left\|\zeta^{h,k}(t,x)\right\| \leq Ce^{C\|x\|^{5}}\|h\| \|k\|. \nonumber
\end{align}
\end{lemma}

\begin{proof}
Multiply both sides of the equation \eref{XI1} by $\zeta^{h,k}(t,x)$,
integrate with respect to $\xi$, then we obtain
\begin{eqnarray} \label{etaHK}
\frac{1}{2}\frac{d}{dt} \|\zeta^{h,k}(t,x)\|^{2}
+ |\zeta^{h,k}(t,x)|_{1}^{2}
=\!\!\!\!\!\!\!\!&&
\int_{0}^{1}\left[D^{2}\bar{f}(\bar{X}(t,x))\cdot(\eta^{h}(t,x),\eta^{k}(t,x))\right]\zeta^{h,k}(t,x)d\xi \nonumber\\
\!\!\!\!\!\!\!\!&&
+\int_{0}^{1}\left[D\bar{f}(\bar{X}(t,x))\cdot \zeta^{h,k}(t,x)\right]\zeta^{h,k}(t,x)d\xi \nonumber\\
\!\!\!\!\!\!\!\!&&
+\int_{0}^{1}\left\{D_{\xi}\left[\eta^{k}(t,x)\eta^{h}(t,x)\right]\right\}\zeta^{h,k}(t,x)d\xi \nonumber\\
\!\!\!\!\!\!\!\!&&
+\int_{0}^{1}\left\{D_{\xi}[\bar{X}(t,x)\zeta^{h,k}(t,x)]\right\}\zeta^{h,k}(t,x)d\xi \nonumber\\
=\!\!\!\!\!\!\!\!&&
I_1+I_2+I_3+I_4.
%\leq\!\!\!\!\!\!\!\!&&C|h||\eta^{k}(t,x)|_{\theta}|\xi^{h,k}(t,x)|+C|\xi^{h,k}(t,x)|^{2}+C|h||\eta^{k}(t,x)|_{1}|\xi^{h,k}(t,x)| \nonumber\\
%\!\!\!\!\!\!\!\!&&+C|k||\eta^{h}(t,x)|_{1}|\xi^{h,k}(t,x)|+C|\xi^{h,k}(t,x)||\xi^{h,k}(t,x)|_{1}.  \nonumber
\end{eqnarray}
For $I_1$, using \eqref{LemDeriEtah 01}, we get
\begin{align} \label{etaHK 01}
\|I_1\| \leq
Ce^{C\|x\|^{5}} \|h\|^2 \|k\|^2 + \|\zeta^{h,k}(t,x)\|^{2}.
\end{align}
For $I_2$, we have
\begin{align} \label{etaHK 02}
\|I_2\| \leq C \|\zeta^{h,k}(t,x)\|^{2}.
\end{align}
For $I_3$,  according to Lemma \ref{Property B1}, we get
\begin{eqnarray} \label{etaHK 03}
\|I_3\| \leq\!\!\!\!\!\!\!\!&&
\left|b\left(\eta^{k}(t,x),\eta^{h}(t,x),\zeta^{h,k}(t,x)\right)+b\left(\eta^{h}(t,x),\eta^{k}(t,x),\zeta^{h,k}(t,x)\right)\right| \nonumber\\
\leq \!\!\!\!\!\!\!\!&&
\left(\|\eta^{k}(t,x)\| |\eta^{h}(t,x)|_1
+ \| \eta^{h}(t,x)\| |\eta^{k}(t,x)|_1\right)|\zeta^{h,k}(t,x)|_1 \nonumber\\
\leq\!\!\!\!\!\!\!\!&&
C\left(\|\eta^{k}(t,x)\| |\eta^{h}(t,x)|_1
+ \| \eta^{h}(t,x)\| |\eta^{k}(t,x)|_1\right)^2+\frac{1}{4}|\zeta^{h,k}(t,x)|^2_1.
\end{eqnarray}
For $I_4$, using Lemma \ref{Property B1} and the interpolation inequality, we obtain
\begin{eqnarray} \label{etaHK 04}
\|I_4\| =\!\!\!\!\!\!\!\!&&
-\int_{0}^{1}\left[\bar{X}(t,x)\zeta^{h,k}(t,x)\right]D_{\xi}\zeta^{h,k}(t,x)d\xi \nonumber\\
\leq\!\!\!\!\!\!\!\!&&
\left|b(\bar{X}(t,x),\zeta^{h,k}(t,x),\zeta^{h,k}(t,x))\right| \nonumber\\
\leq\!\!\!\!\!\!\!\!&&
C \| \bar{X}(t,x)\| |\zeta^{h,k}(t,x)|_1|\zeta^{h,k}(t,x)|_{\frac{3}{5}} \nonumber\\
\leq\!\!\!\!\!\!\!\!&&
C \| \bar{X}(t,x)\| |\zeta^{h,k}(t,x)|_1^{\frac{8}{5}} \| \zeta^{h,k}(t,x)\|^{\frac{2}{5}} \nonumber\\
\leq\!\!\!\!\!\!\!\!&&
C(1+ \|x \|^{5}) \| \zeta^{h,k}(t,x)\|^{2} + \frac{1}{4}|\zeta^{h,k}(t,x)|^2_{1}.
\end{eqnarray}
Then, combining \eqref{etaHK}-\eqref{etaHK 04}, using Lemma \ref{eta}, we get
\begin{eqnarray*}
\|\zeta^{h,k}(t,x)\|^{2}
\leq\!\!\!\!\!\!\!\!&& Ce^{C\|x\|^{5}} \|h\|^2 \|k\|^2
+ C\sup_{0\leq s\leq T} \|\eta^{k}(s,x)\|^{2} \int^t_0|\eta^{h}(s,x)|_1^{2}ds\\
&& +C\sup_{0\leq s\leq T} \| \eta^{h}(s,x)\|^{2} \int^t_0|\eta^{k}(s,x)|_1^{2}ds
+ C \int^t_0(1+\|x\|^{5}) \|\zeta^{h,k}(s,x)\|^{2} ds\\
\leq\!\!\!\!\!\!\!\!&& Ce^{C\|x\|^{5}} \|h\|^2 \|k\|^2+
C\int^t_0(1+ \|x\|^{5}) \|\zeta^{h,k}(s,x)\|^{2} ds.
\end{eqnarray*}
The desired result follows by using Gronwall's inequality.
%Eventually, Gronwall inequality yields the desired result.  \hspace{\fill}$\square$
\end{proof}

\subsection{Properties of \texorpdfstring{$(X^{\vare}_N, Y^{\vare}_N)$} {Lg} } \label{Subsection 5.2}

This subsection is to establish some properties of $(X^{\vare}_N, Y^{\vare}_N)$.
\textbf{For simplicity, we omit the index $N$.}

\begin{lemma} \label{highorder of X}
Assume the conditions \ref{A1}, \ref{A2} and \ref{A4} hold.  \\
(1) For any $x\in L^2$, $t\in [0, T]$, there exists a constant $C>0$ such that
\begin{align} \label{Highorder X L2}
\| X^{\varepsilon}_{t} \| \leq C(1+\|x\|).
\end{align}
(2) For any $x\in H^{\theta}$ with $\theta\in(0,1)$, $y\in L^2$,
 $\gamma\in(1,\frac{3}{2})$, $p\geq1$, $t\in (0, T]$, there exist $k\in\NN$
 and a constant $C=C_{p, \theta, \gamma, T}$ such that
\begin{align}
\EE|X^{\varepsilon}_{t}|^{p}_{\gamma}
\leq C t^{-\frac{p(\gamma-\theta)}{2}}(1+|x|^p_{\theta} + \|y\|^p) e^{C\|x\|^k}. \label{B.1.1}
\end{align}
%where $C$ is a constant depending on $p, \theta, \gamma, T$.
\end{lemma}
\begin{proof}
Recall that
\begin{align}
\frac{d}{dt}X^{\varepsilon}_{t}=AX^{\varepsilon}_{t}+B(X^{\varepsilon}_{t})+f(X^{\varepsilon}_{t},Y^{\varepsilon}_{t}). \nonumber
\end{align}
Multiplying both sides of the above equation by $2X^{\varepsilon}_{t}$
and using \eqref{WeakDissIne} in \ref{A4}, we obtain
\begin{align*}
\frac{1}{2}\frac{d}{dt} \|X^{\varepsilon}_{t}\|^{2}=
-|X^{\varepsilon}_{t}|_{1}^{2}
+\langle f(X^{\varepsilon}_{t},Y^{\varepsilon}_{t}),X^{\varepsilon}_{t}\rangle
\leq  C(1+\|X^{\varepsilon}_{t}\|^{2}).
\end{align*}
Then, \eqref{Highorder X L2} follows by using Gronwall's inequality.
%Consequently, Gronwall inequality yields that $|X^{\varepsilon}_{t}|^{2}\leq C(1+|x|^{2})$,
%and the result follows.

To show \eqref{B.1.1},
recall that
\begin{align} \label{XNtEsti}
X^{\varepsilon}_t=e^{tA}x+\int^t_0e^{(t-s)A}B(X^{\varepsilon}_s)ds+\int^t_0e^{(t-s)A}f(X^{\varepsilon}_s, Y^{\varepsilon}_s)ds: = I_1 + I_2 + I_3.
\end{align}
For $I_1$, \eref{PSG} implies
\begin{eqnarray} \label{XNtEsti a}
|e^{tA}x|_{\gamma}\leq C t^{-\frac{\gamma-\theta}{2}}|x|_{\theta}.
\end{eqnarray}
For $I_2$, similarly as the proof of Lemma \ref{BarXgamma}, using \eqref{Highorder X L2}, we obtain
\begin{eqnarray} \label{XNtEsti b}
\Big|\int^t_0e^{(t-s)A}B(X^{\varepsilon}_s)ds\Big|_{\gamma}
\leq \!\!\!\!\!\!\!\!&&C\int^t_0(t-s)^{-\frac{1+2\gamma}{4}}(1+\|x\|)
|X^{\varepsilon}_{s}|_{\gamma}ds.
\end{eqnarray}
For the last term, we get
\begin{eqnarray}  \label{XNtEsti c}
\Big|\int_{0}^{t}e^{(t-s)A}f(X_{s}^{\vare},Y_{s}^{\vare})ds\Big|_{\gamma}
\leq \!\!\!\!\!\!\!\!&&
C\int^t_0(t-s)^{-\frac{\gamma}{2}}(1+ \|X^{\varepsilon}_{s}\| + \|Y^{\varepsilon}_s\|)ds.
\end{eqnarray}
Combining \eqref{XNtEsti}-\eqref{XNtEsti c},
it follows from the Minkowski inequality that  for any $p>1$,
\begin{eqnarray*}
\left[\EE|X^{\varepsilon}_{t}|^{p}_{\gamma}\right]^{1/p}
\leq \!\!\!\!\!\!\!\!&&
C t^{-\frac{\gamma-\theta}{2}}|x|_{\theta}+C\int^t_0(t-s)^{-\frac{1+2\gamma}{4}}(1+\|x\|)
\left[\EE|X^{\varepsilon}_{s}|^p_{\gamma}\right]^{1/p}ds  %\\
%\!\!\!\!\!\!\!\!&&
+C(1+\|x\|+ \|y\|).
\end{eqnarray*}
Using Lemma \ref{Gronwall 2}, we get \eqref{B.1.1}.
%Hence, by Lemma \ref{Gronwall 2}, there exists $k\in\mathbb{N}$ such that
%\begin{eqnarray*}
%\left[\EE|X^{\varepsilon}_{t}|^{p}_{\gamma}\right]^{1/p}\leq \!\!\!\!\!\!\!\!&&C t^{-\frac{\gamma-\theta}{2}}(1+|x|_{\theta}+|y|)e^{C|x|^k},
%\end{eqnarray*}
%which implies the result.
\end{proof}

%\begin{remark}
Similarly to \eqref{BarX1}, for any $x\in H^{\theta}$ with $\delta\in(0, \frac{1}{2})$,
$y\in L^2$, $\gamma\in(0,1]$,  $p\geq1$, $t\in (0, T]$,
there exist $k\in\NN$ and a constant $C=C_{p, \delta, \theta, T}$ such that
%for any , $\theta\in(0,1)$, ,
\begin{eqnarray} \label{Remark5.3}
\EE|X^{\varepsilon}_{t}|^{p}_{\gamma}
\leq C t^{-\frac{p\gamma(1+\delta-\theta)}{2(1+\delta)}}(1+|x|^p_{\theta}+\|y\|^p) e^{C\|x\|^k}.
\end{eqnarray}
%where $C$ is a constant depending on $p, \delta, \theta, T$.
%\end{remark}

\begin{lemma} \label{LB.2}
Assume the conditions \ref{A1}, \ref{A2} and \ref{A4} hold. \\
(1) For any $x\in H^{\theta}$ with $\theta\in(0,1)$, $y\in L^2$,
$\alpha\in(0,\frac{1}{4})$, $0<s<t\leq T$,
there exist $k\in\NN$ and a constant $C=C_{\theta, \alpha, T}$ such that
\begin{align}
\left[\EE|X_{t}^{\vare}-X_{s}^{\vare}|^4_{1}\right]^{1/4}
\leq C(t-s)^{\frac{\alpha}{2}}s^{-\frac{1+\alpha-\theta}{2}}(|x|_{\theta}+\|y\|+1)
e^{C\|x\|^k}. \nonumber
\end{align}
%where $C$ is a constant depending on $\theta, \alpha, T$.
(2) For any $x,y\in L^2$,
$\alpha\in(0,\frac{1}{4})$, $0<s<t\leq T$, there exists a constant $C>0$ such that
\begin{align*}
\EE \|Y_{t}^{\vare}-Y_{s}^{\vare}\|^2
\leq C(\|x\|^2 + \|y\|^2+1) \big(s^{-2\alpha} + \varepsilon^{-2\alpha} \big) (t-s)^{2\alpha}.
%\left[\left(\frac{t-s}{s}\right)^{2\alpha}
%+\left(\frac{t-s}{\varepsilon}\right)^{2\alpha}\right]. \nonumber
\end{align*}
\end{lemma}

\begin{proof}
The proof of (1) can be carried out in the same way as  in Lemma \ref{COXT}.
For the proof of (2), since the approach here is almost the same as the one in \cite{B1}, 
we refer to \cite[Proposition A.4]{B1} for details. 
%and
%the approach here is almost the same as the one in \cite{B1}
%even if we do not assume that function $g$ is bounded.
%We omit the details.
%Therefore we omit the proof.
%\hspace{\fill}$\square$
\end{proof}

\begin{lemma} \label{Xvare2}
Under the conditions \ref{A1}, \ref{A2} and \ref{A4},
for any $x\in H^{\theta}$ with $\theta\in(0,1)$,
$\alpha\in(0,\frac{1}{4})$,  $0<t\leq T$, there exist $k\in\NN$
and a constant $C=C_{\theta, \alpha, T}$ such that
\begin{align}
\left[\mathbb{E}\|AX^{\vare}_{t}\|^2\right]^{\frac{1}{2}}
\leq C t^{-1+\frac{\theta}{2}}(|x|_{\theta}^{2}+ \|y\|^2+1)e^{C\|x\|^k}+C\vare^{-\alpha}.  \nonumber
\end{align}
%where $C$ is a constant depending on $\theta, \alpha, T$.
\end{lemma}

\begin{proof}
For $0\leq s<t$, we write
\begin{eqnarray}
X^{\vare}_{t}
%=\!\!\!\!\!\!\!\!&&e^{tA}x
%+\int_{0}^{t}e^{(t-s)A}B(X^{\vare}_{s})ds
%+\int_{0}^{t}e^{(t-s)A}f(X^{\vare}_{s},Y^{\vare}_{s})ds  \nonumber\\
=\!\!\!\!\!\!\!\!&&e^{tA}x+\int_{0}^{t}e^{(t-s)A}B(X^{\vare}_{t})ds+\int_{0}^{t}e^{(t-s)A}\left[B(X^{\vare}_{s})-B(X^{\vare}_{t})\right]ds  \nonumber\\
\!\!\!\!\!\!\!\!&&+\int_{0}^{t}e^{(t-s)A}f(X^{\vare}_{t},Y^{\vare}_{t})ds
+\int_{0}^{t}e^{(t-s)A}\left[f(X^{\vare}_{s},Y^{\vare}_{s})-f(X^{\vare}_{t},Y^{\vare}_{t})\right]ds \nonumber\\
=\!\!\!\!\!\!\!\!&&I_{1}+I_{2}+I_{3}+I_{4}+I_{5}. \nonumber
\end{eqnarray}
For $I_{1}$, using \eqref{PSG}, we have
\begin{align} \label{AXvare1}
\|Ae^{tA}x\|
\leq C t^{-1+\frac{\theta}{2}}|x|_{\theta}^{2}.
\end{align}
For $I_{2}$, taking $p=4$, $\gamma=1$ and $\delta=1/4$ in \eqref{Remark5.3},
we get
\begin{align} \label{AX2}
\left[\EE\|AI_{2}\|^2\right]^{1/2}=
\left[\EE\left\|(e^{tA}-I)B(X_{t}^{\vare})\right\|^2\right]^{1/2}
& \leq C\left[\EE\left|X_{t}^{\vare}\right|_{1}^{4}\right]^{1/2}  \nonumber\\
& \leq C t^{-1+\frac{4\theta}{5}}(1+|x|^{2}_{\theta}+\|y\|^{2})e^{C\|x\|^k}.
\end{align}
For $I_{3}$, using Lemma \ref{Property B2}, we obtain
\begin{align*}
|AI_{3}|\leq C\int_{0}^{t}\frac{1}{t-s}\left|B(X^{\vare}_{s})-B(X^{\vare}_{t})\right|ds 
\leq C\int_{0}^{t}\frac{1}{t-s}|X^{\vare}_{t}-X^{\vare}_{s}|_{1}(|X^{\vare}_{t}|_{1}+|X^{\vare}_{s}|_{1})ds.  
\end{align*}
According to the Minkowski inequality, by Lemmas \ref{highorder of X} and \ref{LB.2}, it follows that
\begin{eqnarray}  \label{AXvare3}
\left[\mathbb{E}\|AI_{3}\|^2\right]^{1/2}\leq\!\!\!\!\!\!\!\!&&
C\mathbb{E}\int_{0}^{t}\frac{1}{t-s}|X^{\vare}_{t}-X^{\vare}_{s}|_{1}(|X^{\vare}_{t}|_{1}+|X^{\vare}_{s}|_{1})ds \nonumber\\
\leq\!\!\!\!\!\!\!\!&&
C\int_{0}^{t}\frac{1}{t-s}\left\{\mathbb{E}\left[  |X^{\vare}_{t}-X^{\vare}_{s}|_{1}(|X^{\vare}_{t}|_{1}
+|X^{\vare}_{s}|_{1})\right]^2\right\}^{1/2}ds \nonumber\\
\leq\!\!\!\!\!\!\!\!&&
C\int_{0}^{t}\frac{1}{t-s}\left[\mathbb{E}\left(|X^{\vare}_{t}-X^{\vare}_{s}|^{4}_{1}\right)\cdot\left(\mathbb{E}|X^{\vare}_{t}|^{4}_{1}
+\mathbb{E}|X^{\vare}_{s}|^{4}_{1}\right)\right]^{\frac{1}{4}}ds \nonumber\\
\leq\!\!\!\!\!\!\!\!&&C t^{-1+\frac{9\theta}{10}}(|x|^{2}_{\theta}+ \|y\|^{2}+1) e^{C\|x\|^k}.
\end{eqnarray}
For $I_{4}$, we have
\begin{align} \label{AXvare4}
\mathbb{E}\|AI_{4}\|=
\mathbb{E}\left\|(e^{tA}-I)f(X^{\vare}_{t},Y^{\vare}_{t})\right\|
\leq  C(1+\mathbb{E}\|X^{\vare}_{t}\| + \mathbb{E}\|Y^{\vare}_{t}\|)
\leq  C(1+\|x\| + \|y\|).
\end{align}
For $I_{5}$, using the Minkowski inequality and Lemma \ref{LB.2}, we obtain
\begin{eqnarray} \label{AXvare5}
\left[\mathbb{E} \| AI_{5}\|^2\right]^{1/2}
\leq\!\!\!\!\!\!\!\!&&
C\int_{0}^{t}\frac{1}{t-s}
\left[\left(\mathbb{E}\left\|X^{\vare}_{t}-X^{\vare}_{s}\right\|^{2}\right)^{\frac{1}{2}}
+\left(\mathbb{E}\left\|Y^{\vare}_{t}-Y^{\vare}_{s}\right\|^{2}\right)^{\frac{1}{2}}\right]ds \nonumber\\
\leq\!\!\!\!\!\!\!\!&&
C t^{-\frac{1}{2}+\frac{\theta}{2}}(|x|_{\theta}+ \|y\|+1) e^{C\|x\|^k} + C\vare^{-\alpha}.
\end{eqnarray}
Combining \eref{AXvare1}-\eref{AXvare5} yields the desired result.
\end{proof}

\begin{lemma}\label{LemmaB.3}
Under the conditions \ref{A1}, \ref{A2} and \ref{A4},
for any $x\in H^{\theta}$ with $\theta\in(0,1)$,
$\alpha\in(0,\frac{1}{4})$,  $0<t\leq T$, there exist $k\in\NN$
and a constant $C=C_{\theta, \alpha, T}$ such that
\begin{align}
\EE\big\|\frac{d}{dt}X_{t}^{\vare}\big\|\leq
C t^{-1+\frac{\theta}{2}}(|x|_{\theta}^{2}+ \|y\|^2+1)e^{C\|x\|^k}+C\vare^{-\alpha}. \nonumber
\end{align}
%where $C$ is a constant depending on $\theta, \alpha, T$.
\end{lemma}

\begin{proof}
Recall that
\begin{eqnarray}
\frac{d}{dt}X_{t}^{\vare}=AX_{t}^{\vare}+B(X_{t}^{\vare})+f(X_{t}^{\vare},Y_{t}^{\vare}). \nonumber
\end{eqnarray}
Choosing $p=2, \gamma=1$ and $\delta=\frac{1}{4}$ in \eqref{Remark5.3},
we get
$\mathbb{E} \|B(X_{t}^{\vare})\|
\leq C\mathbb{E}|X_{t}^{\vare}|_1^2\leq C t^{-1+\frac{4\theta}{5}}(|x|^{2}_{\theta}+ \|y\|^{2}+1)$.
It is clear that $\mathbb{E}\|f(X_{t}^{\vare},Y_{t}^{\vare})\|\leq C(1+\|x\| + \|y\|)$.
By Lemma \ref{Xvare2}, we have
\begin{eqnarray}
\mathbb{E}\|AX^{\vare}_{t}\| \leq \left[\mathbb{E} \| AX^{\vare}_{t}\|^2\right]^{1/2}\leq
 C t^{-1+\frac{\theta}{2}}(|x|_{\theta}^{2}+ \|y\|^2+1)e^{C\|x\|^k}+C\vare^{-\alpha}. \nonumber
\end{eqnarray}
The proof is complete.
%\hspace{\fill}$\square$
\end{proof}

\subsection{The asymptotic expansion of
\texorpdfstring{$\mathbb{E}[\phi(X_{N}^{\vare}(t))]$}{Lg} } \label{Subsection 5.3}

For any $x,y\in L^2$ and $t\geq0$, set
\begin{equation} \label{Ephix1}
u^{\vare}_N(t,x,y)=\mathbb{E}\left[\phi\left(X^{\vare}_N(t,x,y)\right)\right]
\end{equation}
and
\begin{equation} \label{Ephix2}
\bar{u}_N(t,x)=\phi\big(\bar{X}_N(t,x)\big).
\end{equation}
%\vskip 0.2cm
In this subsection, we shall find an asymptotic expansion of $u^{\vare}_N$ with respect to $\vare$:
\begin{equation} \label{Expand u}
u^{\vare}_N=u_{0}^N+\vare u_{1}^N+v^{\vare}_N,
\end{equation}
where $v^{\vare}_N$ is a residual term,
$u_{0}^N$ and $u_{1}^N$ will  be constructed below.
%We will show that $u_{0}^N=\bar{u}_N$.
%\vskip 0.3cm

For any $\psi(x,y): L^2 \times L^2 \to \mathbb{R}$ with $\psi \in C^2$,
we introduce the differential operators:
%We first introduce the following differential operators:
%for a $C^{2}$ function $\psi(x,y): L^2 \times L^2 \to \mathbb{R}$,
\begin{equation}
L_{1}^N\psi(x,y)=\langle A_Nx+B_N(x)+f_N(x,y),D_{x}\psi(x,y)\rangle, \nonumber
\end{equation}
\begin{equation}
L_{2}^N\psi(x,y)=\langle A_Nx+g_N(x,y),D_{y}\psi(x,y)\rangle+\frac{1}{2}\text{Tr}\left(D_{yy}^{2}\psi(x,y)\right). \nonumber
\end{equation}
For any $\psi: L^2\rightarrow\mathbb{R}$ with $\psi \in C^1$, denote by
\begin{equation}
\bar{L}^N\psi(x)=\langle A_Nx+B_N(x)+\bar{f}_N(x),D_{x}\psi(x)\rangle. \nonumber
\end{equation}
Set
\begin{equation} \label{Global L}
L^{\vare}_N=L_{1}^N+\frac{1}{\vare}L_{2}^N.
\end{equation}

\textbf{For simplicity, we omit the index $N$.}
Notice that $\bar{u}$ does not depend on $y$.
It is well known that $u^{\vare}$ and $\bar{u}$ satisfy the following Kolmogorov equations:
\begin{equation}\left\{\begin{array}{l}  \label{Global LK}
\displaystyle
\frac{\partial u^{\vare}(t,x,y)}{\partial t}=L^{\vare}u^{\vare}(t,x,y)\\
u^{\vare}(0,x,y)=\phi(x),
\end{array}\right.
\end{equation}
and
\begin{equation} \label{KolmogEquBar u}
\left\{\begin{array}{l}
\displaystyle
\frac{\partial \bar{u}(t,x)}{\partial t}=\bar{L}\bar{u}(t,x)\\
\bar{u}(0,x)=\phi(x), 
\end{array}\right.
\end{equation}
respectively. 
%\vskip 0.3cm
Using \eref{Expand u}-\eref{Global LK}, 
%\eref{Global L}, \eref{Global LK} and \eref{Expand u}, 
we have
\begin{equation}
\frac{\partial u_{0}}{\partial t}+\vare\frac{\partial u_{1}}{\partial t}+\frac{\partial v^{\vare}}{\partial t}
=L_{1}u_{0}+\frac{1}{\vare}L_{2}u_{0}+\vare L_{1}u_{1}+L_{2}u_{1}+L_{1}v^{\vare}+\frac{1}{\vare}L_{2}v^{\vare},  \nonumber
\end{equation}
which implies
%The identification with respect to $\vare$ gives the following equations:
\begin{equation}\left\{\begin{array}{l} \label{expand equality}
\displaystyle
L_{2}u_{0}=0\\
\frac{\partial u_{0}}{\partial t}=L_{1}u_{0}+L_{2}u_{1}\\
\frac{\partial v^{\vare}}{\partial t}=L^{\vare}v^{\vare}+\vare(L_{1}u_{1}-\frac{\partial u_{1}}{\partial t}).
\end{array}\right.
\end{equation}
%\newline
To obtain $u_{0}$ and $u_{1}$, we need the following lemma which is similar to \cite[Lemma 4.3]{B1}.
However, instead of imposing the condition that $f$ is bounded,
the coefficient $f$ in our paper is Lipschitz.
% and may be not bounded.
%Notice that \cite[Lemma 4.3]{B1} plays an important role in the proof, we also need a Lemma to prove our main result, but notice that coefficient $f$ in our paper is Lipschitz and not bounded, which is different from the bounded assumption on corresponding coefficient in \cite{B1}, so we need to do some revises in \cite[Lemma 4.3]{B1}, stated as follows:
\begin{lemma} \label{Poisson equation}
Assume the conditions \ref{A1}, \ref{A2} and \ref{A4} hold.
Fix $x\in L^2$.\\
(1) If $\Psi$ is a Lipschitz continuous function
and $\Phi$ is a function of class $C^{2}$
satisfying $L_{2}\Phi=-\Psi$, then for any $y\in L^2$, we have
\begin{equation}
\Phi(y)=\int_{L^2} \Phi(z)\mu^{x}(dz)
+ \int_{0}^{+\infty}\mathbb{E}\left[\Psi(Y^{x,y}_s)\right]ds. \nonumber
\end{equation}
(2) Suppose that $\Psi$ is a Lipschitz continuous function of class $C^{2}$
such that $\int_{L^2}\Psi(y)\mu^x(dy)=0$.
Let $\Phi(y)=\int_{0}^{+\infty}\mathbb{E}\left[\Psi(Y^{x,y}_s)\right]ds$. 
Then, $\Phi$ is of class $C^{2}$ and satisfies $L_{2}\Phi=-\Psi$.
%then $\Phi$ defined by $\Phi(y)=\int_{0}^{+\infty}\mathbb{E}\left[\Psi(Y^{x,y}_s)\right]ds$
%is of class $C^{2}$, satisfies $L_{2}\Phi=-\Psi$.
Moreover, there exists a constant $C$ which is independent of $N$ such that for any $y\in L^2$
\begin{equation}
|\Phi(y)|\leq C(1+\|x\| + \|y\|)|\Psi|_{Lip}. \label{Poission equation}
\end{equation}
\end{lemma}
\begin{proof}
Since the proof of (1) and the proof of the first part of (2)
are similar to \cite[Lemma 4.3]{B1}, we omit the details.
Here we only give a proof of \eref{Poission equation}.

For any $y\in L^2$, Proposition \ref{Rem 4.1} implies
\begin{equation}
\left|\mathbb{E}\left[\Psi(Y^{x,y}_s)\right]
-\int_{L^2}\Psi(z)\mu^x(dz)\right|
\leq C(1+\|x\|+\|y\|)e^{-\frac{(\lambda_1-L_g)s}{2}}|\Psi|_{Lip}.\label{5.11}
\end{equation}
Noting that $\int_{L^2 }\Psi(y)\mu^x(dy)=0$,
we obtain \eref{Poission equation} by integrating \eref{5.11} with respect to $s$.
%and by integrating with respect to time $s$ in \eref{5.11},
%we obtain \eref{Poission equation},
The proof is complete.
% \hspace{\fill}$\square$
\end{proof}

%\vskip 0.3cm
It follows from \eref{expand equality} that
the function $u_{0}$ is independent of $y$,
thus we can write $u_{0}(t,x,y)=u_{0}(t,x)$.
We also choose the initial condition $u_{0}(0,x)=\phi(x)$.
In view of the second equation in \eref{expand equality}
and noting that $\int_{L^2 }L_{2}u_{1}(t,x,y)\mu^{x}(dy)=0$, we have
\begin{eqnarray*}
\frac{\partial u_{0}}{\partial t}(t,x)
=\!\!\!\!\!\!\!\!&&
\int_{L^2} \frac{\partial u_{0}}{\partial t}(t,x)\mu^{x}(dy) \nonumber\\
=\!\!\!\!\!\!\!\!&&
\int_{L^2} L_{1}u_{0}(t,x)\mu^{x}(dy)+\int_{L^2} L_{2}u_{1}(t,x,y)\mu^{x}(dy) \nonumber\\
=\!\!\!\!\!\!\!\!&&
\langle Ax+B(x)+\int_{L^2} f(x,y)\mu^{x}(dy),D_{x}u_{0}(t,x)\rangle \nonumber\\
=\!\!\!\!\!\!\!\!&&
\bar{L}u_{0}(t,x).
\end{eqnarray*}
This, together with the uniqueness of the solution to \eqref{KolmogEquBar u},
implies that $u_{0}=\bar{u}$.

%We can see that $u_{0}$ and $\bar{u}$ are solutions of the same equation,
%so by uniqueness of the solution we can deduce that $u_{0}=\bar{u}$.

From $\bar{L}u_{0}=L_{1}u_{0}+L_{2}u_{1}$ and the definitions of $\bar{L}$ and $L_{1}$,
we deduce that
\begin{align*}
L_{2}u_{1}(t,x,y)= \langle\bar{f}(x)-f(x,y),D_{x}u_{0}(t,x)\rangle
=: -\chi(t,x,y),
\end{align*}
where $\chi$ is of class $C_{b}^{2}$ with respect to $y$,
and satisfies that for any $t\geq0$ and $x\in L^2$, $\int_{L^2}\chi(t,x,y)\mu^{x}(dy)=0$.

According to Lemma \ref{Poisson equation}, we obtain
\begin{eqnarray} \label{Equality u1}
u_{1}(t,x,y)=\int_{0}^{+\infty}\mathbb{E}\left[\chi(t,x,Y^{x,y}_s)\right]ds.
\end{eqnarray}
In what follows, we are going to show the regularity of $u_{1}$
with respect to $t$ and $x,y$.
In order to avoid the non-integrability at $t=0$,
we introduce a parameter $\rho(\vare)=\vare^{\frac{1}{a}}, 0<a\leq\frac{\theta}{2}$.
By the third equation of \eref{expand equality} and It\^{o}'s formula, we have
\begin{eqnarray*}
v^{\vare}(t,x,y)=\!\!\!\!\!\!\!\!&&\mathbb{E}\left[v^{\vare}(\rho(\vare),X^{\vare}(t-\rho(\vare),x,y),Y^{\vare}(t-\rho(\vare),x,y))\right] \nonumber\\
\!\!\!\!\!\!\!\!&&+\vare\mathbb{E}\left[\int_{\rho(\vare)}^{t}(L_{1}u_{1}-\frac{\partial u_{1}}{\partial s})(s,X^{\vare}(t-s,x,y),Y^{\vare}(t-s,x,y))ds\right].
\end{eqnarray*}
Using the expansion \eref{Expand u} and the fact $u_{0}=\bar{u}$,  we get
\begin{eqnarray} \label{Ephix3}
\!\!\!\!\!\!\!\!&&u^{\vare}(t,x,y)-\bar{u}(t,x,y)\nonumber\\
=\!\!\!\!\!\!\!\!&& \vare u_{1}(t,x,y)+\mathbb{E}\left[v^{\vare}(\rho(\vare),X^{\vare}(t-\rho(\vare),x,y),Y^{\vare}(t-\rho(\vare),x,y))\right] \nonumber\\
\!\!\!\!\!\!\!\!&&+\vare\mathbb{E}\left[\int_{\rho(\vare)}^{t}(L_{1}u_{1}-\frac{\partial u_{1}}{\partial s})(s,X^{\vare}(t-s,x,y),Y^{\vare}(t-s,x,y))ds\right].
\end{eqnarray}
Hence, it remains to control each terms in \eqref{Ephix3},
%that is, we have to estimate $u_{1}$, $v^{\vare}$, $L_{1}u_{1}$ and $\frac{\partial u_{1}}{\partial t}$,
which will be shown in the next subsection.

\subsection{Estimates of
\texorpdfstring{ $u_{1}$, $v^{\vare}$, $L_{1}u_{1}$ and $\frac{\partial u_{1}}{\partial t}$}{Lg} }\label{Subsection 5.4}

Note that the index $N$ is omitted in the equations \eqref{main finite equation}
and \eqref{finite averaged equation}.
%we consider Eq. (\ref{finite averaged equation}) and Eq. (\ref{main finite equation}) and
\begin{lemma}\label{ESu1}
Under the conditions \ref{A1}, \ref{A2} and \ref{A4},
there exists a constant $C>0$ such that for any $0\leq t\leq T$, $x,y\in L^2$, we have
\begin{align}
|u_{1}(t,x,y)|\leq Ce^{C\|x\|^{5}}(1+\|y\|).\nonumber
\end{align}
\end{lemma}

\begin{proof}
By \eref{Equality u1} and Lemma \ref{Poisson equation}, we have
\begin{align}
|u_{1}(t,x,y)|\leq C(1+\|x\|+\|y\|) \left|y\longmapsto\chi(t,x,y)\right|_{Lip}, \nonumber
\end{align}
where
\begin{align}
\chi(t,x,y)=\langle f(x,y)-\bar{f}(x), D_{x}u_{0}(t,x)\rangle. \nonumber
\end{align}
Noting that $f$ is Lipschitz, we get that for any $y_1, y_2 \in L^2$,
\begin{eqnarray*}
|\chi(t,x,y_1)-\chi(t,x,y_2)|\leq C\|y_1-y_2\| \|D_{x}u_{0}(t,x)\|.
\end{eqnarray*}
To bound $\|D_{x}u_{0}(t,x)\|$,
recalling that $u_{0}=\bar{u}$ and $\bar{u}(t,x)=\phi(\bar{X}(t,x))$, we have
\begin{align*}
D_{x}u_{0}(t,x)\cdot h=D\phi(\bar{X}(t,x))\cdot \eta^{h}(t,x),
\end{align*}
where
%\begin{align}
$\eta^{h}(t,x)=D_{x}\bar{X}(t,x)\cdot h. $  % \nonumber
%\end{align}
Finally the conclusion follows by using Lemma \ref{eta}.
%\hspace{\fill}$\square$
\end{proof}

%Now we are going to estimate $\frac{\partial u_{1}}{\partial t}(t,x,y)$. At first, we need some estimates about $|\bar{X}(t,x)-\bar{X}(s,x)|_{1}$, $\frac{d}{dt}\bar{X}(t,x)$, %$|\eta^{h}(t,x)|_{2}$ and $\frac{d\eta^{h}(t,x)}{dt}$.

\begin{lemma} \label{du1t}
Under the conditions \ref{A1}, \ref{A2} and \ref{A4},
for any $x\in H^{\theta}$ with $\theta\in(0, 1]$, $y\in L^2$,
$0\leq t\leq T$,  there exist $k\in \mathbb{N}$ and a constant $C=C_{\theta, T}$ such that
\begin{align}
\left|\frac{\partial u_{1}}{\partial t}(t,x,y)\right|
\leq Ct^{-1}(1+|x|_{\theta}^{2})e^{C\|x\|^{k}}(1+\|y\|).  \nonumber
\end{align}
%where $C$ is a constant depending on $\theta, T$.
\end{lemma}

\begin{proof}
By the definition of $u_{1}$, we have
\begin{align}
\frac{\partial u_{1}}{\partial t}(t,x,y)
=\int_{0}^{+\infty} \mathbb{E} \big[\frac{\partial\chi}{\partial t}(t,x,Y^{x,y}_s) \big] ds,  \nonumber
\end{align}
where
\begin{align}
\frac{\partial\chi}{\partial t}(t,x,y)
=\big\langle f(x,y)-\bar{f}(x), \frac{\partial}{\partial t}D_{x}u_{0}(t,x)\big\rangle. \nonumber
\end{align}
Using Lemma \ref{Poisson equation}, we get
\begin{align}    \label{DxUo 01}
\left|\frac{\partial u_{1}}{\partial t}(t,x,y) \right|
\leq C(1+ \|x\| + \|y\|)
\left|y \longmapsto \frac{\partial\chi}{\partial t}(t,x,y) \right|_{Lip}.
\end{align}
It follows that for any $y_1, y_2 \in L^2$,
%Arguing as before???, we find that
\begin{eqnarray} \label{DxUo 02}
\left|\frac{\partial\chi}{\partial t}(t,x,y_{1})-\frac{\partial\chi}{\partial t}(t,x,y_{2})\right|\leq\!\!\!\!\!\!\!\!&&\big\|f(x,y_1)-f(x,y_2)\big\|\big\|\frac{\partial}{\partial t}D_{x}u_{0}(t,x)\big\| \nonumber\\
\leq\!\!\!\!\!\!\!\!&&
C\|y_1-y_2\| \big\|\frac{\partial}{\partial t}D_{x}u_{0}(t,x)\big\|.
\end{eqnarray}
According to  Lemmas \ref{ESDX}, \ref{eta} and \ref{ESDET},
for any $h\in L^2$,  there exist $k\in\mathbb{N}$ and a constant $C=C_{\theta, T}$ such that
\begin{eqnarray}
\left|\left\langle\frac{\partial}{\partial t}D_{x}u_{0}(t,x),h\right\rangle\right|=\!\!\!\!\!\!\!\!&&\left|\frac{\partial}{\partial t}\left[D\phi(\bar{X}(t,x))\cdot \eta^{h}(t,x)\right]\right|  \nonumber\\
\leq\!\!\!\!\!\!\!\!&&\left|D^{2}\phi(\bar{X}(t,x))\left(\eta^{h}(t,x),\frac{d}{dt}\bar{X}(t,x)\right)\right|+\left|D\phi(\bar{X}(t,x))\cdot\frac{d}{dt}\eta^{h}(t,x)\right| \nonumber\\
\leq\!\!\!\!\!\!\!\!&&Ct^{-1+\frac{\theta}{2}}(1+|x|^{2}_{\theta})e^{C \|x\|^{k}} \|h\|
+Ct^{-1}(1+|x|^{2}_{\theta}) e^{C\|x\|^{k}}\|h\|,  \nonumber
\end{eqnarray}
%where $C$ is a constant depending on $\theta, T$.
Hence, we obtain
\begin{align} \label{DxUo 03}
\left\|\frac{\partial}{\partial t}D_{x}u_{0}(t,x)\right\|
\leq Ct^{-1}(1+|x|_{\theta}^2)e^{C \|x\|^{k}}.
\end{align}
The result follows by combining \eqref{DxUo 01}-\eqref{DxUo 03}.
%The proof is complete. \hspace{\fill}$\square$
\end{proof}

\begin{lemma} \label{Theorem 5.3}
Under the conditions \ref{A1}, \ref{A2} and \ref{A4},
there exists a positive constant $C$ such that for any $0\leq t\leq T$,
$x\in H^2$, $y\in L^2$,
\begin{align}
\left|L_1 u_{1}(t,x,y)\right|\leq Ce^{C\|x\|^{5}}(1+\|y\|)(1+|x|^2_1+\|y\|+\|Ax\|). \nonumber
\end{align}
\end{lemma}

\begin{proof}
By the definition of $L_{1}$, we have
\begin{align} \label{ControlUa}
L_{1}u_{1}(t,x,y)=\langle Ax+B(x)+f(x,y),D_{x}u_{1}(t,x,y)\rangle.
\end{align}
It is easy to see that
\begin{align} \label{Controlu1 001}
|Ax+B(x)+f(x,y)|\leq C(1+|x|^2_1+\|y\|+\|Ax\|).
\end{align}
Thus it remains to estimate $D_{x}u_{1}(t,x,y)$.
Recall that $u_{1}$ satisfies
\begin{equation}
L_{2}u_{1}(t,x,y)=-\chi(t,x,y).  \nonumber
\end{equation}
For fixed $r>0$, $h\in L^2$, define
$
\tilde{u}(t,x,y):=\frac{u_1(t,x+rh,y)-u_1(t,x,y)}{r}.
$
It follows that
\begin{eqnarray}
L_{2}\tilde{u}(t,x,y)=\!\!\!\!\!\!\!\!&&
-\frac{\chi(t,x+r h,y)-\chi(t,x,y)}{r} \nonumber\\
\!\!\!\!\!\!\!\!&&
-\big\langle \frac{g(x+r h,y)-g(x,y)}{r},D_{y}u_{1}(t,x+r h,y) \big\rangle \nonumber\\
:=\!\!\!\!\!\!\!\!&&-\Gamma(t,x,y,h,r).   \nonumber
\end{eqnarray}
According to Lemma \ref{Poisson equation}, we obtain
\begin{eqnarray}
\!\!\!\!\!\!\!\!&&\frac{u_{1}(t,x+r h,y)-u_{1}(t,x,y)}{r}-\int_{H}\frac{u_{1}(t,x+r h,y)-u_{1}(t,x,y)}{r}\mu^{x}(dy) \nonumber\\
=\!\!\!\!\!\!\!\!&&\int_{0}^{+\infty}\mathbb{E}[\Gamma(t,x,Y^{x,y}_{s},h,r)]ds.   \nonumber
\end{eqnarray}
In the same way as in the argument in \cite[Section 5.3]{B1},
as $r\rightarrow0$, we deduce that
$$
\left|\lim_{r\rightarrow 0}
\int_{L^2}\frac{u_{1}(t,x+r h,y)-u_{1}(t,x,y)}{r}\mu^{x}(dy)\right|\leq C(1+\|x\|) \|h\|
$$
and
\begin{eqnarray*}
\lim_{r\rightarrow 0}\Gamma(t,x,y,h,r)=\Theta(t,x,y)\cdot h:=D_{x}\chi(t,x,y)\cdot h+\langle D_{x}g(x,y)\cdot h,D_{y}u_{1}(t,x,y)\rangle.
\end{eqnarray*}
On one hand, by the definition of $\chi$, we have
\begin{align}
D_{x}\chi(t,x,y)\cdot h=\langle (D_{x}f(x,y)-D_{x}\bar{f}(x))\cdot h,D_{x}u_{0}(t,x)\rangle+D_{xx}^{2}u_{0}(t,x)\cdot(h,f(x,y)). \nonumber
\end{align}
Using condition \ref{A4}, Lemmas \ref{eta} and \ref{zeta}, it follows that
%By assumption $(\textbf{H4})$, , we have
\begin{align}
|\langle (D_{x}f(x,y)-D_{x}\bar{f}(x))\cdot h,D_{x}u_{0}(t,x)\rangle|
\leq Ce^{C\|x\|^{5}} \|h\| \nonumber
\end{align}
and
\begin{align*}
|D_{xx}^{2}u_{0}(t,x)\cdot(h,k)|
%=\!\!\!\!\!\!\!\!&&
%|D^{2}\phi(\bar{X}(t,x))(\eta^h(t,x),\eta^k(t,x))
%+D\phi(\bar{X}(t,x))\cdot\zeta^{h,k}(t,x)| \nonumber\\
\leq Ce^{C\|x\|^{5}}\|h\| \|k\|.
\end{align*}
Then we obtain
\begin{align} \label{Dxchi}
|D_{x}\chi(t,x,y)\cdot h|\leq  Ce^{C\|x\|^{5}}\|h\| + Ce^{C\|x\|^{5}}\|h\|\cdot|f(x,y)|
\leq Ce^{C\|x\|^{5}}(1+\|y\|)\|h\|.
\end{align}
On the other hand,
using \eqref{WeakDissIne} and following
the argument in \cite[Lemma 4.3]{B1}, we have
%by condition (3) in assumption $(\textbf{H4})$ and following the argument in \cite[Lemma 4.3]{B1}, we have
\begin{align} \label{Dxg}
|\langle D_{x}g(x,y)\cdot h,D_{y}u_{1}(t,x,y)\rangle|
\leq
C\|h\| \|D_{y}u_{1}(t,x,y)\|
\leq C(1+\|y\|^{2}) \|h\|.
\end{align}
Therefore, putting \eref{Dxchi} and \eref{Dxg} together, we get
\begin{eqnarray}
\left|\Theta(t,x,y)\cdot h\right|\leq Ce^{C\|x\|^{5}}(1+\|y\|^2) \|h\|.\label{Theta}
\end{eqnarray}
Notice that for any $t,x,r,h$, by the definition of $\Gamma$,
we have $\int_{L^2}\Gamma(t,x,y,h,r)\mu^x(dy)=0$,
which implies $\int_{L^2}\Theta(t,x,y)\cdot h\mu^x(dy)=0$
by using the dominated convergence theorem. Consequently, we obtain
\begin{eqnarray} \label{Theta001}
|D_x u_1(t,x,y)\cdot h|
=\!\!\!\!\!\!\!\!&&
\lim_{r\rightarrow 0}\int_{L^2}
\frac{u_{1}(t,x+r h,y)-u_{1}(t,x,y)}{r}\mu^{x}(dy) \nonumber\\
&&+\int^{\infty}_0\EE[\Theta(t,x,Y^{x,y}_s,h)]ds.
\end{eqnarray}
One can verify that Lemma \ref{Poisson equation} can be extended
to the case where $\Psi$ has the form of quadratic growth as in \eref{Theta}.
Applying Lemma \ref{Poisson equation} to $\Theta$, from \eqref{Theta001}, we get
%According to Lemma \ref{Poisson equation},
%we don't know whether $\Theta$ is a Lipschitz or bounded function with respect to $y$,
%and we only know it has quadratic growth by \eref{Theta}.
%However, the result of Lemma \ref{Poisson equation}
%can be easily extended to such function, hence we have
\begin{align}
|D_{x}u_{1}(t,x,y)\cdot h|\leq Ce^{C\|x\|^{5}}(1+\|y\|)\|h\|.\nonumber
\end{align}
and therefore
\begin{align} \label{L1u1}
|L_{1}u_{1}(t,x,y)|\leq Ce^{C\|x\|^{5}}(1+\|y\|)(1+|x|^2_1+\|y\|+\|Ax\|).
\end{align}
The result follows by combining \eqref{ControlUa}-\eqref{Controlu1 001} and \eqref{L1u1}.
%\hspace{\fill}$\square$
\end{proof}

%Now we are going to estimate $\mathbb{E}\left[v^{\vare}(\rho(\vare),X^{\vare}(T-\rho(\vare),x,y),Y^{\vare}(T-\rho(\vare),x,y))\right]$.
\begin{lemma} \label{vxy}
Under the conditions \ref{A1}, \ref{A2} and \ref{A4},
for any $x\in H^{\theta}$ with $\theta\in(0,1)$, $y\in L^2$,
$\alpha\in(0,\frac{1}{4})$, $0\leq t\leq T$,
there exist $k\in\NN$ and a constant $C=C_{\alpha, \theta, T}$ such that
\begin{eqnarray}
|v^{\vare}(\rho(\vare),x,y)|
\leq C\frac{\rho(\vare)^{\frac{\theta}{2}}}{\theta}(1+|x|^2_{\theta}+\|y\|^2) e^{C\|x\|^k}
+C\vare e^{C\|x\|^{5}}(1+\|y\|)+C\rho(\vare)\vare^{-\alpha}.  \nonumber
\end{eqnarray}
%where $C$ is a constant depending on $\alpha, \theta, T$.
\end{lemma}
\begin{proof}
Using \eqref{Expand u} and
%By the asymptotic expansion $u^{\vare}=u_{0}+\vare u_{1}+v^{\vare}$,
noting that $u_{0}$ is independent of $y$,  we  write
\begin{eqnarray} \label{vvare xy}
v^{\vare}(\rho(\vare),x,y)
%=\!\!\!\!\!\!\!\!&&
%u^{\vare}(\rho(\vare),x,y)-u_{0}(\rho(\vare),x)-\vare u_{1}(\rho(\vare),x,y) \nonumber\\
=\!\!\!\!\!\!\!\!&&
[u^{\vare}(\rho(\vare),x,y)-u^{\vare}(0,x,y)]
-[u_{0}(\rho(\vare),x)-u_{0}(0,x)]-\vare u_{1}(\rho(\vare),x,y) \nonumber\\
:=\!\!\!\!\!\!\!\!&&I_{1}-I_{2}-I_{3}.
\end{eqnarray}
For $I_{1}$, by Lemma \ref{LemmaB.3}, we have
\begin{align} \label{vvare xy01}
\|I_{1}\|
=& \
\left|\int_{0}^{\rho(\vare)}\frac{d}{dt}u^{\vare}(t,x,y)dt\right|
=  \left\|\int_{0}^{\rho(\vare)}\frac{d}{dt}\mathbb{E}[\phi(X^{\vare}(t,x,y))]dt\right\| \nonumber\\
=& \
\left\|\int_{0}^{\rho(\vare)}
\mathbb{E}\left[D\phi(X^{\vare}(t,x,y))\cdot\frac{d}{dt}X^{\vare}(t,x,y)\right]dt\right\| \nonumber\\
\leq & \
C\int_{0}^{\rho(\vare)}\mathbb{E}\left\|\frac{d}{dt}X^{\vare}(t,x,y)\right\|dt
\leq
C\frac{\rho(\vare)^{\frac{\theta}{2}}}{\theta}\left(|x|^2_{\theta}+\|y\|^2+1\right)e^{C\|x\|^k}
+ C\rho(\vare)\vare^{-\alpha}.
\end{align}
For $I_{2}$, recalling that $u_{0}=\bar{u}$, by Lemma \ref{ESDX}, we get
\begin{eqnarray} \label{vvare xy02}
\|I_{2}\|=\!\!\!\!\!\!\!\!&&\left\|\int_{0}^{\rho(\vare)}\frac{d}{dt}u_{0}(t,x)dt\right\|
=  \left\|\int_{0}^{\rho(\vare)}\frac{d}{dt}\phi(\bar{X}(t,x))dt\right\| \nonumber\\
=\!\!\!\!\!\!\!\!&&
\left\|\int_{0}^{\rho(\vare)}D\phi(\bar{X}(t,x))\cdot\frac{d}{dt}\bar{X}(t,x)dt\right\|
\leq  C\int_{0}^{\rho(\vare)}\left\|\frac{d}{dt}\bar{X}(t,x)\right\|dt  \nonumber\\
\leq \!\!\!\!\!\!\!\!&&
\frac{C\rho(\vare)^{\frac{\theta}{2}}}{\theta}(1+|x|^2_{\theta})e^{C\|x\|^k}.
\end{eqnarray}
For $I_{3}$, it follows from Lemma \ref{ESu1} that
\begin{align} \label{vvare xy03}
\|I_{3}\|\leq C\vare e^{C\|x\|^{5}}(1+\|y\|).
\end{align}
Putting together \eqref{vvare xy}-\eqref{vvare xy03}, the result follows.
%The result follows by combining \eqref{vvare xy}-\eqref{vvare xy03}.
%\hspace{\fill}$\square$
\end{proof}

\subsection{Proof of Theorem \ref{main result 2}}\label{Subsection 5.5}
%\begin{proof}[\textbf{Proof of Theorem \ref{main result 2}}]
From \eqref{Ephix1}-\eqref{Ephix2} and  \eref{Ephix3}, we get
\begin{align*}
& \  \mathbb{E}\phi\left(X_{N}^{\vare}(t)\right) -
\mathbb{E} \phi\left(\bar{X}_{N}(t)\right)
=  \vare u_{1}(t,x,y)+\mathbb{E}\left[v^{\vare}(\rho(\vare),X^{\vare}(t-\rho(\vare),x,y),Y^{\vare}(t-\rho(\vare),x,y))\right] \nonumber\\
 & \qquad \qquad  + \vare\mathbb{E}
 \left[\int_{\rho(\vare)}^{t}(L_{1}u_{1}
 -\frac{\partial u_{1}}{\partial s})(s,X^{\vare}(t-s,x,y),Y^{\vare}(t-s,x,y))ds\right].
\end{align*}
%\vskip 0.2cm
%${\mathbf{Proof of Theorem \ref{main result 2}}:}$
%From the expression \eref{Ephix3},
Using Lemmas \ref{ESu1}-\ref{vxy} and \ref{Xvare2},  \eqref{Remark5.3},
 \eqref{highorder of X}, the H\"{o}lder inequality, and the fact that $\rho(\vare)=\vare^{\frac{1}{a}}, 0<a\leq\frac{\theta}{2}$, we obtain
\begin{eqnarray} \label{PfofThm2}
\!\!\!\!\!\!\!\!&&
|\mathbb{E}\phi\left(X_{N}^{\vare}(t)\right) -
\mathbb{E} \phi(\bar{X}_{N}(t))|   \nonumber\\
\leq\!\!\!\!\!\!\!\!&& C\vare e^{C\|x\|^{5}}(1+\|y\|)
+  C\frac{\rho(\vare)^{\frac{\theta}{2}}}{\theta}
\EE\left[(1+|X^{\vare}(t-\rho(\vare))|^2_{\theta}+\|Y^{\vare}(t-\rho(\vare))\|^2)e^{C\|X^{\vare}(t-\rho(\vare))\|^k}\right]    \nonumber\\
\!\!\!\!\!\!\!\!&&
+ C\vare \EE \left[e^{C\|X^{\vare}(t-\rho(\vare))\|^{5}}(1+\|Y^{\vare}(t-\rho(\vare))\|)\right]+C\rho(\vare)\vare^{-\alpha} \nonumber\\
&&+C\vare\int_{\rho(\vare)}^{t}\!\!\!\EE\left[e^{C\|X^{\vare}(t-s)\|^5}(1+\|Y^{\vare}(t-s)\|)(1+|X^{\vare}(t-s)|^2_1+\|Y^{\vare}(t-s)\|+\|AX^{\vare}(t-s)\|)\right]ds \nonumber\\
\!\!\!\!\!\!\!\!&&
+C\vare\int_{\rho(\vare)}^{t}s^{-1}
\EE\left[(|X^{\vare}(t-s)|^{2}_{\theta}+1)e^{C\|X^{\vare}(t-s)\|^{k}}(1+\|Y^{\vare}(t-s)\|)\right]ds 
\nonumber\\
%\leq\!\!\!\!\!\!\!\!&& C\vare+C\frac{\rho(\vare)^{\frac{\theta}{2}}}{\theta}\EE|X^{\vare}(t-\rho(\vare))|_{\theta}^{2}
%+C\frac{\rho(\vare)^{\frac{\theta}{2}}}{\theta}+C\vare+C\rho(\vare)\vare^{-\alpha} \nonumber\\
%\!\!\!\!\!\!\!\!&&
%+C\vare\int_{\rho(\vare)}^{t}\left[1+\left(\EE|X^{\vare}(t-s)|_{1}^{4}\right)^{\frac{1}{2}}+(\EE\|AX^{\vare}(t-s)\|^{2})^{\frac{1}{2}}\right]ds
%\nonumber\\
%\!\!\!\!\!\!\!\!&&
%+C\vare\int_{\rho(\vare)}^{t}s^{-1}\left[\left(\EE|X^{\vare}(t-s)|_{\theta}^{4}\right)^{1/2}+1\right]ds \nonumber\\
\leq\!\!\!\!\!\!\!\!&& C\vare+Ct^{-\theta+\frac{\theta^2}{1+\delta}}\rho(\vare)^{\frac{\theta}{2}}-C\vare\log\vare
\leq C\vare t^{-\theta+\frac{\theta^2}{1+\delta}}-C\vare\log\vare,
\end{eqnarray}
where the constant $C$ is independent of the dimension $N$.
Then, letting $N\to +\infty$ in \eqref{PfofThm2}, in view of \eqref{ephix}, the desired result follows.
\hspace{\fill}$\square$
%\end{proof}
%we get
%\begin{eqnarray}
%\left|\mathbb{E}[\phi(X^{\vare}(t))]-\mathbb{E}[\phi(\bar{X}(t))]\right|\leq C(1+t^{-\theta+\frac{\theta^2}{1+\delta}})\vare^{1-r} .
%\end{eqnarray}
%The proof is complete.
%
%Note that
%\begin{align*}
%\left|\mathbb{E}[\phi(X^{\vare}(t))]-\mathbb{E}[\phi(\bar{X}(t))]\right|
%\leq
%\left|\mathbb{E}\left[\phi\left(X^{\vare}(t)\right)\right]-\mathbb{E}\left[\phi\left(X_{N}^{\vare}(t)\right)\right]\right|
%+\left|\mathbb{E}\left[\phi\left(\bar{X}_{N}(t)\right)\right]-\mathbb{E}\left[\phi\left(\bar{X}(t)\right)\right]\right|
%+
%\end{align*}

%\vskip 0.4cm
%${\mathbf{Proof of Theorem \ref{main result 3}}:}$
\subsection{Proof of Theorem \ref{main result 3}} \label{ProofThm3}
%\begin{proof}[\textbf{Proof of Theorem \ref{main result 3}}]
Since the proof of Theorem \ref{main result 3}
is similar to that of Theorem \ref{main result 2}, we only sketch the difference here.
Benefited from a higher regularity of initial value,
following the proof of \eqref{A.1.2} in Lemma \ref{BarXgamma},
one can verify that if $\theta\in (1, \frac{3}{2})$, then
%For simplicity, we only give a brief description here and state some key results benefited from a higher regularity of initial value.
%Notice that if $\theta\in (1, \frac{3}{2})$, for any $t\in(0, T]$, we can easily obtain
\begin{align} \label{HighRegu 01}
\sup_{t \in [0,T]}|\bar{X}_{t}|_{\theta} \leq C|x|_{\theta}e^{C\|x\|}.
\end{align}
Similarly to the proof of \eqref{B.1.1}  in Lemma \ref{highorder of X},
we show that if $\theta\in (1, \frac{3}{2})$, then
for any $p\geq 1$, we have
\begin{align} \label{HighRegu 02}
\sup_{t \in [0,T]} \EE|X_{t}^{\vare}|_{\theta}^{p}
\leq C\left(1+|x|_{\theta}^{p}+\|y\|^{p}\right).
\end{align}
%where the constant $C$ is independent of $t$.
Consequently, based on \eqref{HighRegu 01} and \eqref{HighRegu 02},
one can improve the corresponding results in Lemmas \ref{ESu1}-\ref{vxy}. 
Then we can obtain Theorem \ref{main result 3} in an analogous way as in the proof of Theorem
\ref{main result 2}.
\hspace{\fill}$\square$
%Finally, after doing some corresponding modifications
%about the results of Lemmas \ref{ESu1}-\ref{vxy},
%the desired result follows. %\hspace{\fill}$\square$
%\end{proof}

\section{Appendix} \label{Sec appendix}

The following properties of $b(\cdot,\cdot,\cdot)$ and $B(\cdot,\cdot)$ are well-known (for example see \cite{DX}). % and will be used later on.
\begin{lemma} \label{Property B0}
For any $x, y \in H^1_0$, it holds that
$$ b(x,x,y)=-\frac12b(x,y,x),\quad b(x,y,y)=0.$$
%\hspace{\fill}$\square$
\end{lemma}

\begin{lemma} \label{Property B1}
Suppose that $\alpha_{i}\geq 0~(i=1,2,3)$
satisfies one of the following conditions \\
$(1) ~\alpha_{i}\neq\frac{1}{2}(i=1,2,3), \alpha_{1}+\alpha_{2}+\alpha_{3}\geq \frac{1}{2}$, \\
$(2) ~\alpha_{i}=\frac{1}{2}$ for some $i$, $\alpha_{1}+\alpha_{2}+\alpha_{3}>\frac{1}{2}$,\\
then $b$ is continuous from 
$H^{\alpha_{1}}\times H^{\alpha_{2}+1}\times H^{\alpha_{3}}$ to $\mathbb{R}$, i.e.
$$\big|b(x,y,z)\big|\leq C|x|_{\alpha_{1}}|y|_{\alpha_{2}+1}|z|_{\alpha_{3}}.$$
%\hspace{\fill}$\square$
\end{lemma}
The following inequalities can be derived by the above lemma.
\begin{corollary} \label{Property B3} For any $x\in H_{0}^{1}$, we have \\
$(1) \|B(x)\|\leq C|x|_{1}^{2}$.\\
$(2) |B(x)|_{-1}\leq C\|x\|\cdot|x|_{1}.$
%\hspace{\fill}$\square$
\end{corollary}

\begin{lemma} \label{Property B2}
For any $x,y\in H_{0}^{1}$, we have \\
$(1) \|B(x)-B(y)\|\leq C|x-y|_{1}(|x|_{1}+|y|_{1})$.\\
$(2) |B(x)-B(y)|_{-1}\leq C\|x-y\| \left(|x|_{1}+|y|_{1}\right).$
%\hspace{\fill}$\square$
\end{lemma}

%At the end of this section, we give the classical Gronwall inequality and a Gronwall-type inequality.
%% which will be used later.
%
%\begin{lemma}[Gronwall inequality] \label{Gronwall 1}
%Let $\alpha, \beta$ be real-value functions defined on $[0, T]$, assume that $\beta$ and $u$ are continuous and that the negative part of $\alpha$ is integrable on every closed and bounded subinterval of $[0, T]$. \\
%(a) If $\beta$ is non-negative and if $u$ satisfies the integral inequality
%\begin{eqnarray*}
%u_t\leq \alpha_t+\int^t_0\beta_s u_sds, \quad \forall t\in [0, T],
%\end{eqnarray*}
%then
%$$
%u_t\leq \alpha_t+\int^t_0 \alpha_s \beta_s\exp{\left(\int^t_s \beta_rdr\right)}ds,\quad \forall t\in [0, T].
%$$
%(b) If, in addition, the function $\alpha$ is nondecreasing, then
%$$
%u_t\leq \alpha_t\exp{\left(\int^t_0 \beta_rdr\right)},\quad \forall t\in [0, T].
%$$
%\end{lemma}

Similar as the argument in the proof of \cite[Theorem 2.6]{LZ06}, we can easily obtain the following Gronwall-type inequality, whose proof is based on iteration and we omit the details.

\begin{lemma}[Gronwall-type inequality] \label{Gronwall 2}
%$\beta\in(0,1)$,
Let $f(t)$ be a non-negative real-valued integrable function on $[0,T]$.
For any given $\alpha, \beta \in(0,1)$,
if there exist two positive constants $C_1, C_2$ such that
%and the following inequality holds for some positive constants $C_1, C_2$
\begin{align}
f(t)\leq C_1t^{-\alpha}+C_2\int_{0}^{t}(t-s)^{-\beta}f(s)ds, \quad \forall t\in[0,T], \nonumber
\end{align}
then there exists some $k\in \mathbb{N}$ and a positive constant $C:=C_{\alpha, \beta, T}$
such that
%satisfying
\begin{align}
f(t)\leq C C_1t^{-\alpha}e^{C C^{k}_2}, \quad \forall t\in[0,T].  \nonumber
\end{align}
%where $C$ is a positive constant depending on $\alpha, \beta$ and $T$.
\end{lemma}

%\begin{proof}
%By iterating and Fubini theorem, we have
%\begin{eqnarray}
%f(t)\leq\!\!\!\!\!\!\!\!&& C_1 t^{-\alpha}+C_2 \int_{0}^{t}(t-s)^{-\beta}\left[C_1 s^{-\alpha}+C_2 \int_{0}^{s}(s-r)^{-\beta}f(r)dr\right]ds  \nonumber\\
%\leq\!\!\!\!\!\!\!\!&& C_1t^{-\alpha}+C_1 C_2 \int_{0}^{t}(t-s)^{-\beta}s^{-\alpha}ds+C^2_2\int_{0}^{t}f(r)\left[\int_{r}^{t}(t-s)^{-\beta}(s-r)^{-\beta}ds\right]dr  \nonumber\\
%\leq\!\!\!\!\!\!\!\!&& C C_1(C_2+1)t^{-\alpha}+C C^2_2\int_{0}^{t}(t-r)^{1-2\beta}f(r)dr,  \nonumber
%\end{eqnarray}
%where $C$ is a constant depending on $\alpha, \beta, T$.
%
%If $1-2\beta\geq0$, then we can easily obtain the result by Lemma \ref{Gronwall 1}. However, if $1-2\beta<0$, then after iterating finite times, there exist  $\gamma\geq 0, k_1, k_2\in\mathbb{N}$ such that
%\begin{eqnarray}
%f(t)\leq\!\!\!\!\!\!\!\!&& C C_1 (C^{k_1}_2+1)t^{-\alpha}+C C^{k_2}_2\int_{0}^{t}(t-r)^{\gamma}f(r)dr \nonumber\\
%\leq\!\!\!\!\!\!\!\!&& C C_1 (C^{k_1}_2+1)t^{-\alpha}+C C^{k_2}_2\int_{0}^{t}f(r)dr. \nonumber
%\end{eqnarray}
%
%Finally  by Lemma \ref{Gronwall 1}, we obtain
%\begin{eqnarray*}
%f(t)\leq\!\!\!\!\!\!\!\!&& C C_1 (C^{k_1}_2+1)t^{-\alpha}e^{C C^{k_2}_2}\\
%\leq\!\!\!\!\!\!\!\!&& C C_1t^{-\alpha}e^{C C^{k}_2}, \nonumber
%\end{eqnarray*}
%for some $k\in\mathbb{N}$.
%  \hspace{\fill}$\square$
%\end{proof}

\vskip 0.5cm
\textbf{Acknowledge}. 
%We thank the referees for their valuable comments which helped to considerably improve the quality of the paper.
 %and, in particular, for suggesting the references \cite{LZ06,LZ09}, which have helped to fill the mathematical gap in the first version of the proof of Lemma \ref{PMY}.
 This paper is partially supported by Key Laboratory of Random Complex Structures and Data Science, 
 PCSDS, 
 Academy of Mathematics and Systems Science, Chinese Academy of Sciences (No: 2008DP173182), by the Priority Academic Program Development of Jiangsu Higher Education Institutions, by Natural Science Foundation of the Higher Education Institutions of Jiangsu Province (No:16KJB110006), by Key Research Program of Frontier Sciences, CAS (No: QYZDB-SSW-SYS009) and the Fundamental Research Funds for the Central Universities (No. WK 3470000008), by NSFC ( No: 11271356, No: 11371041, No: 11431014, No: 11271169, No: 11601196, No: 11771187, No. 11671372, No. 11721101).

\end{document}